\documentclass[10pt, a4paper, oneside, leqno]{amsart}

\usepackage[T1]{fontenc}
\usepackage[abbrev,backrefs]{amsrefs}
\usepackage{lmodern}
\usepackage{microtype}

\usepackage[backref]{hyperref}

\usepackage{amsmath}
\usepackage{amstext}
\usepackage{enumerate}
\usepackage{a4wide}
\usepackage{esint}
\usepackage{mathrsfs}%adds \mathscr{} fonts
\usepackage{tocvsec2}

\newtheorem{theorem}{Theorem}

\newtheorem{proposition}{Proposition}[section]
\newtheorem{claim}{Claim}
\newtheorem{lemma}[proposition]{Lemma}

\theoremstyle{definition}
\newtheorem{remark}{Remark}[section]

\numberwithin{equation}{section}
\newtheorem{definition}{Definition}
\newcommand{\abs}[1]{\lvert #1 \rvert}
\newcommand{\bigabs}[1]{\bigl\lvert #1 \bigr\rvert}
\newcommand{\Bigabs}[1]{\Bigl\lvert #1 \Bigr\rvert}
\newcommand{\norm}[1]{\lVert #1 \rVert}

\newtheorem{innercustomthm}{Theorem}
\newenvironment{customthm}[1]
  {\renewcommand\theinnercustomthm{#1}\innercustomthm}
  {\endinnercustomthm}

\newenvironment{proofclaim}[1][Proof of the claim]{\begin{proof}[#1]}{\end{proof}}

\newcommand{\N}{{\mathbb N}}
\newcommand{\R}{{\mathbb R}}
\newcommand{\Sset}{{\mathbb S}}
\newcommand{\loc}{{\mathrm{loc}}}
\newcommand{\rad}{{\mathrm{rad}}}
\newcommand{\mmu}{{m}}

\DeclareMathOperator{\Div}{div}
\DeclareMathOperator{\supp}{supp}
\DeclareMathOperator{\esssup}{ess\, sup}

\newcommand{\dif}{\,\mathrm{d}}
\newcommand{\dualprod}[2]{\langle #1, #2 \rangle}

\allowdisplaybreaks

\title[Schr\"odinger--Poisson--Slater equations at the critical frequency]{Groundstates and radial solutions\\ to nonlinear
Schr\"odinger--Poisson--Slater equations\\ at the critical frequency}

\author{Carlo Mercuri}
\address{Swansea University\\ Department of Mathematics\\ Singleton Park\\
Swansea\\ SA2~8PP\\ Wales, United Kingdom}	
\email{C.Mercuri@swansea.ac.uk}

\author{Vitaly Moroz}
\address{Swansea University\\ Department of Mathematics\\ Singleton Park\\
Swansea\\ SA2~8PP\\ Wales, United Kingdom}	
\email{V.Moroz@swansea.ac.uk}

\author{Jean Van Schaftingen}
\address{Universit\'e Catholique de Louvain\\ Institut de Recherche en Math\'ematique et Phy\-sique\\ Chemin du Cyclotron 2 bte L7.01.01\\ 1348 Louvain-la-Neuve \\ Belgium}
\email{Jean.VanSchaftingen@uclouvain.be}


\begin{document}

\begin{abstract}
We study the nonlocal Schr\"odinger--Poisson--Slater type equation
$$
 - \Delta u + (I_\alpha \ast \vert u\vert^p)\vert u\vert^{p - 2} u= \vert u\vert^{q-2}u\quad\text{in \(\mathbb{R}^N\),}
$$
where $N\in\mathbb{N}$, $p>1$, $q>1$ and $I_\alpha$ is the Riesz potential of order $\alpha\in(0,N).$
We introduce and study the Coulomb--Sobolev function space which is natural for the energy functional of the problem
and we establish a family of
associated optimal interpolation inequalities. We prove existence of optimizers for the inequalities,
which implies the existence of solutions to the equation for a certain range of the parameters.
We also study regularity and some qualitative properties of solutions.
Finally, we derive radial Strauss type estimates and use them to prove
the existence of radial solutions to the equation in a range of parameters which is in
general wider than the range of existence parameters obtained via interpolation inequalities.
\end{abstract}

\maketitle

%\vspace{-1em}
\tableofcontents
\newpage

\section{Introduction}
\settocdepth{section}

\subsection{Setting of the problem.}

We study the Schr\"odinger--Poisson--Slater type equation
\begin{equation}\tag{$\mathcal{SPS}$}\label{sps}
 - \Delta u + (I_\alpha \ast \abs{u}^p)\abs{u}^{p - 2} u= \abs{u}^{q-2}u\quad\text{in \(\R^N\),}
\end{equation}
where $N\in\N$, $p>1$, $q>1$ and $I_\alpha : \R^N \to \R$ is the Riesz potential of order $\alpha\in(0,N)$,
defined for \(x \in \R^N \setminus \{0\}\) as
$$
  I_\alpha(x)=\frac{A_\alpha}{\abs{x}^{N-\alpha}},\qquad A_\alpha=\frac{\Gamma(\tfrac{N-\alpha}{2})}{\Gamma(\tfrac{\alpha}{2})\pi^{N/2}2^{\alpha}}.
$$
The choice of normalisation constant $A_\alpha$ ensures that the kernel $I_\alpha$ enjoys the semigroup property
$$I_{\alpha+\beta}=I_\alpha*I_\beta\qquad \text{for each \(\alpha,\beta\in(0,N)\) such that
\(\alpha+\beta<N\)},
$$
see for example \cite{DuPlessis}*{pp.\thinspace{}73--74}. Since $I_\alpha*\varphi\to\varphi$ as $\alpha\to 0$ for all $\varphi\in C^\infty_c(\R^N)$,
the local equation
$$ - \Delta u + \abs{u}^{2p - 2} u= \abs{u}^{q-2}u\quad\text{in \(\R^N\)}$$
can be seen as a formal limit of \eqref{sps} when $\alpha\to 0$.
Local equations of this type were studied, e.g. in \citelist{\cite{MerlePeletier1992}\cite{DelPinoDolbeault2002}}.
Because for $\alpha\in(0,N)$ the Riesz potential $I_\alpha$ is the Green function of the fractional Laplacian $(-\Delta)^{\alpha/2}$  (see for example \cite{Stein1970}*{Section 5.1.1}),
the system of equations
\begin{equation*}
\left\{
\begin{array}{rcl}
 - \Delta u + v \abs{u}^{p-2}u &=& \abs{u}^{q-2}u,\\
 (-\Delta)^{\alpha/2}v &=&u^p,\\
\end{array}
\right.
\end{equation*}
is formally equivalent to  equation \eqref{sps}.

The nonlocal nonlinear Schr\"odinger equation, in natural units,
\begin{equation}\label{NLS-3}
i\partial_t\psi=-\Delta\psi + V_{\mathrm{ext}}(x)\psi+\Big(\frac{1}{4\pi\abs{x}}\ast \abs{\psi}^2\Big)\psi- \abs{\psi}^{q-2}\psi,
\quad(x,t)\in\R^3\times\R,
\end{equation}
and its stationary counterpart
\begin{equation}\label{sps-3}
 - \Delta u + V_{\mathrm{ext}}(x)u+\Big(\frac{1}{4\pi\abs{x}} \ast \abs{u}^2\Big)u= \abs{u}^{q-2}u,\quad x\in\R^3,
\end{equation}
appear in the physical literature as an approximation of the Hartree--Fock model of a quantum many--body system of electrons under the
presence of the external potential $V_{\mathrm{ext}}$, see \cite{LeBris-Lions-2005}
for a mathematical introduction into Hartree--Fock  method and further references
therein. Within this context equations \eqref{NLS-3} and \eqref{sps-3} are known
under the names of Schr\"odinger--Poisson--Slater \cite{Bokanowski-Lopez-Soler},
Schr\"odin\-ger--Poisson--$X_\alpha$
\citelist{\cite{Mauser-2001}\cite{Bao-Mauser-Stimming-2003}}, or
Maxwell--Schr\"o\-din\-ger--Poisson
\citelist{\cite{Benci-Fortunato}\cite{Catto-2013}} equations.
The function $\abs{u}^2:\R^3\to\R$ in equation \eqref{sps-3} is the density of electrons in the original many--body system.
The nonlocal convolution term represents the Coulombic repulsion between the electrons.
The local term $\abs{u}^{q-2}u$ was introduced by Slater \cite{Slater} with $q=8/3$ as a local approximation of the exchange potential in the Hartree--Fock model \citelist{\cite{Mauser-2001}\cite{Bokanowski-Lopez-Soler}}.
The multidimensional version of Schr\"odin\-ger--Poisson--Slater equation in $\R^N$ was proposed in \cite{Bao-Mauser-Stimming-2003},
where the approximated exchange term corresponds to $q=2+2/N$.
The equation \eqref{sps-3} is also related to Thomas--Fermi--Dirac--von\;Weizs\"acker (TFDW) models of density functional theory \cite{LeBris-Lions-2005}*{pp. 311--313}.

From a physical point of view, the most relevant objects in the study of equation \eqref{sps-3} are minimisers and critical points
of the energy \(\mathcal J\) corresponding to \eqref{sps-3}, defined by
$$
  \mathcal J(u)
  =\frac{1}{2}\int_{\R^3}\abs{D u}^2
   +\frac{1}{2}\int_{\R^3} V_{\mathrm{ext}}\abs{u}^2
   +\frac{1}{16\pi}\iint_{\R^3}\frac{\abs{u (x)}^2|u(y)|^2}{\abs{x - y}}\dif y\dif x
   -\frac{1}{q}\int_{\R^3}\abs{u}^q,
$$
subject to a prescribed mass constraint $\|u\|_{L^2(\R^3)}^2=m>0$.
This leads to solutions of equation \eqref{sps-3} with a free Lagrange multiplier $\lambda=\lambda(m)$,
\begin{equation}\label{sps-3m}
 - \Delta u + \big(V_{\mathrm{ext}}(x)+\lambda\big) u+\Big(\frac{1}{4\pi\abs{x}} \ast \abs{u}^2\Big)u= \abs{u}^{q-2}u,\quad x\in\R^3,
\end{equation}
while the ansatz $\psi(x,t)=e^{i\lambda t}u(x)$ produces standing--wave solutions of the time--dependent equation \eqref{NLS-3}.
Equation \eqref{sps-3m} where $V_{\mathrm{ext}} + \lambda>0$ and $\lambda$ is either a free Lagrange multiplier or a fixed parameter had been extensively studied by many authors, see for example survey papers \citelist{\cite{Ambrosetti-survey}\cite{Catto-2013}} for an extensive list of references.

In the critical frequency case, when $V_{\mathrm{ext}} + \lambda=0$, equation \eqref{sps-3m} becomes
\begin{equation}\label{sps-30}
 - \Delta u + \Big(\frac{1}{4\pi\abs{x}} \ast \abs{u}^2\Big)u= \abs{u}^{q-2}u,\quad x\in\R^3,
\end{equation}
and had been studied by Ruiz \cite{Ruiz-ARMA} and Ianni and Ruiz \cite{Ianni-Ruiz-2012}.
Their results reveal a complex mathematical structure behind \eqref{sps-30}.
In particular, \eqref{sps-30} admits a positive solution for $3<q<6$ \cite{Ianni-Ruiz-2012}*{Theorem 1.2}
and radial positive solution for $18/7<q<3$ \cite{Ruiz-ARMA}*{Theorem 1.3}.
In addition, \cite{Ruiz-ARMA}*{Theorem 1.4} shows that for $18/7<q<3$
solutions of \eqref{sps-3m} with $V_{\mathrm{ext}}\equiv 0$ and $\lambda\to 0$
converge to nontrivial solutions of \eqref{sps-30}.
In other words, under some circumstances \eqref{sps-30} plays a role of the
limit equation for \eqref{sps-3m} in the critical frequency r\'egime $\lambda\to 0$.

In order to capture and to understand various mathematical features of
Schr\"odinger--Poisson--Slater equations with zero potential, we embed it in a wider class
of equations \eqref{sps} in arbitrary space dimensions and with general parameters.
As we point out later in the present introduction, our analysis of the functional setting provides new insights on the significance of the exponents $q=3$ and $q=18/7$ and  highlights new phenomena occurring at different ranges of the parameters. In particular we identify several values of the parameters which play the role of critical thresholds for continuous, and locally and globally compact embeddings of a natural functional space associated with \eqref{sps}, as well as for the existence and the nonexistence of solutions.

\subsection{Function spaces and interpolation inequalities.}
Equation \eqref{sps} has a variational structure.
Using the semigroup property of the Riesz potential, the energy functional that corresponds to \eqref{sps} can be written as
$$
\mathcal J_*(u)=\frac{1}{2}\int_{\R^N} \abs{D u}^2+\frac{1}{2p}\int_{\R^N} \big|I_{\alpha/2}*\abs{u}^p\big|^2-\frac{1}{q}\int_{\R^N} \abs{u}^q.
$$
Our first step in the study of Schr\"o\-dinger--Poisson--Slater type equation \eqref{sps}
is to define a natural energy space associated with the energy functional $\mathcal J_*$.
In Section~\ref{Sect-CS} we define the \emph{Coulomb--Sobolev space} $E^{\alpha,p}(\R^N)$
as the space of weakly differentiable function for which the norm
\begin{equation*}
  \norm{u}_{E^{\alpha, p}}:=\biggl(\int_{\R^N} \abs{D u}^2+\Bigl(\int_{\R^N}\big|I_{\alpha/2}*\abs{u}^p\big|^2\Bigr)^{1/p}\biggr)^{1/2}
\end{equation*}
is finite.
In the case $N=3$, $\alpha=2$ and $p=2$, this space has been studied by P.-L.\thinspace{}Lions \citelist{\cite{Lions-1981}*{Lemma 4}\cite{Lions1987}*{(55)}} and D.\thinspace{}Ruiz \cite{Ruiz-ARMA}*{section 2}.
From their works, it is known that $E^{2,2}(\R^3)$ is a uniformly convex separable Banach space, that $E^{2,2}(\R^3)\subset L^q(\R^N)$ for every $q\in[3,6)$
and the embedding fails for $q<3$. The proofs rely on the quadratic algebraic structure
of the nonlocal term in the case $p=2$ and cannot be directly extended to a general $p\neq 2$.
Note that both the norm $\norm{\cdot}_{E^{\alpha, p}}$ and the energy $\mathcal J_*$,
can be considered for all $p\ge 1$ and $q\ge 1$, although for $p=1$ or $q=1$
the interpretation of equation \eqref{sps} as the Euler--Lagrange equation for $\mathcal J_*$ becomes delicate.

In Section~\ref{Sect-CS} we show that the Coulomb--Sobolev space $E^{\alpha,p}(\R^N)$
is indeed a well defined Banach space for a general set of parameters $N\in\N$, $\alpha\in(0,N)$ and $p\ge 1$.
The space is uniformly convex and reflexive for $p>1$.
We also show that $L^1_{\loc}$--convergence can efficiently replace
the usual notion of the weak convergence for $p>1$ and that the subspace of smooth test functions $C^\infty_c(\R^N)$ is dense in $E^{\alpha,p}(\R^N)$.

In Section~\ref{sect-Embeddings} we study embedding properties of the space $E^{\alpha,p}(\R^N)$.
Our main result in that section is a Gagliardo--Nirenberg type interpolation inequality
for Coulomb--Sobolev spaces and associated optimal embeddings into Lebesgue spaces.

\begin{theorem}[Coulomb--Sobolev interpolation inequality]
\label{theoremEmbedding}
Let \(N \in \N\), \(\alpha \in (0, N)\), $p,q\in[1,+\infty)$.
The space \(E^{\alpha, p} (\R^N)\) is continuously embedded in \(L^q (\R^N)\) if and only if
the following assumption holds:
\begin{equation}\tag{$\mathcal Q$}\label{Q}
\begin{split}
\text{either}\quad & \frac{1}{p} \ge \frac{(N - 2)_+}{N + \alpha} \quad\text{and}\quad
   \frac{1}{2} - \frac{1}{N} \le \frac{1}{q} \le \frac{1}{2} - \frac{p - 1}{\alpha + 2 p}\,,\bigskip\\
\text{or}\quad & \frac{1}{p} < \frac{(N - 2)_+}{N + \alpha} \quad\text{and}\quad
   \frac{1}{2} - \frac{1}{N} \ge \frac{1}{q} \ge \frac{1}{2} - \frac{p - 1}{\alpha + 2 p}\,.
\end{split}
\end{equation}
Moreover, if \eqref{Q} holds and \(\alpha + N \ne p (N - 2)\) then there exists \(S=S(N,\alpha,p,q) > 0\) such that for every \(u \in E^{\alpha, p}(\R^N)\)
the interpolation estimate
\begin{equation}\label{Coulomb-Sobolev-estimate}
  S\Bigl(\int_{\R^N} \abs{u}^q\Bigr)^{\frac{1}{q}}
  \le \Bigl(\int_{\R^N} \abs{D u}^2\Bigr)^\frac{\theta}{2}  \Bigl( \int_{\R^N} \bigl(I_{\alpha/2} \ast \abs{u}^p\bigr)^2\Bigr)^{\frac{1 - \theta}{2 p}}
\end{equation}
is valid with
\begin{equation}\label{e-theta}
 \frac{1}{q} = \theta \Bigl(\frac{1}{2} - \frac{1}{N}\Bigr) + (1 - \theta) \frac{N + \alpha}{2 N p}
\end{equation}
where
$$\frac{\alpha}{2 p + \alpha} \le \theta \le 1\quad\text{if $N\ge 2$},\qquad
\frac{\alpha}{2p+\alpha}\le\theta<\frac{1+\alpha}{1+\alpha+p}\quad\text{if $N=1$}.$$
If $N\ge 3$ and \(\alpha + N = p (N - 2)\) then there exists \(C=C(N,\alpha) > 0\) such that for every \(u \in E^{\alpha, \frac{N+\alpha}{N-2}}(\R^N)\),
\begin{equation}\label{Coulomb-Sobolev-estimate-crit}
C\int_{\R^N} \abs{u}^{\frac{2 N}{N - 2}}
\le  \Bigl(\int_{\R^N} \abs{Du}^2 \Bigr)^\frac{\alpha}{2 + \alpha}
\Bigl( \int_{\R^N} \bigabs{I_{\alpha/2} \ast \abs{u}^\frac{N+\alpha}{N-2}}^2\Bigr)^\frac{2}{2 + \alpha}.
\end{equation}
\end{theorem}

In the cases $N=1$ and $N=2$ only the first option of condition \eqref{Q} occurs, while $\frac{1}{2} - \frac{1}{N} \le \frac{1}{q}$ reduces to $q<+\infty$,
so the assumption \eqref{Q} reads as
$$p\ge 1 \quad\text{and}\quad  q\ge 2 \frac{2 p + \alpha}{2 + \alpha}.$$

The parameter \(\theta\) in the interpolation estimate \eqref{Coulomb-Sobolev-estimate} can be computed explicitly in terms of the exponent \(q\) as
\begin{align}\label{eq-theta}
 \theta & = \frac{N + \alpha - \dfrac{2 pN}{q}}{(N + \alpha) - p (N - 2)},
 & & 1 - \theta = \frac{\dfrac{2 p N}{q} - p (N - 2)}{(N + \alpha) - p (N - 2)}.
\end{align}
In the assumption \eqref{Q} the reader will recognize \(q=\frac{2 N}{N - 2}\) as the classical Sobolev critical exponent.
The exponent \(q=2 \frac{2 p + \alpha}{2 + \alpha}\) is a new \emph{Coulomb--Sobolev critical exponent}.
This case corresponds to estimate \eqref{Coulomb-Sobolev-estimate} with $\theta=\frac{\alpha}{2 p + \alpha}$, which reads as
\begin{equation}
\label{ineqCriticalCoulombSobolev}
S^{2 \frac{2 p + \alpha}{2 + \alpha}}\int_{\R^N} \abs{u}^{2 \frac{2 p + \alpha}{2 + \alpha}}
\le  \Bigl(\int_{\R^N} \abs{Du}^2 \Bigr)^\frac{\alpha}{2 + \alpha}
\Bigl( \int_{\R^N} \abs{I_{\alpha/2} \ast \abs{u}^p}^2\Bigr)^\frac{2}{2 + \alpha}.
\end{equation}
When \(N = 3\), \(\alpha = 2\) and \(p = 2\), the value \(q=3\) was already known to be Coulomb--Sobolev critical \citelist{\cite{Lions1987}*{(55)}\cite{Ruiz-ARMA}*{theorem 1.5}}. Inequality \eqref{Coulomb-Sobolev-estimate}
for $N\in\N$, $\alpha\in(0,N)$ and $p=2$ was obtained in \cite{BellazziniFrankVisciglia}*{proposition 2.1}.
We emphasise that unlike the classical Hardy--Littlewood--Sobolev inequality
\begin{equation*}
\int_{\R^N} \bigl(I_{\alpha/2} \ast \abs{u}^p\bigr)^2\le C\Big(\int_{\R^N}\abs{u}^\frac{2Np}{N+\alpha}\Big)^\frac{N+\alpha}{N},
\end{equation*}
inequality \eqref{Coulomb-Sobolev-estimate} is a \emph{lower} bound on the Coulomb term and,
via Young's inequality, on the $\|\cdot\|_{E^{\alpha,p}}$--norm. The latter ensures the continuous embedding
$$E^{\alpha,p}(\R^N)\subset L^{2\frac{2 p + \alpha}{2 + \alpha}}(\R^N).$$

In section 3 we also study \emph{weighted} lower estimates on the Coulomb term.
By homogeneity considerations, a natural candidate would be
\begin{equation}
\label{eqImpossibleRuiz-intro}
 \Bigl(\int_{\R^N} \frac{\abs{u (x)}^p}{\abs{x}^{\frac{N - \alpha}{2}}} \dif x\Bigr)^2
 \le C\int_{\R^N} \abs{I_{\alpha/2} \ast \abs{u}^p}^2.
\end{equation}
However, as already observed by Ruiz \cite{Ruiz-ARMA}*{section 3} this estimate cannot hold.
In proposition~\ref{propositionRuizCounterexample} we show that a necessary condition on the weight $W:\R^N\to\R$ for inequality
\begin{equation}
\label{eqImpossibleRuiz-intro-W}
 \Bigl(\int_{\R^N} W(x)\abs{u (x)}^p\dif x\Bigr)^2
 \le C\int_{\R^N} \abs{I_{\alpha/2} \ast \abs{u}^p}^2
\end{equation}
to hold for all $u\in E^{\alpha,p}(\R^N)$ is
\[
  \int_{\R^N \setminus B_2} \frac{W (x)}{\abs{x}^\frac{N + \alpha}{2} (1 + \abs{\log \abs{x}})^\frac{1}{2} (1 + \log (1 + \abs{\log \abs{x}}))^\delta}\dif x < \infty
\]
for every $\delta>\frac{1}{2}$.
This completes the study of Ruiz who showed when \(\alpha = 2\)  that inequality \eqref{eqImpossibleRuiz-intro-W}
does not hold when \(W (x) \ge \abs{x}^{-\frac{N - 2}{2}}\abs{\log \abs{x}}^{\beta}\) when \(\beta < \frac{1}{2} - \frac{1}{N}\) \cite{Ruiz-ARMA}*{remark 3.3}.
We also obtain a family of weighted estimates (proposition~\ref{propositionRuizAverage}),
which as a particular case asserts that for $\gamma<\frac{1}{2}$,
\begin{equation*}
 \Bigl(\int_{\R^N} \frac{\abs{u (x)}^p}{\abs{x}^{\frac{N - \alpha}{2}}(1 + \abs{\log \abs{x}})^\gamma}\dif x\Bigr)^2
 \le C\int_{\R^N} \abs{I_{\alpha/2} \ast \abs{u}^p}^2
\end{equation*}
for all $u\in E^{\alpha,p}(\R^N)$. This particular inequality was established by Ruiz \cite{Ruiz-ARMA}*{theorem 3.1},
but the new proof we give is more direct.

\subsection{Existence of groundstates.}
We prove the existence of optimizers for the Coulomb--Sobolev inequalities except at the Coulomb--Sobolev critical exponent \(q=2 \frac{2 p + \alpha}{2 + \alpha}\) and Sobolev critical exponent $q=\frac{2N}{N-2}$.

\begin{theorem}[Existence of optimizers]\label{t-multiplicative-intro}
Let \(N \in \N\), \(\alpha \in (0, N)\) and \(p,q\in[1,+\infty)\)
be such that the following assumption holds:
\begin{equation}\tag{$\mathcal Q'$}\label{Q_0}
\begin{split}
\text{either}\quad & \frac{1}{p} > \frac{(N - 2)_+}{N + \alpha} \quad\text{and}\quad
   \frac{1}{2} - \frac{1}{N} < \frac{1}{q} < \frac{1}{2} - \frac{p - 1}{\alpha + 2 p}\,,\bigskip\\
\text{or}\quad & \frac{1}{p} < \frac{(N - 2)_+}{N + \alpha} \quad\text{and}\quad
   \frac{1}{2} - \frac{1}{N} > \frac{1}{q} > \frac{1}{2} - \frac{p - 1}{\alpha + 2 p}\,.
\end{split}
\end{equation}
Then the best constant
\begin{equation}\label{Coulomb-Sobolev-estimate-best}
S:=\inf_{u\in E^{\alpha,p}(\R^N)\setminus\{0\}}\frac{\Bigl(\int_{\R^N} \abs{D u}^2\Bigr)^\frac{\theta}{2}  \Bigl( \int_{\R^N} \bigl(I_{\alpha/2} \ast \abs{u}^p\bigr)^2\Bigr)^{\frac{1 - \theta}{2 p}}}{\Bigl(\int_{\R^N} \abs{u}^q\Bigr)^{\frac{1}{q}}}
\end{equation}
where $\theta$ is given by \eqref{eq-theta}, is achieved.
\end{theorem}

Existence of optimizers for \eqref{Coulomb-Sobolev-estimate-best} was previously established
in \cite{BellazziniFrankVisciglia}*{theorem 2.2} for $N\in\N$, $\alpha\in(0,N)$ and $p=2$.

Existence of optimizers for the \emph{multiplicative} minimization problem
\eqref{Coulomb-Sobolev-estimate-best} is equivalent up to a rescaling to the existence of optimizers
for the constrained \emph{additive} minimization problem
\begin{equation}\label{Coulomb-Sobolev-additive-best-intro}
M_c:=\inf\Big\{\mathcal{E}_*(u)\;\big|\;u\in E^{\alpha,p}(\R^N), \int_{\R^N}\abs{u}^q=c\Big\},
\end{equation}
where $c>0$ and we denote
\begin{equation}\label{E-defn-intro}
\mathcal{E}_*(u)=\frac{1}{2}\int_{\R^N}\abs{D u}^2+\frac{1}{2p}\int_{\R^N} \abs{I_{\alpha/2} \ast \abs{u}^p}^2.
\end{equation}
The direct proof of the existence of optimizers for $M_c$ which we give in theorem~\ref{thm-additive}
provides some additional understanding about the behaviour of minimizing sequences. In particular we show that all the minimising sequences are relatively compact modulo translations, this result is in the spirit of P.-L. Lions \cite{Lions1984CC1}.

The multiplicative and additive minimization problems share up to a rescaling the same Euler--Lagrange equation,
in the sense that minimizers of \eqref{Coulomb-Sobolev-additive-best-intro} and,
after a rescaling, minimizers of \eqref{Coulomb-Sobolev-estimate-best}
are weak solutions of the equation
\begin{equation}\label{sps-mu}
 - \Delta u + (I_\alpha \ast \abs{u}^p)\abs{u}^{p - 2} u= \mu\abs{u}^{q-2}u\quad\text{in \(\R^N\),}
\end{equation}
with an unknown Lagrange multiplier $\mu>0$, see section \ref{sect-additive}.
If $q\neq 2\frac{\alpha+2p}{\alpha+2}$ then a rescaling of solutions of \eqref{sps-mu} allows to get rid
of the Lagrange multiplier and to obtain a solution of the original equation \eqref{sps}.
Solutions of \eqref{sps} obtained as rescaled minimizers of
\eqref{Coulomb-Sobolev-estimate-best} or \eqref{Coulomb-Sobolev-additive-best-intro} are called in what follows
\emph{groundstate} solutions.

The proof of theorem~\ref{t-multiplicative-intro} relies on a ``compactness up to translations'' type lemma (lemma~\ref{liebtypelemma0}, see also \cite{LiebLoss2001}*{p.\thinspace{}215}) and on a novel nonlocal Brezis--Lieb type lemma (proposition~\ref{propositionBrezisLieb}).

In contrast with its local counterpart \cite{BrezisLieb1983}, our nonlocal Brezis--Lieb lemma is an \emph{inequality}: we prove that if \((u_n)_{n \in \N}\) converges to \(u : \R^N \to \R\) almost everywhere and if \((I_{\alpha/2} \ast \abs{u_n}^p)_{n \in \N}\) is bounded in \(L^2 (\R^N)\), then
\begin{equation}
\label{ineqIntroBL}
  \liminf_{n \to \infty}
  \int_{\R^N} \bigabs{I_{\alpha/2} \ast \abs{u_n}^p}^2 - \bigabs{I_{\alpha/2} \ast \abs{u_n - u}^p}^2 \ge \int_{\R^N} \bigabs{I_{\alpha/2} \ast \abs{u}^p}^2.
\end{equation}
This inequality is sufficient for the purpose of proving theorem~\ref{t-multiplicative-intro}.

It is natural to ask whether equality holds in the inequality \eqref{ineqIntroBL}, similarly
to the classical local Brezis--Lieb lemma.
The answer is that equality holds if and only if \((u_n)_{n \in \N}\) converges
strongly to \(u\) in \(L^p_{\mathrm{loc}} (\R^N)\) (proposition~\ref{propositionBrezisLiebEquiv}).
This will be the case if \((u_n)_{n \in \N}\) is bounded in some \(L^q_{\mathrm{loc}} (\R^N)\) for some \(q > p\), allowing to recover a result of Moroz and Van Schaftingen  when \(q = \frac{2 N p}{N + \alpha}\) \cite{MorozVanSchaftingen}*{lemma 2.4}  and  of  Bellazzini, Frank and Visciglia when \(p = 2\) and \(q > 2\) \cite{BellazziniFrankVisciglia}*{lemma 6.2}.

One can also wonder whether a stronger assumption of boundedness in the
Coulomb--Sobolev \(E^{\alpha, p} (\R^N)\) could imply equality in \eqref{ineqIntroBL},
or equivalently strong compactness in \(L^p_{\mathrm{loc}} (\R^N)\).
An elementary local estimate on Riesz potentials (proposition~\ref{propositionRuizBall}) always ensures the local
embedding $E^{\alpha,p}(\R^N)\subset L^p_{\textrm{loc}}(\R^N)$.
However, this embedding is compact
if and only if $p<\frac{2\alpha}{\alpha-2}$.
In fact, for each $\alpha>2$ and $p\ge \frac{2\alpha}{\alpha-2}$
we construct a sequence of smooth functions which is bounded in $E^{\alpha,p}(\R^N)$,
converges almost everywhere to zero and has $L^p$-norm which is bounded away from zero.
Functions in the sequence have uniformly bounded supports located around a $d$--dimensional hypercube
in $\R^N$ with $d>N-\alpha$ (see lemma~\ref{lemmaCoulombSobolev-noncompact}).

Going back to the question about equality in the Brezis--Lieb inequality \eqref{ineqIntroBL}, we obtain that if \(\alpha \le 2\) or \(p < \frac{2 \alpha}{\alpha - 2}\), and if the sequence \((u_n)_{n \in \N}\) is bounded in the Coulomb--Sobolev space \(E^{\alpha, p} (\R^N)\), then equality holds in \eqref{ineqIntroBL}.
We show that this restriction is optimal by constructing for every $p\ge\frac{2\alpha}{(\alpha-2)_+}$ a sequence of smooth functions whose support concentrate
to a bounded Cantor--like set of positive \(H^{\alpha/2}\)--capacity and vanishing Lebesgue measure (lemma~\ref{lemmaCantor}) .

In section~\ref{sect-equation} we study some additional qualitative properties of solutions
of \eqref{sps} such as regularity and positivity. The following theorem summarizes our findings.

\begin{theorem}[Existence of groundstates]\label{t-groundstate}
Let \(N \in \N\), \(\alpha \in (0, N)\) and \(p,q\in(1,+\infty)\).

Assume that assumption \eqref{Q_0} holds.
Then there exists a nontrivial nonnegative groundstate solution $u\in E^{\alpha,p}(\R^N)\cap C^2(\R^N)$ to equation \eqref{sps} and $u\in C^\infty(\R^N\setminus u^{-1}(0))$.
In addition, if $p\ge 2$ then $u(x)>0$ for all $x\in \R^N$.
\end{theorem}

In the case $N=3$, $\alpha=2$, $p=2$ the existence of a positive solution to \eqref{sps}
was proved by Ianni and Ruiz \cite{Ianni-Ruiz-2012}*{theorem 1.2} using a mountain--pass type argument.

We do not know whether or not the restriction $p\ge 2$ is essential for the positivity of groundstates.
We also do not know whether or not groundstate solutions obtained in Theorem~\ref{t-groundstate} are radial.
We study the existence of radial groundstates separately.

\subsection{Radial estimates and existence of radial groundstates.}
In section~\ref{sect-radial} we study the embeddings of the subspace $E^{\alpha, p}_\rad (\R^N)$ of radially symmetric functions in $E^{\alpha, p}(\R^N)$ into
Lebesgue spaces, in the spirit of the seminal result of Strauss \cite{Strauss1977} (see also \citelist{\cite{SuWangWillem2007b}\cite{SuWangWillem2007}})
and its counterpart for some Coulomb--Sobolev spaces \cite{Ruiz-ARMA} (see also \citelist{\cite{Mercuri2008}\cite{BonheureMercuri2011}\cite{BellaziniGhimentiOzawa}}).
For $\alpha>1$ the radial embedding intervals are wider then the intervals given by Theorem~\ref{theoremEmbedding}: the critical Coulomb--Sobolev exponent \(2 \frac{2 p + \alpha}{2 + \alpha}\) is replaced by a stronger critical exponent.

\begin{theorem}[Radial embeddings]\label{theoremRadialSobolev-intro}
Let \(N \in \N\), \(\alpha \in (0, N)\) and \(p,q\in[1,+\infty)\).

If \(\alpha \le 1\), then \(E^{\alpha, p}_\rad (\R^N)\) is embedded in \(L^q (\R^N)\) if and only if \(E^{\alpha, p} (\R^N)\) is embedded in \(L^q (\R^N)\).

If \(\alpha > 1\), then
\(E^{\alpha, p}_\rad (\R^N)\) is embedded in \(L^q (\R^N)\) if and only if
the following assumption holds:
\begin{equation}\tag{$\mathcal Q_{\mathrm{rad}}$}\label{Qrad}
\begin{split}
\text{either}\quad & \frac{1}{p} \ge \frac{(N - 2)_+}{N + \alpha} \quad\text{and}\quad
   \frac{1}{2} - \frac{1}{N} \le \frac{1}{q} < \frac{3 N + \alpha - 4}{2(2 p (N - 1) + N - \alpha)}\,,\bigskip\\
\text{or}\quad & \frac{1}{p} \le \frac{(N - 2)_+}{N + \alpha} \quad\text{and}\quad
   \frac{1}{2} - \frac{1}{N} \ge \frac{1}{q} > \frac{3 N + \alpha - 4}{2(2 p (N - 1) + N - \alpha)}\,.
\end{split}
\end{equation}
In this case, the
interpolation estimate \eqref{Coulomb-Sobolev-estimate} is valid with $\theta$ given by \eqref{eq-theta}.

\end{theorem}

The embedding \(E^{2, 2}_\rad (\R^3)\subset L^q (\R^3)\) for $q\in(18/7,6]$
was established by Ruiz \cite{Ruiz-ARMA}*{theorem 1.2}. Ruiz also conjectured that \(E^{2, 2}_\rad (\R^3)\not\subset L^{18/7}(\R^3)\)
\cite{Ruiz-ARMA}*{remark 4.1}.

The intervals of assumption \eqref{Qrad} are wider then intervals in \eqref{Q}.
As a consequence, for $N\ge 2$, $\alpha>1$ and $c>0$ we establish the existence of optimizers for the radial
minimization problem
\begin{equation}\label{Coulomb-Sobolev-additive-best+r}
M_{c,\rad}:=\inf\Big\{\mathcal{E}_*(u)\;\big|\;u\in E^{\alpha,p}_\rad(\R^N), \int_{\R^N}\abs{u}^q=c\Big\}
\end{equation}
where the functional $\mathcal{E}_*$ is defined by \eqref{E-defn-intro}, for a range of $q$ which is wider then the range given in theorem~\ref{t-multiplicative}.
In particular, the existence range will include the Coulomb-Sobolev critical exponent $q=2\frac{2p+\alpha}{2+\alpha}$
where the Euler--Lagrange equation \eqref{sps-mu} should be interpreted as a nonlinear eigenvalue problem rather then as an equation,
since the Lagrange multiplier in \eqref{sps-mu} cannot be scaled out.
The following summarizes the results.

\begin{theorem}\label{t-radial-intro}
Let $N\geq 2$, $\alpha\in (0,N)$ and $p,q\in[1,+\infty)$.
Assume that either $\alpha\le 1$ and \eqref{Q_0} holds,
or $\alpha>1$ and the following condition holds:
\begin{equation}\tag{$\mathcal Q'_\rad$}\label{Qrad0}
\begin{split}
\text{either}\quad & \frac{1}{p} > \frac{(N - 2)_+}{N + \alpha} \quad\text{and}\quad
   \frac{1}{2} - \frac{1}{N} < \frac{1}{q} < \frac{3 N + \alpha - 4}{2(2 p (N - 1) + N - \alpha)}\,,\bigskip\\
\text{or}\quad & \frac{1}{p} < \frac{(N - 2)_+}{N + \alpha} \quad\text{and}\quad
   \frac{1}{2} - \frac{1}{N} > \frac{1}{q} > \frac{3 N + \alpha - 4}{2(2 p (N - 1) + N - \alpha)}.\,
\end{split}
\end{equation}
Then the embedding \(E^{\alpha, p}_\rad (\R^N)\subset L^q (\R^N)\) is compact and
the constrained minimization problem \eqref{Coulomb-Sobolev-additive-best+r}
admits a nonnegative optimizer $w\in E^{\alpha,p}_\rad(\R^N)$.
When $\alpha\neq 1$ the conditions are also necessary for the compactness of the embedding.

In addition, assume that $p>1$.
Then for $q\neq 2\frac{2p+\alpha}{2+\alpha}$ a rescaling of $w$ solves equation \eqref{sps},
while if $q=2\frac{2p+\alpha}{2+\alpha}$ then $w$ solves the eigenvalue problem
\begin{equation}\label{sps-mu*}
 - \Delta w + (I_\alpha \ast \abs{w}^p)\abs{w}^{p - 2} w=qM_{1,\rad}\abs{w}^{q-2}w\quad\text{in \(\R^N\).}
\end{equation}
Moreover, $w\in C^2(\R^N)$, $w\in C^\infty(\R^N\setminus w^{-1}(0))$ and $w(|x|)\to 0$ as $\abs{x}\to\infty$.
If $p\ge 2$ then $w(|x|)>0$ for all $x\in \R^N$.
\end{theorem}

For $N=3$, $\alpha=2$, $p=2$ the above theorem was proved by Ruiz in \cite{Ruiz-ARMA}*{theorem 1.3} (case $q\in(18/7,3)$)
and by Ianni and Ruiz \cite{Ianni-Ruiz-2012}*{theorem 1.2 and theorem 1.3} (cases $q\in(3,6)$ and $q=3$).

In the Coulomb--Sobolev critical case $q=2\frac{2p+\alpha}{2+\alpha}$ the problem is invariant by scaling,
in the sense that if $w$ is a solution of
\begin{equation}\label{sps-mu**}
 - \Delta u + (I_\alpha \ast \abs{u}^p)\abs{u}^{p - 2} u= q\mu \abs{u}^{q-2}u\quad\text{in \(\R^N\)}
\end{equation}
for some $\mu>0$, then for every $\lambda>0$ the function $w_\lambda(x)=\lambda^{-\frac{\alpha+2}{2(p-1)}}w(x/\lambda)$ is also a solution of \eqref{sps-mu**} and $\mathcal{E}_*(w_\lambda)\equiv \mathcal{E}_*(w)$. In particular, the scale invariance implies that $M_c\equiv M_1$
and $M_{c,\rad}\equiv M_{1,\rad}$ for every $c>0$ (see also \eqref{scaledinfimum}).
Note that unless $p=\frac{N+\alpha}{N-2}$, the $L^q$--norm of $w_\lambda$ is not preserved and the scale invariance does not lead to the
nonuniqueness of the minimizer for $M_c$ or $M_{c,\rad}$.
We show in remark~\ref{universalmubound} that if  $q=2\frac{2p+\alpha}{2+\alpha}$, $p\neq\frac{N+\alpha}{N-2}$ and
the pair $(\mu,u)\in \R_+\times E^{\alpha,p}(\R^N)\setminus\{0\}$ is a solution of \eqref{sps-mu**} then
$$
\mu\ge M_1,
$$
so that $M_1$ could be seen as the least eigenvalue of \eqref{sps-mu**}.
This further justifies that in the Coulomb--Sobolev critical case \eqref{sps-mu**} should be interpreted as
an eigenvalue problem rather then an equation.

\subsection{Open questions.}
We close this introduction by listing several problems that are left open in the present work.

\subsubsection{Radial versus non radial minimizers.}
Radial symmetry of the minimizer constructed in theorem~\ref{t-multiplicative-intro} is open.
It is also unclear whether the optimal constants $M_1$ and $M_{1,\rad}$ share the same value.
A result in \cite{Ruiz-ARMA}*{theorem 1.7} shows that for $N=3$, $\alpha=2$, $p=2$ and $q<3$
a breakup of symmetry may occur when equation \eqref{sps} is considered with Dirichlet boundary data
on a ball of sufficiently large radius. This construction however excludes the range of $q$ where
both $M_1$ and $M_{1,\rad}$ are attained.

\subsubsection{Uniqueness of the minimizers.}
Uniqueness up to translations of the minimizers constructed in theorems~\ref{t-multiplicative-intro} and~\ref{t-radial-intro} is open. If the minimizer for \eqref{Coulomb-Sobolev-estimate-best} was a nonradial function, then due to rotations, we would have a strong non-uniqueness situation.

\subsubsection{Positivity versus dead--cores for $p<2$.}
Positivity of the groundstate solutions constructed in theorems~\ref{t-multiplicative-intro} and~\ref{t-radial-intro} is open for $p<2$ and a-priori, the possibility of dead cores (regions in which the solution vanishes) cannot be ruled out.

\subsubsection{Radial compactness in the case $\alpha=1$ and $q=\frac{2}{3}(2p+1)$.}
 Compactness of the embedding $E^{1,p}_\rad (\R^N)\subset L^{\frac{2}{3}(2p+1)}(\R^N)$ in the case $(N-2)p\neq N+1$
 is open. If answered positively, this would lead to the existence of a radial solution to the eigenvalue problem
\eqref{sps-mu*} in that case.

\subsubsection{Double--critical case.}
The existence of optimizers in the double--critical case $p=\frac{N+\alpha}{N-2}$
and $q=2\frac{2p+\alpha}{2+\alpha}=\frac{2N}{N-2}$ is open. Observe that
in the particular case $N\ge 3$ and $\alpha=2$ solutions of \eqref{sps-mu} for
$p=\frac{N+2}{N-2}$ and $q=p+1=\frac{2N}{N-2}$ can be formally constructed
from the explicit radial groundstate solutions $U\in E^{2,p}_\rad(\R^N)$ of the critical Emden--Fowler equation
\begin{equation*}
-\Delta U=U^\frac{N+2}{N-2},\quad U>0\quad\text{in $\R^N$.}
\end{equation*}
Indeed, then $I_2\ast |U|^\frac{N+2}{N-2}=U$ and therefore,
$$-\Delta U+(I_2\ast U^{\frac{N+2}{N-2}})U^{\frac{4}{N-2}}=2 U^{\frac{N+2}{N-2}}\quad\text{in $\R^N$.} $$
It is unclear however whether $U$ is a groundstate of \eqref{sps}, i.e. if it is
a minimizer of \eqref{Coulomb-Sobolev-additive-best+r}.

\section{Coulomb--Sobolev spaces}\label{Sect-CS}
\settocdepth{subsection}

\subsection{Definition of Coulomb--Sobolev spaces}

\begin{definition}
Let $N\in\N$, $\alpha \in (0,N)$ and $p\geq 1$.
We define the \emph{Coulomb space} \(Q^{\alpha, p} (\R^N)\) as the vector space of measurable functions \(u:\R^N\to\R\) such that
$$
  \norm{u}_{Q^{\alpha, p}}:= \Bigl(\int_{\R^N}\bigabs{I_{\alpha/2} \ast \abs{u}^p} ^2\Bigr)^{\frac{1}{2 p}} < \infty.
$$
\end{definition}

It can be observed that for every measurable function, \(u \in Q^{\alpha, p} (\R^N)\) if and only if \(\abs{u}^p \in Q^{\alpha, 1} (\R^N)\).
By the Hardy--Littlewood--Sobolev inequality \cite{LiebLoss2001}*{theorem 4.3} we have
\[
 \int_{\R^N} \bigabs{I_{\alpha/2} \ast \abs{u}^p}^2 \le C \Bigl(\int_{\R^N} \abs{u}^{\frac{2 N p}{N + \alpha}} \Bigr)^\frac{N + \alpha}{N}
\]
and thus \(L^\frac{2 N p}{N + \alpha} (\R^N) \subset Q^{\alpha, p} (\R^N)\).

\begin{proposition}
\label{norma}
Let $N\in\N$, $\alpha \in (0,N)$ and $p\geq 1$. Then \(\norm{\cdot}_{Q^{\alpha, p}}\) defines a norm.
\end{proposition}

\begin{proof}
By the classical integral weighted Minkowski inequality \cite{LiebLoss2001}*{theorem 2.4} we have for every \(x \in \R^N\),
\[
 \bigabs{(I_{\alpha/2} \ast \abs{u + v}^p) (x)}^\frac{1}{p}
 \le \bigabs{(I_{\alpha/2} \ast \abs{u}^p) (x)}^\frac{1}{p}
 + \bigabs{(I_{\alpha/2} \ast \abs{v}^p) (x)}^\frac{1}{p}.
\]
By integrating and applying the Minkowski inequality in \(L^{2 p} (\R^N)\) we conclude that
\[
\begin{split}
 \Bigl(\int_{\R^N} \bigabs{(I_{\alpha/2} \ast \abs{u + v}^p)}^2\Bigr)^\frac{1}{2 p}
 &\le \Bigl(\int_{\R^N} \bigabs{(I_{\alpha/2} \ast \abs{u}^p)}^2\Bigr)^\frac{1}{2 p} + \Bigl(\int_{\R^N} \bigabs{(I_{\alpha/2} \ast \abs{v}^p) }^2\Bigr)^\frac{1}{2 p}.\qedhere
\end{split}
\]
\end{proof}

When $p=1$ the norm \(\norm{\cdot}_{Q^{\alpha, 1}}\) is not generated by an inner product and does not coincide with the $\alpha$--energy norm
of the classical potential theory. The latter is defined by $\norm{u}_{\mathcal E^{\alpha}}:= \|I_{\alpha/2} \ast u\|_{L^2(\R^N)}$, see for example \cite{DuPlessis}*{pp.\thinspace{}80--81}.

\begin{definition}
Let $N\in\N$, $\alpha \in (0,N)$ and $p\geq 1$.
We define the \emph{Coulomb--Sobolev space} \(E^{\alpha,p} (\R^N)\) as the vector space of functions \(u \in Q^{\alpha, p} (\R^N)\) such that \(u\) is weakly differentiable in \(\R^N\), \(Du \in L^2 (\R^N)\) and
\begin{equation*}
  \norm{u}_{E^{\alpha, p}}:=\biggl(\int_{\R^N} \abs{D u}^2+\Bigl(\int_{\R^N}\big|I_{\alpha/2}*\abs{u}^p\big|^2\Bigr)^{1/p}\biggr)^{1/2}<\infty.
\end{equation*}
The function \(\norm{\cdot}_{E^{\alpha, p}}\) defines a norm in view of proposition~\ref{norma}.
\end{definition}

The space $E^{2,2}(\R^3)$ was introduced and studied by Ruiz \cite{Ruiz-ARMA}*{section 2}.

\subsection{Completeness of Coulomb--Sobolev spaces}
It is not difficult to see that the Coulomb space \(Q^{\alpha, p} (\R^N)\) is complete.
We are going to prove that the Coulomb--Sobolev space \(E^{\alpha, p} (\R^N)\) is a Banach space.

\begin{proposition} \label{propositionCompleteness}
For all $N\in\N$, $\alpha \in (0,N)$ and $p\geq 1$ the normed space \(E^{\alpha, p} (\R^N)\) is complete.
\end{proposition}

The first ingredient of the completeness is the following local estimate, which in particular implies
that $E^{\alpha, p} (\R^N)\subset L^p_\loc(\R^N)$.

\begin{proposition}
\label{propositionRuizBall}
For all $N\in\N$, $\alpha \in (0,N)$ and $p\geq 1$, there exists \(C>0\) such that for every \(a \in \R^N\) and \(\rho > 0\),
\[
  \int_{B_\rho (a)} \abs{u}^p \le C\rho^\frac{N - \alpha}{2} \Bigl( \int_{B_\rho(a)} \bigabs{I_{\alpha/2} \ast \abs{u}^p}^2\Bigr)^\frac{1}{2}.
\]
\end{proposition}
\begin{proof}
We observe that for every \(x \in B_\rho (a)\),
\[
  (2\rho)^{-(N-\alpha/2)}\int_{B_\rho (a)} \abs{u}^pdy \le \int_{B_\rho (a)} \frac{|u(y)|^p}{\abs{x - y}^{N-\alpha/2}}dy \leq C_1\bigl(I_{\alpha/2} \ast \abs{u}^p\bigr) (x).
\]
The conclusion follows by integration.
\end{proof}

The second ingredient is the following Fatou property for locally converging sequences.

\begin{proposition}[Fatou property]
\label{propositionSemiContinuityLocalConvergence}
Let $N\in\N$, $\alpha \in (0,N)$ and $p\geq 1$.
If \((u_n)_{n \in \N}\) is a bounded sequence in \(E^{\alpha, p} (\R^N)\) that converges to a
function \(u:\R^N\to\R\) in \(L^1_{\mathrm{loc}} (\R^N)\), then
\(u \in E^{\alpha, p} (\R^N)\),
\[
 \int_{\R^N} \abs{D u}^2
 \le \liminf_{n \to \infty} \int_{\R^N} \abs{D u_n}^2
\]
and
\[
  \int_{\R^N} \bigabs{I_{\alpha/2} \ast \abs{u}^p}^2 \le \liminf_{n \to \infty} \int_{\R^N} \bigabs{I_{\alpha/2} \ast \abs{u_n}^p}^2.
\]
Moreover $Du_n\rightharpoonup Du$ weakly in $L^2(\R^N)$.
\end{proposition}

This property is known in the context of classical Sobolev spaces \cite{Willem2013}*{theorem 6.1.7}.

\begin{proof}[Proof]
If \(\psi \in C^1_c (\R^N; \R^N)\), then
\[
 -\int_{\R^N} u \Div \psi
 =-\lim_{n \to \infty}\int_{\R^N} u_n \Div \psi
 =\lim_{n \to \infty} \int_{\R^N} Du_n \cdot \psi,
\]
and therefore by the Cauchy--Schwarz inequality
\[
 \Bigabs{\int_{\R^N} u \Div \psi} \le \liminf_{n \to \infty} \Bigl( \int_{\R^N} \abs{D u_n}^2\Bigr)^\frac{1}{2} \Bigl(\int_{\R^N} \abs{\psi}^2 \Bigr)^\frac{1}{2}.
\]
By the Hahn--Banach theorem, the distribution \(D u\) induces then a linear functional on \(L^2 (\R^N)\). Therefore, by the Riesz representation theorem, there exists \(F \in L^2 (\R^N; \R^N)\) such that for every \(\psi \in C^1_c (\R^N; \R^N)\),
\[
 -\int_{\R^N} u \Div \psi =\int_{\R^N} F \cdot \psi,
\]
and
\[
 \int_{\R^N} \abs{F}^2 \le \liminf_{n \to \infty} \int_{\R^N} \abs{D u_n}^2,
\]
that is \(F = D u\) in the weak sense and \(Du\) satisfied the required estimate. From the fact that \(F = D u\) it follows that $Du_n\rightharpoonup Du$ weakly in $L^2(\R^N)$.

Next we observe that by Fatou's lemma, for almost every \(x \in \R^N\) holds
\[
 \bigl(I_{\alpha/2} \ast \abs{u}^p\bigr) (x)
 \le \liminf_{n \to \infty} \bigl(I_{\alpha/2} \ast \abs{u_n}^p\bigr) (x).
\]
Hence, by Fatou's lemma again,
\[
 \int_{\R^N} \bigabs{I_{\alpha/2} \ast \abs{u}^p}^2
 \le \liminf_{n \to \infty} \int_{\R^N} \bigabs{I_{\alpha/2} \ast \abs{u_n}^p}^2.\qedhere
\]
\end{proof}

\begin{proof}[Proof of proposition~\ref{propositionCompleteness}]
Let \((u_n)_{n \in \N}\) be a Cauchy sequence in $E^{\alpha,p}(\R^N)$.
By Proposition~\ref{propositionRuizBall}, \((u_n)_{n \in \N}\) is also a Cauchy sequence in \(L^p_{\mathrm{loc}} (\R^N)\). Hence there exists \(u \in L^p_{\mathrm{loc}} (\R^N)\) such that \((u_n)_{n \in \N}\) converges to \(u\) in \(L^p_{\mathrm{loc}} (\R^N)\).
In view of proposition~\ref{propositionSemiContinuityLocalConvergence} we conclude that \(u \in E^{\alpha, p} (\R^N)\).
Moreover, for every \(n \in \N\) the sequence \((u_n - u_m)_{m \in \N}\) converges to \(u_n - u\) in \(L^p_{\mathrm{loc}} (\R^N)\). Hence, by proposition~\ref{propositionSemiContinuityLocalConvergence} again,
\[
 \int_{\R^N} \abs{D u_n - D u}^2 + \int_{\R^N} \bigabs{I_\alpha \ast \abs{u_n - u}^p }^2
 \le \limsup_{m \to \infty} \int_{\R^N} \abs{D u_n - D u_m}^2 + \int_{\R^N} \bigabs{I_\alpha \ast \abs{u_n - u_m}^p }^2.
\]
Since \((u_n)_{n \in \N}\) is a Cauchy sequence in $E^{\alpha,p}(\R^N)$,
\begin{multline*}
 \limsup_{n \to \infty} \int_{\R^N} \abs{D u_n - D u}^2 + \int_{\R^N} \bigabs{I_\alpha \ast \abs{u_n - u}^p }^2\\
 \le \limsup_{n \to \infty} \limsup_{m \to \infty} \int_{\R^N} \abs{D u_n - D u_m}^2 + \int_{\R^N} \bigabs{I_\alpha \ast \abs{u_n - u_m}^p }^2\\
 \le \limsup_{m, n \to \infty} \int_{\R^N} \abs{D u_n - D u_m}^2 + \int_{\R^N} \bigabs{I_\alpha \ast \abs{u_n - u_m}^p }^2 = 0.
\end{multline*}
\end{proof}

We complete this description of the completeness properties of \(E^{\alpha, p} (\R^N)\) with an embryonic local compactness result.

\begin{proposition}[Elementary local compactness]
\label{propositionElementaryLocalWeakCompactness}
Let $N\in\N$, $\alpha \in (0,N)$ and $p\ge 1$.
If \((u_n)_{n \in \N}\) is a bounded sequence in \(E^{\alpha, p} (\R^N)\),
then there exists a subsequence \((u_{n_k})_{k \in \N}\) that converges in \(L^1_\mathrm{loc} (\R^N)\).
\end{proposition}
\begin{proof}
By proposition~\ref{propositionRuizBall} and H\"older inequality, the sequence of functions \((u_n)_{n \in \N}\) is bounded in \(L^1_\mathrm{loc} (\R^N)\) and thus in \(W^{1, 1}_{\mathrm{loc}}(\R^N)\).
The conclusion then follows from the classical Rellich theorem and a diagonal argument.
\end{proof}

\subsection{Density of test functions in Coulomb--Sobolev spaces}

We are going to show that Coulomb--Sobolev space \(E^{\alpha, p}(\R^N)\) can be naturally identified with the completion of the set of test functions \(C^\infty_c (\R^N)\) under the norm \(\norm{\cdot}_{E^{\alpha, p}}\).

\begin{proposition}[Density of test functions]\label{smoothapprox}
Let $N\in\N$, $\alpha \in (0,N)$ and $p\ge 1$.
The space of test functions \(C^\infty_{c}(\R^N)\) is dense in the Coulomb space \(E^{\alpha, p} (\R^N)\).
\end{proposition}

For \(N \ge 3\) this implies that the Coulomb--Sobolev space \(E^{\alpha, p} (\R^N)\) can be represented as
$$
 E^{\alpha,p} (\R^N) = Q^{\alpha, p} (\R^N) \cap \mathcal{D}^{1, 2}(\R^N),
$$
where \(\mathcal{D}^{1, 2}(\R^N)\) is the homogeneous Sobolev space of weakly differentiable functions \citelist{\cite{Willem2013}*{definition 7.2.1}}.
Taking this as a definition would not allow to cover the low dimensions \(N \in \{1, 2\}\),
when \(\mathcal{D}^{1, 2}(\R^N)\) is not defined.

\begin{proof}[Proof of proposition~\ref{smoothapprox}]
Let \(\theta \in C^\infty (\R)\) be a cut-off function such that \(\abs{\theta'} \le 1\) on \(\R\), \(\theta = 0\) on \([-1, 1]\) and \(\theta (t) = t\) if \(\abs{t} \ge 2\). For \(n\in\N\), let \(\theta_n (t) = \theta (n t)/n\). Choose a nonnegative radial test function \(\eta \in C^\infty_c (\R^N)\) such that \(\int_{\R^N} \eta = 1\) and \(\supp \eta \subset B_1\) and set \(\eta_n (x) = n^N \eta (n x)\).
For a given \(u \in E^{\alpha, p} (\R^N)\), define
\[
 u_n: = \theta_n \circ (\eta_n \ast u).
\]
By proposition~\ref{propositionRuizBall}, \(u \in L^1_\mathrm{loc} (\R^N)\) and therefore $u_n$ is well-defined.
By smoothing properties of the convolution and by the chain rule, \(u_n \in C^\infty (\R^N)\).

We are going to show that the support of \(u_n \) is compact.
Observe that for each \(x \in \R^N\), by H\"older inequality and by proposition~\ref{propositionRuizBall},
\[
  \abs{\eta_n \ast u (x)} \le C n^N \int_{B_{1/n} (x)} \abs{u}
  \le C n^\frac{N}{p} \Bigl(\int_{B_{1/n} (x)} \abs{u}^p \Bigr)^\frac{1}{p}
  \le C' n^\frac{N + \alpha}{2 p} \Bigl(\int_{B_{1/n} (x)} \bigabs{I_{\alpha/2} \ast \abs{u}^p}^2 \Bigr)^\frac{1}{2 p}.
\]
Since \(\int_{\R^N} \bigabs{I_{\alpha/2} \ast \abs{u}^p}^2 < \infty\), we deduce therefrom that
\(\lim_{\abs{x} \to \infty} (\eta_n \ast u) (x) = 0\), and thus the set \(\supp (\theta_n \circ (\eta_n \ast u))\) is compact.

Moreover, since the function \(u\) is locally integrable and weakly differentiable, the sequence \((\eta_n \ast u)_{n \in \N}\)
converges to \(u\) in \(W^{1, 1}_{\mathrm{loc}} (\R^N)\).
By the properties of \(\theta_n\) and by Lebesgue's dominated convergence theorem, it follows immediately that the sequence \((u_n)_{n \in \N}\) converges to \(u\) in \(W^{1,1}_{\mathrm{loc}} (\R^N)\).

Since \(\eta\) is nonnegative, we observe that for every \(n \in \N\),
\[
 \int_{\R^N} \abs{D u_n}^2
 \le \int_{\R^N} \abs{D (\rho_n \ast u)}^2
 \le \int_{\R^N} \abs{D u}^2.
\]
Since the sequence \((D u_n)_{n \in \N}\) converges to \(D u\) in measure, it follows that
\[
 \lim_{n \to \infty} \int_{\R^N} \abs{D u_n - D u}^2 = 0
\]
(cf.\thinspace{}\cite{Willem2013}*{proposition 4.2.6}).

Next, by Fatou's lemma, we have for every \(x \in \R^N\),
\begin{equation}
\label{eqDensityRieszLiminf}
 \liminf_{n \to \infty} (I_{\alpha/2} \ast \abs{u_n}^p) (x) \ge (I_{\alpha/2} \ast \abs{u}^p)(x).
\end{equation}
On the other hand, by Jensen's inequality
\[
 I_{\alpha/2} \ast \abs{u_n}^p \le I_{\alpha/2} \ast \abs{\rho_n \ast u}^p
 \le \rho_n \ast (I_{\alpha/2} \ast \abs{u}^p),
\]
and hence
\begin{equation}\label{eqDensityRieszLimsup}
 \limsup_{n \to \infty} \int_{\R^N} \bigabs{I_{\alpha/2} \ast \abs{u_n}^p}^2
 \le \int_{\R^N} \bigabs{I_{\alpha/2} \ast \abs{u}^p}^2.
\end{equation}
Therefore, we deduce from \eqref{eqDensityRieszLiminf} and \eqref{eqDensityRieszLimsup} that for almost every \(x \in \R^N\),
\begin{equation}
\label{eqDensityRieszLim}
 \lim_{n \to \infty} (I_{\alpha/2} \ast \abs{u_n}^p) (x) = (I_{\alpha/2} \ast \abs{u}^p)(x),
\end{equation}
and thus (see for example \cite{Willem2013}*{proposition 4.2.6}), for every such \(x \in \R^N\),
\[
  \lim_{n \to \infty} (I_{\alpha/2} \ast \abs{u_n - u}^p) (x) = 0.
\]
We deduce also from \eqref{eqDensityRieszLimsup} and \eqref{eqDensityRieszLim} that
\[
 \lim_{n \to \infty} \int_{\R^N} \bigabs{I_{\alpha/2} \ast \abs{u_n}^p - I_{\alpha/2} \ast \abs{u}^p}^2 = 0.
\]
Finally, since
\[
 I_{\alpha/2} \ast \abs{u_n - u}^p \le 2^{p - 1} I_{\alpha/2} \ast (\abs{u}^p + \abs{u_n}^p)
 \le 2^p I_{\alpha/2} \ast \abs{u}^p + 2^{p - 1} \bigabs{I_{\alpha/2} \ast \abs{u_n}^p - I_{\alpha/2} \ast \abs{u}^p},
\]
we conclude by Lebesgue's dominated convergence theorem.
\end{proof}

\subsection{Further properties of Coulomb--Sobolev spaces}

\subsubsection{Uniform convexity and reflexivity}
In order to study the uniform convexity of the Coulomb--Sobolev space \(E^{\alpha, p} (\R^N)\),
we first establish uniform convexity of the Coulomb space \(Q^{\alpha, p} (\R^N)\).
The property will follow from the following inequalities.

\begin{proposition}[Clarkson--Boas--Koskela inequalities]
Let $N\in\N$, $\alpha \in (0,N)$.
If \(p, r, s \in (1, \infty)\) satisfy
\[
 r \ge \max\Big\{2p, \frac{p}{p - 1}, \frac{s}{s - 1}\Big\},
\]
then for every \(u \in Q^{\alpha, p} (\R^N)\),
\begin{multline*}
  \Bigl(\int_{\R^N} \abs{I_{\alpha/2} \ast \abs{u + v}^p}^2\Bigr)^\frac{r}{2p} +
  \Bigl(\int_{\R^N} \abs{I_{\alpha/2} \ast \abs{u - v}^p}^2\Bigr)^\frac{r}{2p}\\
  \le 2^{r(1 - \frac{1}{s})} \biggl(\Bigl(\int_{\R^N} \abs{I_{\alpha/2} \ast \abs{u}^p}^2\Bigr)^\frac{s}{2 p} +
  \Bigl(\int_{\R^N} \abs{I_{\alpha/2} \ast \abs{v}^p}^2\Bigr)^\frac{s}{2 p}\biggr)^\frac{r}{s}.
\end{multline*}
\end{proposition}
\begin{proof}
The Clarkson--Boas--Koskela inequalities in uniformly convex spaces \cite{MaligrandaSabourova2007}*{theorem 3.2} (see also \citelist{\cite{Koskela1979}*{theorem 3}\cite{Boas1940}*{theorem 2}\cite{BenedekPanzone1961}*{lemma 8.1}}) states that for every \(U, V \in L^{2p} (\R^N; L^p (\R^N))\),
\begin{multline*}
  \Bigl(\int_{\R^N} \Bigl(\int_{R^N} \abs{U (x, y)+ V (x, y)}^p\dif y\Bigr)^2\dif x\Bigr)^\frac{r}{2p} +
  \Bigl(\int_{\R^N} \Bigl(\int_{\R^N} \abs{U (x,y) - V (x, y)}^p \dif y\Bigr)^2\dif x\Bigr)^\frac{r}{2p}\\
  \le 2^{r(1 - \frac{1}{s})} \biggl(\Bigl(\int_{\R^N} \Bigl(\int_{\R^N} \abs{U (x, y)}^p\dif y\Bigr)^2\dif x\Bigr)^\frac{s}{2 p} +
  \Bigl(\int_{\R^N} \Bigl(\int_{\R^N} \abs{V (x, y)}^p\dif y\Bigr)^2\dif x\Bigr)^\frac{s}{2 p}\biggr)^\frac{r}{s}.
\end{multline*}
In particular, if we take \(U (x, y) = I_\alpha (x - y)^\frac{1}{p} u (y)\) and \(V (x, y) = I_\alpha (x - y)^\frac{1}{p} v (y)\), we reach the conclusion.
\end{proof}

In particular, if \(p \le \frac{3}{2}\) one can take \(r = \frac{p}{p - 1}\) and \(s = p\) and we obtain
\begin{multline}
\label{ineqClarksonLow}
  \Bigl(\int_{\R^N} \abs{I_{\alpha/2} \ast \abs{u + v}^p}^2\Bigr)^\frac{1}{2(p - 1)} +
  \Bigl(\int_{\R^N} \abs{I_{\alpha/2} \ast \abs{u - v}^p}^2\Bigr)^\frac{1}{2(p - 1)}\\
  \le 2 \biggl(\Bigl(\int_{\R^N} \abs{I_{\alpha/2} \ast \abs{u}^p}^2\Bigr)^\frac{1}{2} +
  \Bigl(\int_{\R^N} \abs{I_{\alpha/2} \ast \abs{v}^p}^2\Bigr)^\frac{1}{2}\biggr)^\frac{1}{p - 1}.
\end{multline}
whereas when \(p \ge \frac{3}{2}\) one can take \(r = 2 p\) and \(s = \frac{2 p}{2 p - 1}\) and we obtain
\begin{multline}
\label{ineqClarksonUpper}
  \int_{\R^N} \abs{I_{\alpha/2} \ast \abs{u + v}^p}^2 +
  \int_{\R^N} \abs{I_{\alpha/2} \ast \abs{u - v}^p}^2\\
  \le 2 \biggl(\Bigl(\int_{\R^N} \abs{I_{\alpha/2} \ast \abs{u}^p}^2\Bigr)^\frac{1}{2 p - 1} +
  \Bigl(\int_{\R^N} \abs{I_{\alpha/2} \ast \abs{v}^p}^2\Bigr)^\frac{1}{2 p - 1}\biggr)^{2 p - 1}.
\end{multline}

An important consequence of these considerations is the uniform convexity of the Coulomb spaces with $p>1$.

\begin{proposition}
Let $N\in\N$, $\alpha \in (0,N)$ and \(p \in (1, \infty)\).
Then \(Q^{\alpha, p} (\R^N)\) is uniformly convex.
\end{proposition}
\begin{proof}
Follows immediately from the inequalities \eqref{ineqClarksonLow} and \eqref{ineqClarksonUpper}.
\end{proof}

\begin{remark}
Alternatively, one can observe that the map \(\mathcal{L} : Q^{\alpha, p} (\R^N) \to L^{2p} (\R^N; L^p (\R^N))\) defined  by
\begin{equation}\label{e-isometry}
  (\mathcal{L} u) (x, y) = \bigl(I_{\alpha/2} (x - y)\bigr)^\frac{1}{p} u (y),
\end{equation}
is a linear isometry from \(Q^{\alpha, p}(\R^N)\) into \(L^{2p} (\R^N; L^p (\R^N))\).
Since the latter is uniformly convex \citelist{\cite{Boas1940}*{theorem 2}\cite{BenedekPanzone1961}*{\S 8}}, its linear subspace \(\mathcal{L} \big(Q^{\alpha, p}(\R^N)\big) \subset L^{2p} (\R^N; L^p (\R^N))\) is also uniformly convex and thus the space \(Q^{\alpha, p} (\R^N)\) is uniformly convex.
\end{remark}

\begin{remark}
When \(p \ge 2\), it is possible to deduce the uniform convexity directly from the classical \(L^p\) Clarkson inequality.
Since \(p \ge 2\), for every \(a, b \in \R\),
\begin{equation*}
\abs{a+b}^p+\abs{a-b}^p\leq 2^{p - 1}\big(\abs{a}^p+\abs{b}^p\big).
\end{equation*}
Hence, by linearity of the convolution and by positivity of the Riesz-kernel,
\[
  \int_{\R^N} \bigabs{I_{\frac{\alpha}{2}}*\abs{u+v}^p+I_{\frac{\alpha}{2}}*\abs{u-v}^p}^2
 \leq 2^{2 p - 2} \int_{\R^N} \bigabs{I_{\frac{\alpha}{2}}*\abs{u}^p+I_{\frac{\alpha}{2}}*\abs{v}^p}^2.
\]
Expanding the square and applying the Cauchy--Schwarz inequality then implies
\[
  \int_{\R^N} \bigabs{I_{\frac{\alpha}{2}}*\abs{u+v}^p}^2 + \int_{\R^N} \bigabs{I_{\frac{\alpha}{2}}*\abs{u-v}^p}^2
  \le 2^{2 p - 1} \Bigl(\int_{\R^N} \bigabs{I_{\frac{\alpha}{2}}*\abs{u}^p}^2+\int_{\R^N} \bigabs{I_{\frac{\alpha}{2}}*\abs{v}^p}^2\Bigr).
\]
This inequality was proved for \(p = 2\) by D.\thinspace{}Ruiz \cite{Ruiz-ARMA}*{p.\thinspace{} 355} and
follows directly from the stronger inequality \eqref{ineqClarksonUpper} by the discrete H\"older inequality.
\end{remark}

The uniform convexity of the Coulomb space $Q^{\alpha,p} (\R^N)$ implies uniform convexity of the Coulomb--Sobolev space $E^{\alpha,p}(\R^N)$.

\begin{proposition}
\label{propositionUniformlyConvex}
Let $N\in\N$, $\alpha \in (0,N)$ and \(p \in (1, \infty)\).
Then $E^{\alpha,p} (\R^N)$ is uniformly convex.
\end{proposition}

It follows from proposition~\ref{propositionUniformlyConvex} that the space \(E^{\alpha, p} (\R^N)\) is reflexive \cite{Milman}*{} (see also for example \citelist{\cite{Brezis2011}*{theorem 3.31}\cite{Yosida1980}*{theorem V.2.2}}).

\begin{proof}[Proof of proposition~\ref{propositionUniformlyConvex}]
We consider the map \(\mathcal{P} : E^{\alpha, p} (\R^N) \to Q^{\alpha, p} (\R^N) \times L^2 (\R^N; \R^N)\) defined by \(\mathcal{P} u = (u, D u)\) and we set
\[
  \norm{(u, U)}_{Q^{\alpha, p} (\R^N) \times L^2 (\R^N; \R^N)} = \sqrt{\norm{u}_{Q^{\alpha, p} (\R^N)}^2 + \norm{U}_{L^2 (\R^N;\R^N)}^2},
\]
so that \(\mathcal{P}\) in an isometry.
By a general result on the product of uniformly convex spaces \cite{Day1941}*{theorem 3} (see also \cite{Ruiz-ARMA}*{lemma 2.3}), the space \(Q^{\alpha, p} (\R^N) \times L^2 (\R^N; \R^N)\) is uniformly convex and thus \(E^{\alpha, p} (\R^N)\) is also uniformly convex.
\end{proof}

\subsubsection{Weak convergence in Coulomb--Sobolev spaces}

The above results allow to obtain a useful characterization the weak convergence in \(E^{\alpha, p} (\R^N)\), which we are going to use extensively in the sequel.

\begin{proposition}\label{weakL1loc}
Let $N\in\N$, $\alpha \in (0,N)$ and \(p \in (1, \infty)\).
A sequence \((u_n)_{n \in \N}\) in \(E^{\alpha, p}(\R^N)\) converges weakly to \(u \in E^{\alpha, p} (\R^N)\) if and only if it is bounded in \(E^{\alpha, p}(\R^N)\)
and converges to $u$ strongly in \(L^1_{\mathrm{loc}} (\R^N)\).
\end{proposition}
\begin{proof}
Since the space \(E^{\alpha, p} (\R^N)\) with $p>1$ is reflexive, its unit ball is weakly compact.
In view of the compactness property of proposition~\ref{propositionElementaryLocalWeakCompactness}, it suffices to show that if \((u_n)_{n \in \N}\) converges weakly to \(u\) in \(E^{\alpha, p} (\R^N)\) and converges strongly to \(\Tilde{u}\) in \(L^1_{\mathrm{loc}} (\R^N)\), then \(u=\Tilde{u}\) almost everywhere.

To check this, we take \(\varphi \in C_c (\R^N)\) and we define the linear functional \(\ell\),
\[
 \dualprod{\ell}{v} = \int_{\R^N} \varphi v.
\]
By proposition~\ref{propositionRuizBall}, \(\ell\) is well-defined and continuous on \(E^{\alpha, p} (\R^N)\). Therefore,
\[
 \lim_{n \to \infty} \int_{\R^N} \varphi u_n = \int_{\R^N} \varphi u.
\]
On the other hand, we conclude immediately that
\[
 \lim_{n \to \infty} \int_{\R^N} \varphi u_n = \int_{\R^N} \varphi \Tilde{u}.
\]
Since \(\varphi \in C_c (\R^N)\) is arbitrary, we conclude that \(u = \Tilde{u}\) almost everywhere.
\end{proof}

\subsubsection{A characterization of the dual space}
Since $Q^{\alpha, p} (\R^N)$ can be identified with a linear subspace of $L^{2p} (\R^N; L^p (\R^N))$
via isometry \eqref{e-isometry}, any linear functional $Q^{\alpha, p} (\R^N)$ can be extended to a linear functional on \(L^{2p}(\R^N; L^p (\R^N))\), whose elements can be represented by functions in \(L^\frac{2 p}{2 p - 1} (\R^N; L^\frac{p}{p - 1} (\R^N))\). This gives a representation
of linear functionals on $Q^{\alpha, p} (\R^N)$.

\begin{proposition}
\label{propositionDualCoulomb}
Let \(T\) be a distribution, then \(T \in \big(Q^{\alpha, p} (\R^N)\big)'\) if and only if there exists \(F : \R^N \times \R^N \to \R\) such that
\[
 \int_{\R^N} \Bigl(\int_{\R^N} \abs{F (x, y)}^\frac{p}{p - 1}\dif x \Bigr)^\frac{2 (p - 1)}{2p - 1}\dif y < \infty
\]
and for every \(\varphi \in C^\infty_c (\R^N)\),
\[
  \dualprod{T}{\varphi} = \int_{\R^N} \Bigl(\int_{\R^N} F (x, y) I_{\alpha/2} (x-y)^\frac{1}{p} \dif y\Bigr) \varphi (x)\dif x.
\]
\end{proposition}

This representation is still valid when \(p = 1\), it should then be understood that
\[
 \int_{\R^N} \Bigl(\esssup_{x \in \R^N} \abs{F (x, y)} \Bigr)^2 < \infty.
\]

The representation of proposition~\ref{propositionDualCoulomb} allows to represent the dual of \(E^{\alpha, p} (\R^N)\).

\begin{proposition}
Let \(T\) be a distribution, then \(T \in Q^{\alpha, p} (\R^N)'\) if and only if there exists \(F : \R^N \times \R^N \to \R\) and \(G : \R^N \to \R^N\) such that
\[
  \int_{\R^N} \abs{G}^2 + \int_{\R^N} \Bigl(\int_{\R^N} \abs{F (x, y)}^\frac{p}{p - 1}\dif x \Bigr)^\frac{2 (p - 1)}{2p - 1}\dif y < \infty.
\]
and for every \(\varphi \in C^\infty_c (\R^N)\),
\[
  \dualprod{T}{\varphi} = \int_{\R^N} G (x) \cdot \nabla \varphi (x) + \Bigl(\int_{\R^N} F (x, y) I_{\alpha/2} (x-y)^\frac{1}{p} \dif y\Bigr) \varphi (x)\dif x.
\]
\end{proposition}

\section{Coulomb--Sobolev embeddings}\label{sect-Embeddings}

In this section we prove theorem~\ref{theoremEmbedding} and study local compactness properties
of the embedding of Coulomb--Sobolev spaces into Lebesgue spaces.

\subsection{Continuous embedding into Lebesgue spaces and proof of theorem~\ref{theoremEmbedding}}

In order to prove theorem~\ref{theoremEmbedding} we first establish the critical Coulomb--Sobolev inequality \eqref{ineqCriticalCoulombSobolev}.

\begin{proposition}[Coulomb--Sobolev inequality]
\label{propositionRieszSobolevinterpolation}
Let $N\in\N$, $\alpha\in(0,N)$ and $p\ge 1$.
There exists \(C > 0\) such that for every \(u \in E^{\alpha, p}(\R^N)\),
\begin{equation}\label{e-critical-inequality}
\int_{\R^N} \abs{u}^{2\frac{2 p + \alpha}{2 + \alpha}}
\le C \Bigl(\int_{\R^N} \abs{Du}^2 \Bigr)^\frac{\alpha}{2 + \alpha}
\Bigl( \int_{\R^N} \abs{I_{\alpha/2} \ast \abs{u}^p}^2\Bigr)^\frac{2}{2 + \alpha}.
\end{equation}
\end{proposition}

\begin{proof}
For every \(x \in \R^N\) and \(\rho > 0\), we have \cite{Mazya2010}*{theorem 1.1.10/1} the estimate
\[
  \abs{u (x)} \le C \Bigl(\int_{B_\rho (x)} \frac{\abs{D u (y)}}{\abs{x - y}^{N - 1}}\dif y + \fint_{B_\rho (x)} \abs{u}\Bigr).
\]
Since for $N>1$
$$\frac{1}{\abs{x - y}^{N-1}}=(N-1)\int^\infty_0\frac{\chi_{\{y\,:\,\abs{x - y}<t\}}(y)}{t^N}\dif t,$$
by Fubini's theorem and by the definition of the maximal function
$$\mathcal M f(x):=\sup_{\rho>0}\fint_{B_\rho(x)}|f(x)| \dif x,$$
we have
\[
\begin{split}
  \int_{B_\rho (x)} \frac{\abs{D u (y)}}{\abs{x - y}^{N - 1}}\dif y
  &= (N-1)\int_{B_\rho (x)} \int_{0}^\infty \frac{\abs{D u (y)}\chi_{\{y\,:\,\abs{x - y}<t\}}(y)}{t^{N}}\dif t\dif y \\
  &= (N-1)\int_0^\rho \frac{1}{t^N}\Bigl(\int_{B_t (x)} \abs{D u}dy\Bigr)\dif t
  \le C' \rho \mathcal{M} \abs{D u} (x).
\end{split}
\]
The same estimate is obviously true when $N=1$.
On the other hand, by H\"older's inequality and by definition of the Riesz potential \(I_{\alpha/2}\),
\[
  \fint_{B_\rho (x)} \abs{u} \le \Bigl( \fint_{B_\rho (x)} \abs{u}^p \Bigr)^\frac{1}{p}
  \le \frac{C''}{\rho^\frac{\alpha}{2 p}} \bigl(I_{\alpha/2} \ast \abs{u}^p (x) \bigr)^\frac{1}{p}.
\]
Therefore for every \(x \in \R^N\),
\[
  \abs{u (x)} \le C'''\Bigl(\rho \mathcal{M} \abs{D u} (x) + \rho^{-\frac{\alpha}{2 p}} \bigl(I_{\alpha/2} \ast \abs{u}^p (x) \bigr)^\frac{1}{p}\Bigr).
\]
If we take
\[
  \rho = \Bigl(\frac{I_{\alpha/2} \ast \abs{u}^p (x)}{\mathcal{M} \abs{D u} (x)}\Bigr)^\frac{2 p}{2p + \alpha},
\]
then
\[
  \abs{u (x)} \le C'''' \bigl(\mathcal{M} \abs{D u} (x)\bigr)^\frac{\alpha}{2 p + \alpha}
  \bigl(I_{\alpha/2} \ast \abs{u}^p (x) \bigr)^\frac{2}{2 p + \alpha},
\]
and
\[
  \abs{u (x)}^{2\frac{2 p + \alpha}{2 + \alpha}}
  \le C'''''\bigl(\mathcal{M} \abs{D u} (x)\bigr)^\frac{2 \alpha}{2 + \alpha}
  \bigl(I_{\alpha/2} \ast \abs{u}^p (x) \bigr)^\frac{4}{2 + \alpha}.
\]
By integration and by H\"older's inequality,
\[
  \int_{\R^N} \abs{u}^{2 \frac{2 p + \alpha}{2 + \alpha}}
  \le C''''' \Bigl(\int_{\R^N} (\mathcal{M} \abs{D u})^2\Bigr)^\frac{2}{2 + \alpha}
  \Bigl(\int_{\R^N} \bigl(I_{\alpha/2} \ast \abs{u}^p\bigr)^2\Bigr)^\frac{\alpha}{2 + \alpha}.
\]
By the classical maximal function theorem \cite{Stein1970}*{theorem I.1}, we conclude that
\[
  \int_{\R^N} \abs{u}^{2 \frac{2 p + \alpha}{2 + \alpha}} \le
  C'''''' \Bigl(\int_{\R^N} \abs{D u}^2\Bigr)^\frac{\alpha}{2 + \alpha}
  \Bigl(\int_{\R^N} \bigl(I_{\alpha/2} \ast \abs{u}^p\bigr)^2\Bigr)^\frac{2}{2 + \alpha}.\qedhere
\]
\end{proof}

As a direct consequence of the critical Coulomb--Sobolev inequality \eqref{e-critical-inequality} we obtain a pointwise one--dimensional estimate.
We give a proof for completeness, which could also be obtained by invoking directly the classical homogeneous Gagliardo--Nirenberg interpolation inequality.

\begin{theorem}\label{DimensionOne}
Let $\alpha\in(0,1)$ and $p\ge 1$.
There exists \(C > 0\) such that for every \(u \in E^{\alpha, p}(\R)\),
\[
\sup_{x\in\R}\abs{u (x)}\le C \Bigl(\int_{\R} \abs{u'}^2\Bigr)^\frac{1 + \alpha}{2(1 + \alpha + p)} \Bigl(\int_{\R} \abs{I_{\alpha/2} \ast \abs{u}^p}^2 \Bigr)^\frac{1}{2(1 + \alpha + p)}.
\]
\end{theorem}
\begin{proof}
We prove the inequality for $u \in C^\infty_c(\R),$ the conclusion will follow by density. Let $t=2\frac{2p+\alpha}{2+\alpha},$ from the fundamental theorem of calculus and by the Cauchy-Schwarz inequality we have:
$$\abs{u (x)}^\frac{t+2}{2}\leq \frac{t+2}{2}\int^x _{-\infty}\abs{u}^\frac{t}{2}\abs{u'}\leq C\Bigl(\int_{\R} \abs{ u}^t\Bigr)^\frac{1}{2}\Bigl(\int_{\R} \abs{ u'}^2\Bigr)^\frac{1}{2}.$$
By Proposition~\ref{propositionRieszSobolevinterpolation}, the integral of \(\abs{u}^t\) can be estimated and the theorem follows.
\end{proof}

\begin{proof}[Proof of theorem \ref{theoremEmbedding}]
We first assume  that $N=1$. Since $q\geq t=2\frac{2p+\alpha}{2+\alpha},$ then $$\Big(\int_{\mathbb R}\abs{u}^q\Big)^{\frac{1}{q}}=\Big(\int_{\mathbb R}\abs{u}^{q-t+t}\Big)^{\frac{1}{q}}\leq \|u\|^{1-\frac{t}{q}}_\infty \Big(\int_{\mathbb R}\abs{u}^t\Big)^{\frac{1}{q}}.$$ The statement easily follows by estimating the two factors, respectively, by theorem~\ref{DimensionOne} and proposition~\ref{propositionRieszSobolevinterpolation}.

Assume that $N \ge 2$.
By the Gagliardo--Nirenberg interpolation inequality \citelist{\cite{Gagliardo1958}\cite{Nirenberg1959}} and by the Coulomb--Sobolev inequality (proposition~\ref{propositionRieszSobolevinterpolation}), we have
\[
\begin{split}
 \Bigl(\int_{\R^N} \abs{u}^q\Bigr)^\frac{1}{q}
 & \le C_1 \Bigl(\int_{\R^N} \abs{D u}^2\Bigr)^\frac{\mu}{2} \Bigl(\int_{\R^N} \abs{u}^{2 \frac{2 p + \alpha}{2 + \alpha}} \Bigr)^{\frac{1 - \mu}{2} \frac{2 + \alpha}{2 p + \alpha}}\\
 & \le C_2 \Bigl(\int_{\R^N} \abs{D u}^2\Bigr)^\frac{\mu 2 p + \alpha}{2(2p + \alpha)} \Bigl(\int_{\R^N} \bigabs{I_{\alpha/2} \ast \abs{u}^p}^2\Bigr)^{\frac{1 - \mu}{2 p + \alpha}},
\end{split}
\]
with $p,q\in[1,+\infty)$, \(\mu \in [0, 1]\) and
\[
 \frac{1}{q} = \mu \Bigl(\frac{1}{2} - \frac{1}{N}\Bigr) + (1 - \mu)  \frac{2 + \alpha}{2(2 p + \alpha)}.
\]
Then \eqref{Coulomb-Sobolev-estimate} follows with
$\theta\in\big[\frac{\alpha}{2 p + \alpha},1\big]$ and
\[
 \frac{1}{q} = \theta \Bigl(\frac{1}{2} - \frac{1}{N}\Bigr) + (1 - \theta)  \frac{N+\alpha}{2Np}.\qedhere
\]
\end{proof}

\begin{proposition}[Necessary condition for the embeddings]
Let $N\in\N$, $\alpha\in(0,N)$ and $p\ge 1$.
If for every $u\in C^\infty_c(\R^N)$,
\begin{equation}\label{e-embedd}
  \Bigl(\int_{\R^N} \abs{u}^q\Bigr)^\frac{1}{q}
  \le C\biggl(\int_{\R^N} \abs{D u}^2 +\Bigl( \int_{\R^N} \bigl(I_{\alpha/2} \ast \abs{u}^p\bigr)^2\Bigr)^\frac{1}{p}\biggr)^\frac{1}{2},
\end{equation}
then assumption \eqref{Q} holds.
\end{proposition}

\begin{proof}
Let \(u \in C^\infty_c (\R^N) \setminus \{0\}\).
For \(\lambda > 0\), define the function \(u_\lambda \in C^\infty_c (\R^N)\) by
\[
  u_\lambda (x) = u (x/\lambda).
\]
We compute
\begin{gather*}
  \int_{\R^N} \abs{u_\lambda}^q = \lambda^N \int_{\R^N} \abs{u}^q,\\
  \int_{\R^N} \abs{D u_\lambda}^2 = \lambda^{N - 2} \int_{\R^N} \abs{D u}^2,\\
  \int_{\R^N} \bigl(I_{\alpha/2} \ast \abs{u_\lambda}^p\bigr)^2 = \lambda^{N + \alpha} \int_{\R^N} \bigl(I_{\alpha/2} \ast \abs{u}^p\bigr)^2.
\end{gather*}
By \eqref{e-embedd}, we have for every \(\lambda > 0\),
\[
  \lambda^\frac{N}{q} \le c\bigl(\lambda^\frac{N - 2}{2} + \lambda^\frac{N + \alpha}{2 p}\bigr).
\]
We deduce therefrom that
\[
  \min \Bigl\{\frac{N - 2}{2}, \frac{N + \alpha}{2 p} \Bigr\} \le \frac{N}{q} \le \max \Bigl\{\frac{N - 2}{2}, \frac{N + \alpha}{2 p} \Bigr\},
\]
which is weaker then \eqref{Q}.

Assume that $N + \alpha \neq p (N - 2)$.
Optimizing with respect to $\lambda>0$ the quotient
\begin{equation}\label{e-embedd-quotient}
   R(\lambda):=\frac{\int_{\R^N} \abs{D u_\lambda}^2 +\Bigl( \int_{\R^N} \bigl(I_{\alpha/2} \ast \abs{u_\lambda}^p\bigr)^2\Bigr)^\frac{1}{p}}{\Bigl(\int_{\R^N} \abs{u_\lambda}^q\Bigr)^\frac{2}{q}},
\end{equation}
we see that $R(\lambda)$ attains an optimal value at
\begin{equation}\label{quotient-optimal}
\lambda_*=C_*\left(\frac{\Bigl( \int_{\R^N} \bigl(I_{\alpha/2} \ast \abs{u}^p\bigr)^2\Bigr)^\frac{1}{p}}{\int_{\R^N} \abs{D u}^2}\right)^\frac{p}{(N-2)p-(N+\alpha)},
\end{equation}
where $C_*=C_*(N,\alpha,p,q)$. This leads to the estimate
\begin{equation}\label{e-optimize-lambda}
  \Bigl(\int_{\R^N} \abs{u}^q\Bigr)^{2}
  \le C\Bigl(\int_{\R^N} \abs{D u}^2\Bigr)^{\frac{q(N + \alpha) - 2 pN}{(N + \alpha) - p (N - 2)}} \Bigl( \int_{\R^N} \bigl(I_{\alpha/2} \ast \abs{u}^p\bigr)^2\Bigr)^{\frac{2 N - q(N - 2)}{(N + \alpha) - p (N - 2)} }.
\end{equation}
Given a vector \(a \in \R^N\setminus\{0\}\) and \(n \in \N\), define the function \(u_{n, a} \in C^\infty_c (\R^N)\) by
\[
  u_{n, a} (x) = \sum_{i = 1}^n u (x + i a).
\]
Then
\begin{gather*}
  \lim_{\abs{a} \to \infty} \int_{\R^N} \abs{D u_{n, a}}^2
  = n \int_{\R^N} \abs{D u}^2,\\
  \lim_{\abs{a} \to \infty} \int_{\R^N} \abs{u_{n, a}}^q
  = n \int_{\R^N} \abs{u}^q,\\
  \lim_{\abs{a} \to \infty} \int_{\R^N} \abs{I_{\alpha/2} \ast \abs{u_{n, a}}^p}^2
  = n \int_{\R^N} \abs{I_{\alpha/2} \ast \abs{u}^p}^2.
\end{gather*}
Using the diagonal argument, from \eqref{e-optimize-lambda} we deduce that for all sufficiently large \(n \in \N\) must hold
\[
  n^2 \le C' n^{\frac{q(N + \alpha) - 2 pN}{(N + \alpha) - p (N - 2)}} n^{\frac{2 N - q(N - 2)}{(N + \alpha) - p (N - 2)} },
\]
which implies \eqref{Q}.

Next assume that $N + \alpha = p (N - 2)$.
For \(\lambda > 0\), consider the rescaling
\[
  u_\lambda (x) = \lambda^{-\frac{N-2}{2}}u (x/\lambda).
\]
Substituting to \eqref{e-embedd} we obtain
$\lambda^{-\frac{N-2}{2}+\frac{N}{q}}\le C$,
which requires $q=\frac{2N}{N-2}$.
\end{proof}

\begin{remark}
In the critical case \(q = 2 \frac{2 p  + \alpha}{2 + \alpha}\) a rescaling of the sequence $(u_{n,a})_{n\in\N}$ is bounded but not compact up to translations in $E^{\alpha,p}(\R^N)$.
In fact such a sequence is vanishing in the sense of P.-L.\thinspace{}Lions \cite{Lions1984CC1}.
\end{remark}

\subsection{Local compactness in Lebesgue spaces}
Since the functional space \(E^{\alpha, p} (\R^N)\) is invariant under translations, the embedding \(E^{\alpha, p} (\R^N)\) into \(L^q (\R^N)\) is never compact.

By theorem~\ref{theoremEmbedding} and proposition~\ref{propositionRuizBall}, we have continuous embedding $E^{\alpha,p}(\R^N)\subset L^q_{loc}(\R^N)$ for all $q\geq 1$  such that
\[
 \frac{1}{q} \ge \min \Bigl\{\frac{1}{2} - \frac{1}{N},\frac{1}{p},\frac{2 + \alpha}{2 (2 p + \alpha)}\Bigr\}.
\]
We show that this embedding is compact if and only if the inequality is strict.

\begin{proposition}[Local compactness in $L^p$]
\label{propositionAdvancedLocalWeakCompactness}
Let \(N \in \N\), \(\alpha \in (0, N)\) and \(p\ge 1\).
The embedding $E^{\alpha,p}(\R^N)\subset L^q_{loc}(\R^N)$ is compact if and only if
\[
 \frac{1}{q} > \min \Bigl\{\frac{1}{2} - \frac{1}{N},\frac{1}{p},\frac{2 + \alpha}{2 (2 p + \alpha)}\Bigr\}.
\]
\end{proposition}

In particular, the embedding \(E^{\alpha, p} (\R^N) \subset L^p_{\mathrm{loc}} (\R^N)\) is compact if and only if
\begin{equation}\label{e-pcomp}
 \frac{1}{p} > \Big(\frac{1}{2} - \frac{1}{\alpha}\Big)_+.
\end{equation}
Indeed, \eqref{e-pcomp} is equivalent to $\frac{1}{p} > \frac{2 + \alpha}{2 (2 p + \alpha)}$.
Then by proposition \ref{propositionAdvancedLocalWeakCompactness} we should have either $\frac{1}{p}>\frac{1}{2}-\frac{1}{N}$, or $\frac{1}{p}>\frac{1}{p}$,
or $\frac{1}{p}>\frac{1}{2}-\frac{1}{\alpha}$.
The second condition can never be satisfied, and the third is weaker than the first one since $\alpha<N$.

We shall see later in lemma~\ref{lemmaCantor} that \eqref{e-pcomp} is equivalent to the weak continuity of the map \(u \in E^{\alpha, p} (\R^N) \mapsto I_{\alpha/2} \ast \abs{u}^p \in L^2 (\R^N)\).

\begin{proof}[Proof of proposition~\ref{propositionAdvancedLocalWeakCompactness}]
If the sequence of functions \((u_n)_{n \in \N}\) is bounded in the Coulomb--Sobolev space \(E^{\alpha,p} (\R^N)\), then passing to if necessary to a subsequence by proposition~\ref{propositionElementaryLocalWeakCompactness}, \((u_n)_{n \in \N}\) converges strongly to a function \(u \in L^1_{\mathrm{loc}} (\R^N)\). By theorem~\ref{theoremEmbedding} and proposition~\ref{propositionRuizBall}, the sequence \((u_n)_{n \in \N}\) is bounded in \(L^{\Bar{q}}_{\mathrm{loc}} (\R^N)\) for every \(\Bar{q}\) such that
\[
 \frac{1}{\Bar{q}} \ge \min \Bigl\{\frac{1}{2} - \frac{1}{N},\frac{1}{p},\frac{2 + \alpha}{2 (2 p + \alpha)} \Bigr\}.
\]
In particular, we can take \(\Bar{q} > q\). Then by the classical H\"older inequality we have for every compact set \(K \subset \R^N\),
\[
 \int_{K} \abs{u_n - u}^q \le \Bigl(\int_{K} \abs{u_n - u} \Bigr)^\frac{\Bar{q} - q}{\Bar{q} - 1}
 \Bigl(\int_{K} \abs{u_n - u}^{\Bar{q}} \Bigr)^\frac{q - 1}{\Bar{q} - 1}.
\]
Therefore, the sequence \((u_n)_{n \in \N}\) converges to \(u\) in \(L^q_{\mathrm{loc}} (\R^N)\).

If \(p \le \frac{N + \alpha}{N - 2}\), then
\[
  \min \Bigl\{\frac{1}{2} - \frac{1}{N},\frac{1}{p},\frac{2 + \alpha}{2 (2 p + \alpha)} \Bigr\} = \frac{1}{2} - \frac{1}{N}.
\]
In this case the lack of compact embedding from $E^{\alpha,p}(\R^N)$ into $L^\frac{2N}{N-2}_\loc(\R^N)$
can be seen by considering, for a given function $u\in C^\infty_c(\R^N)\setminus\{0\}$,
a sequence of functions \((u_n)_{n \in \N}\) defined by
$$u_n (x) = n^{\frac{N-2}{2}} u (n x).$$
Then  $\supp(u_n)$ is uniformly bounded in $\R^N$, $\norm{u_n}_{E^{\alpha,p}}=\|u\|_{E^{\alpha,p}}\big(1+o(1)\big)$ and $\norm{u_n}_{L^1} \to 0$ as $n\to\infty$, but $\norm{u_n}_{L^\frac{2N}{N-2}}=\|u\|_{L^\frac{2N}{N-2}}$.

If \(p > \frac{N + \alpha}{N - 2}\), then
the lack of compact embedding from $E^{\alpha,p}(\R^N)$ into $L^q_\loc(\R^N)$
is a consequence of lemma~\ref{lemmaCoulombSobolev-noncompact} below.
\end{proof}

\begin{lemma}
\label{lemmaCoulombSobolev-noncompact}
Let \(N \ge 3\), \(\alpha \in (0, N)\) and \(p > \frac{N + \alpha}{N - 2}\). There exists a sequence of functions \((u_n)_{n \in \N}\) in \(C^\infty_c (\R^N)\) such that \(\supp(u_n)\) is uniformly bounded, \((u_n)_{n \in \N}\) is bounded in \(E^{\alpha, p}(\R^N)\), \(u_n \to 0\) almost everywhere and, if \(\alpha \le 2\) or \(p \le \frac{2 \alpha}{\alpha - 2}\),
\[
 \liminf_{n \to \infty} \int_{\R^N} \abs{u_n}^{2 \frac{2 p + \alpha}{2 + \alpha}} > 0,
\]
whereas if \(\alpha > 2\) and \(p > \frac{2 \alpha}{\alpha - 2}\),
\[
 \liminf_{n \to \infty} \int_{\R^N} \abs{u_n}^{p} > 0.
 \]
\end{lemma}

\begin{proof}
For $R>0$, let \(Q_R = [-R, R]^n\) be the cube in $\R^N$. Take a nonnegative \(w_0 \in C^\infty_c (Q_1) \setminus \{0\}\).
Fix an integer $d\in\{1,\dots,N\}$ such that $d>N-\alpha$.
For \(n \in \N_*:=\N\cup\{0\}\) and $\rho_n>N$ which will be specified later,
define the sequence of functions $(w_n)_{n\in\N_*}\subset C^\infty_c (\R^N)$ by
\[
  w_{n} (x) = \sum_{a \in \rho_n\{-n, \dotsc, n\}^d} w_0(x - a).
\]
We compute
\begin{gather*}
  \int_{\R^N} \abs{D w_{n}}^2
  = (2n+1)^d \int_{\R^N} \abs{D w_0}^2,\\
  \int_{\R^N} \abs{w_n}^q
  = (2n+1)^d \int_{\R^N} \abs{w_0}^q,
\end{gather*}
while we estimate the Coulomb term as
\begin{multline*}
  \int_{\R^N} \abs{I_{\alpha/2} \ast \abs{w_n}^p}^2\\
  \le (2n+1)^d \int_{\R^N} \abs{I_{\alpha/2} \ast \abs{w_0}^p}^2+
  \sum_{\substack{a\neq b\\ a \in \rho_n\{-n, \dotsc, n\}^d\\b \in \rho_n\{-n, \dotsc, n\}^d}}
  \int_{Q_1}\int_{Q_1} I_\alpha(x-y+a-b)|w_0(x)|^p|w_0(y)|^p\dif x \dif y\\
  \le (2n+1)^d \int_{\R^N} \abs{I_{\alpha/2} \ast \abs{w_0}^p}^2+
  \sum_{\substack{a\neq b\\ a \in \rho_n\{-n, \dotsc, n\}^d\\b \in \rho_n\{-n, \dotsc, n\}^d}}
  \frac{C}{|a-b|^{N-\alpha}}\Big(\int_{\R^N}\abs{w_0}^p\Big)^2.
\end{multline*}
Since $N-\alpha<d$,
\begin{multline*}
\sum_{\substack{a\neq b\\ a \in \rho_n\{-n, \dotsc, n\}^d\\b \in \rho_n\{-n, \dotsc, n\}^d}}\frac{1}{|a-b|^{N-\alpha}}=
\frac{1}{\rho_n^{N-\alpha}}\sum_{\substack{a\neq b\\ a \in \{-n, \dotsc, n\}^d\\b \in \{-n, \dotsc, n\}^d}}
\frac{1}{|a-b|^{N-\alpha}}\\
\le \frac{C}{\rho_n^{N-\alpha}}\int_{Q^d_n}\int_{Q^d_n}\frac{1}{|a-b|^{N-\alpha}}\dif a \dif b=
\frac{Cn^{2d-(N-\alpha)}}{\rho_n^{N-\alpha}}\int_{Q^d_1}\int_{Q^d_1}\frac{1}{|a-b|^{N-\alpha}}\dif a \dif b,
\end{multline*}
where $Q^d_n = [-n, n]^d$ is the cube in $\R^d$.
Therefore, we estimate the Coulomb term by
\begin{equation*}
  \int_{\R^N} \abs{I_{\alpha/2} \ast \abs{w_n}^p}^2
  \le (2n+1)^d \int_{\R^N} \abs{I_{\alpha/2} \ast \abs{w_0}^p}^2+
  \frac{Cn^{2d-(N-\alpha)}}{\rho_n^{N-\alpha}}\Big(\int_{\R^N}\abs{w_0}^p\Big)^2.
\end{equation*}
Set
$$\rho_n=n^{\frac{d}{N-\alpha}-1},$$
and define for $\lambda_n>0$ the rescaled sequence
$$u_n(x)=w_n\Big(\frac{x}{\lambda_n}\Big).$$
Then, as $n\to \infty$,
\begin{gather*}
  \int_{\R^N} \abs{D u_n}^2
  =\lambda_n^{N-2}(2n)^d\bigl(1+o(1)\bigr)\int_{\R^N} \abs{D w_0}^2,\\
    \int_{\R^N} \bigl(I_{\alpha/2} \ast \abs{u_n^p}\bigr)^2
  = \lambda_n^{N+\alpha}(2n)^d\bigl(1+O(1)\bigr) \int_{\R^N} \bigl(I_{\alpha/2} \ast \abs{w_0}^p\bigr)^2,\\
  \int_{\R^N} \abs{u_n}^q
  =\lambda_n^{N}(2n)^d\bigl(1+o(1)\bigr)\int_{\R^N} \abs{w_0}^q.
\end{gather*}

Assume that $\alpha\le 2$ or $p\le\frac{2\alpha}{\alpha-2}$ and
let
$$q=2\frac{\alpha+2p}{\alpha+2}.$$
Taking into account \eqref{quotient-optimal}, define
\begin{equation*}
\lambda_n=n^{-\frac{d(p-1)}{(N-2)p-(N+\alpha)}}.
\end{equation*}
Then we compute
$$\lim_{n\to\infty}\frac{\norm{u_n}_{E^{\alpha,p}}}{\norm{u_n}_{L^q}}= \frac{\|w_0\|_{E^{\alpha,p}}}{\|w_0\|_{L^q}}>0.$$
Note that $\mathrm{Supp}(u_n)\subseteq Q_{R_n}$, where
$$R_n=\lambda_n(n\rho_n+1)=n^{-\frac{d(p-1)}{(N-2)p-(N+\alpha)}}n^{\frac{d}{N-\alpha}}(1+o(1)).$$
Therefore, we compute that
$$R_n\to 1\quad\text{if $\alpha>2$ and $p=\frac{2\alpha}{\alpha-2}$,}$$
and
$$R_n\to 0\quad\text{if $\alpha\le 2$ or $\frac{N+\alpha}{N-2}<p<\frac{2\alpha}{\alpha-2}$.}$$
Moreover, it is clear that
$$
 \abs{\{ x \in \R^N : u_n(x) \ne 0\}} \le C\lambda_n^N n^d\to 0,$$
so in both cases $u_n\to 0$ almost everywhere in $\R^N$.
\medskip

Next assume that  $\alpha>2$ and $p>\frac{2\alpha}{\alpha-2}$. We set
\begin{align*}
q&=p &
&\text{ and }&
\lambda_n&:=n^{-\frac{d}{N-\alpha}}.
\end{align*}
Since $p>\frac{2\alpha}{\alpha-2}$, we compute that
$$\lim_{n\to\infty}\frac{\norm{u_n}_{E^{\alpha,p}}}{\norm{u_n}_{L^p}}= \frac{\|I_{\alpha/2}*|w_0|^p\|_{L^2}^{1/p}}{\|w_0\|_{L^p}}>0.$$
Note that $\mathrm{Supp}(u_n)\subseteq Q_{R_n}$, where
$$R_n=\lambda_n(n\rho_n+1)=n^{-\frac{d}{N-\alpha}}n^{\frac{d}{N-\alpha}}\bigl(1+o(1)\bigr)=1+o(1).$$
Moreover, it is clear that $u_n\to 0$ almost everywhere in $\R^N$.
\end{proof}

\subsection{Weighted Coulomb estimates}

The goal of this section is to improve the Coulomb estimate on balls of proposition~\ref{propositionRuizBall} to global weighted estimates.
By homogeneity considerations, a natural candidate would be
\begin{equation}
\label{eqImpossibleRuiz}
 \int_{\R^N} \frac{\abs{u (x)}^p}{\abs{x}^{\frac{N - \alpha}{2}}} \dif x
 \le \Bigl(\int_{\R^N} \abs{I_{\alpha/2} \ast \abs{u}^p}^2\Bigr)^\frac{1}{2}.
\end{equation}
However, as already observed by Ruiz \cite{Ruiz-ARMA}*{section 3}, this estimate cannot hold.

\begin{proposition}
\label{propositionRuizCounterexample}
Let \(N \in \N\), \(\alpha \in (0, N)\), \(p \ge 1\) and \(W : \R^N \to \R\).
If for every \(u\) in \(E^{\alpha, p} (\R^N)\)
\[
\int_{\R^N} W \abs{u}^p
 \le \Bigl(\int_{\R^N} \abs{I_{\alpha/2} \ast \abs{u}^p}^2\Bigr)^\frac{1}{2},
\]
then for any \(\delta > \frac{1}{2}\),
\[
  \int_{\R^N} \frac{W (x)}{\abs{x}^\frac{N + \alpha}{2} (1 + \abs{\log \abs{x}\,})^\frac{1}{2}
  \bigl(1 + \log (1 + \abs{\log \abs{x}\,})\bigr)^\delta}\dif x < \infty.
\]
\end{proposition}

In particular, since
\[
 \int_{\R^N} \frac{1}{\abs{x}^N (1 + \abs{\log \abs{x}\,})^\frac{1}{2}
    (1 + \log (1 + \abs{\log \abs{x}\,}))^\delta}\dif x = \infty
\]
for every \(\delta \in \R\), the estimate \eqref{eqImpossibleRuiz} cannot hold.

\begin{proof}[Proof of proposition~\ref{propositionRuizCounterexample}]
Given \(\delta > 0\), we define the function \(u : \R^N \to \R\) for each \(x \in \R^N\) by
\[
 u (x) = \frac{1}{(\log \abs{x})^\frac{1}{2 p} (\log (\log \abs{x}))^{\frac{\delta}{p}}\abs{x}^\frac{N + \alpha}{2 p}} \chi_{\R^N \setminus B_3} (x).
\]
Observe that for every \(x \in \R^N\)
\[
 \abs{u (x)}^p
 \le C  \int_3^\infty \frac{1}{\abs{x}^\frac{N + \alpha}{2} \rho (\log \rho)^\frac{3}{2} (\log \log \rho)^{\delta} } \chi_{\R^N \setminus B_\rho} (x) \dif \rho.
\]
Since for every \(x \in \R^N\)
\[
 \bigl(I_{\alpha/2} \ast (I_{(N - \alpha)/2} \chi_{\R^N \setminus B_\rho})\bigr) (x)
 \le \frac{C'}{(\abs{x} + \rho)^\frac{N}{2}},
\]
we have
\[
 I_{\alpha/2} \ast \abs{u (x)}^p
 \le C C' \int_3^\infty \frac{1}{(\abs{x} + \rho)^\frac{N}{2} \rho (\log \rho)^{\frac{3}{2}} (\log \log \rho)^\delta } \dif \rho
 \le C'' \frac{1}{(\log \abs{x})^\frac{1}{2} (\log \log \abs{x})^\delta \abs{x}^\frac{N}{2}}.
\]
Therefore \(I_{\alpha/2} \ast \abs{u}^p \in L^2 (\R^N)\) as soon as \(\delta > \frac{1}{2}\).
This implies that
\[
  \int_{\R^N \setminus B_3} \frac{W (x)}{\abs{x}^\frac{N + \alpha}{2} (\log \abs{x})(\log \log \abs{x})^{\delta}}\dif x < \infty.
\]
To obtain the condition around the origin, we define for \(\delta > 0\),
\[
 u (x) = \frac{1}{(\log 1/\abs{x})^\frac{1}{2 p} (\log \log 1/\abs{x})^{\frac{\delta}{p}}\abs{x}^\frac{N + \alpha}{2 p}} \chi_{B_{1/3}} (x).\qedhere
\]
\end{proof}

Although \eqref{eqImpossibleRuiz} does not hold, it is still possible to prove a scaling invariant inequality that implies the local estimate on balls of proposition~\ref{propositionRuizBall}.

\begin{proposition}
\label{propositionRuizAverage}
Let \(N \in \N\), \(\alpha \in (0, N)\), \(p \ge 1\).
For every \(a \in \R^N\),
\[
  \int_0^\infty \Bigl(\fint_{B_\rho (a)} \abs{u}^p \Bigr)^2 \rho^{\alpha + N - 1}\dif \rho
  \le C \int_{\R^N} \bigabs{I_{\alpha/2} \ast \abs{u}^p}^2.
\]
\end{proposition}
It is clear that proposition~\ref{propositionRuizAverage} implies proposition~\ref{propositionRuizBall}.

\begin{proof}[Proof of proposition~\ref{propositionRuizAverage}]
For every \(x \in B_\rho \setminus B_{\rho/2}\), we have
\[
   \rho^\frac{\alpha}{2} \fint_{B_\rho} \abs{u}^p
   \le C \bigl(I_{\alpha/2} \ast \abs{u}^p\bigr) (x)
\]
Therefore, by integrating over \(B_\rho \setminus B_{\rho/2}\),
\[
   \rho^{N + \alpha} \Bigl(\fint_{B_\rho} \abs{u}^p\Bigr)^2
    \le C' \int_{B_\rho \setminus B_{\rho/2}} \bigabs{I_{\alpha/2} \ast \abs{u}^p}^2.
\]
Hence by integration and by Fubini's theorem
\[
  \int_0^\infty \Bigl(\fint_{B_\rho} \abs{u}^p\Bigr)^2 \rho^{\alpha + N - 1}\dif \rho
  \le C'\int_0^\infty \Bigl(\int_{B_\rho \setminus B_{\rho/2}} \bigabs{I_{\alpha/2} \ast \abs{u}^p}^2\Bigr) \frac{\dif \rho}{\rho}
  = C' \ln 2 \int_{\R^N} \bigabs{I_{\alpha/2} \ast \abs{u}^p}^2.\qedhere
\]
\end{proof}

More generally, we can deduce from proposition~\ref{propositionRuizAverage} families of weighted estimates.

\begin{proposition}
\label{propositionRuizWeight}
Let \(N \in \N\), \(\alpha \in (0, N)\), \(p \ge 1\) and \(w : (0,\infty) \to \R\).
If
$$\int_0^\infty \rho^{1 + N -\alpha} w (\rho)^2 \dif \rho < \infty$$
and $W(\rho)=\int_\rho^\infty w (r)\dif r$,
then
\[
\Bigl(\int_{\R^N} \abs{u (x)}^{p} W (\abs{x})\dif x\Bigr)^2
\le C
\Bigl(\int_0^\infty \abs{w (\rho)}^2 \rho^{1 + N - \alpha}\dif \rho\Bigr)\Bigl( \int_0^\infty \Bigl(\fint_{B_\rho} \abs{u}^{p}\Bigr)^2 \rho^{\alpha + N - 1} \dif \rho\Bigr).
\]
\end{proposition}

Observe that by the Cauchy--Schwarz inequality, we have
\[
\begin{split}
 W (r) \le \int_r^\infty w (\rho) \dif \rho &\le \Bigl(\int_r^\infty w (\rho) \rho^{1 + N -\alpha}\Bigr)^\frac{1}{2} \Bigl(\int_r^\infty \frac{\dif \rho}{\rho^{1 + N - \alpha}} \Bigr)^\frac{1}{2}\\
 &\le \frac{1}{\sqrt{N - \alpha} \abs{x}^\frac{N - \alpha}{2}} \Bigl( \int_0^\infty w (\rho) \rho^{1 + N - \alpha} \dif \rho \Bigr)^\frac{1}{2}.
\end{split}
\]
In particular, if we choose
\[
  W (\rho) = \frac{1}{\rho^\frac{N - \alpha}{2} (1 + \abs{\log \rho})^\gamma},
\]
with \(\gamma < \frac{1}{2}\), then
\begin{equation}\label{weighted}
  \Bigl(\int_{\R^N} \frac{\abs{u (x)}^p}{\abs{x}^\frac{N- \alpha}{2} (1 + \bigabs{\log \abs{x}})^\gamma}\dif x \Bigr)^2
  \le C \int_{\R^N} \bigabs{I_{\alpha/2} \ast \abs{u}^p}^2,
\end{equation}
and we recover the inequality obtained by Ruiz for \(p = 2\) \cite{Ruiz-ARMA}*{theorem 3.1}.
By proposition~\ref{propositionRuizCounterexample}, the restriction \(\gamma < \frac{1}{2}\) is optimal for \eqref{weighted} to hold. This completes the study of Ruiz who has showed that the inequality does not hold when \(p = 2\) and \(\gamma < \frac{1}{2} - \frac{1}{N}\) \cite{Ruiz-ARMA}*{remark 3.3}.

\begin{proof}[Proof of proposition~\ref{propositionRuizWeight}]
Integrating by parts and using the Cauchy--Schwarz inequality, we obtain
\[
\begin{split}
  \Bigl(\int_{\R^N} \abs{u (x)}^{p} W (\abs{x})\dif x\Bigr)^2
  &=\Bigl(\int_{\R^N} \abs{u (x)}^{p} \int_{\abs{x}}^\infty w (\rho)\dif r \dif x\Bigr)^2\\
  &=C'\Bigl(\int_0^\infty w (\rho) \rho^N \fint_{B_\rho} \abs{u}^{p} \dif \rho\Bigr)^2\\
  &\le C \Bigl(\int_0^\infty \abs{w (\rho)}^2 \rho^{1 + N - \alpha}\dif \rho\Bigr)\Bigl( \int_0^\infty \Bigl(\fint_{B_\rho} \abs{u}^{p}\Bigr)^2 \rho^{\alpha + N - 1} \dif \rho\Bigr).\qedhere
\end{split}
\]
\end{proof}

In the sequel, we shall use the following particular case which gives a good practical substitute to \eqref{eqImpossibleRuiz}.

\begin{proposition}
\label{propositionRuizPowerExterior}
Let \(N \in \N\), \(\alpha \in (0, N)\) and \(p \ge 1\).
If \(\beta > \frac{N - \alpha}{2}\), then for every \(u \in Q^{\alpha, p} (\R^N)\)
\[
  \int_{\R^N \setminus B_R} \frac{\abs{u (x)}^p}{\abs{x}^{\beta}}\dif x
  \le \frac{C}{R^{\beta - \frac{N - \alpha}{2}}}\Bigl(\int_{\R^N} \bigabs{I_{\alpha/2} \ast \abs{u}^p}^2\Bigr)^\frac{1}{2}.
\]
If \(\beta < \frac{N - \alpha}{2}\), then for every \(u \in Q^{\alpha, p} (\R^N)\)
\[
  \int_{B_R} \frac{\abs{u (x)}^p}{\abs{x}^{\beta}}\dif x
  \le CR^{\frac{N - \alpha}{2} - \beta} \Bigl(\int_{\R^N} \bigabs{I_{\alpha/2} \ast \abs{u}^p}^2\Bigr)^\frac{1}{2}.
\]
\end{proposition}
In view of proposition~\ref{propositionRuizCounterexample}, the restrictions on \(\beta\) are optimal.
\begin{proof}[Proof of proposition~\ref{propositionRuizPowerExterior}]
For the first inequality we apply proposition~\ref{propositionRuizWeight} with the function \(w : (0, \infty) \to [0, \infty)\) defined for \(\rho \in (0, \infty)\) by
\[
w (\rho) = \begin{cases}
              \frac{\beta}{\abs{x}^{\beta + 1}}  & \text{if \(\rho \ge R\)},\\
              0 & \text{if \(\rho \le R\)}.
           \end{cases}
\]
The proof of the second inequality is similar.
\end{proof}

\section{Nonlocal Brezis--Lieb lemma}\label{sect-Brezis--Lieb}
\settocdepth{section}

\subsection{General nonlocal Brezis--Lieb lemma}

The main result of this section is the following nonlocal Brezis--Lieb property.

\begin{proposition}[Nonlocal Brezis--Lieb lemma]
\label{propositionBrezisLieb}
Let \(N \in \N\), \(\alpha \in (0, N)\) and \(p \ge 1\).
Assume that \((u_n)_{n \in \N}\) is a sequence of measurable functions from \(\R^N\) to \(\R\)
that converges to \(u : \R^N \to \R\) almost everywhere.
If the sequence \((I_{\alpha/2} \ast \abs{u_n}^p)_{n \in \N}\) is bounded in \(L^2 (\R^N)\), then
\[
  \lim_{n \to \infty}
  \int_{\R^N} \Bigabs{\bigabs{I_{\alpha/2} \ast \abs{u_n}^p}^2 - \bigabs{I_{\alpha/2} \ast (\abs{u_n - u}^p + \abs{u}^p)}^2}
  = 0.
\]
\end{proposition}
In particular,
\begin{multline}
\label{ineqBrezisLieb}
  \liminf_{n \to \infty}
  \int_{\R^N} \bigabs{I_{\alpha/2} \ast \abs{u_n}^p}^2 - \bigabs{I_{\alpha/2} \ast \abs{u_n - u}^p}^2\\
  = \int_{\R^N} \bigabs{I_{\alpha/2} \ast \abs{u}^p}^2 + 2 \liminf_{n \to \infty} \int_{\R^N} (I_{\alpha/2} \ast \abs{u}^p)(I_{\alpha/2} \ast \abs{u_n - u}^p)
  \ge \int_{\R^N} \bigabs{I_{\alpha/2} \ast \abs{u}^p}^2,
\end{multline}
that is, we have a Brezis--Lieb type inequality.
Versions of the nonlocal Brezis--Lieb
property with equality were previously obtained in
\citelist{%
\cite{Ackermann2004}%
\cite{Ackermann2006}%
\cite{MorozVanSchaftingen}%
\cite{YangWei2013}%
\cite{BellazziniFrankVisciglia}%
}.

We shall deduce proposition~\ref{propositionBrezisLieb} from the following variant of the Brezis--Lieb lemma with mixed norms.

\begin{proposition}[Brezis--Lieb lemma with mixed norms]
\label{propositionBrezisLiebMixed}
Let \(p \in (0, \infty)\), \(q \in (0, \infty)\), \(A \subset \R^N\), \(B \subset \R^M\) be measurable sets, and
\((U_n)_{n \in \N}\) be a sequence of measurable functions from \(A \times B\) to \(\R\).
If the sequence \((U_n)_{n \in \N}\) converges almost everywhere to a function \(U : A \times B \to \R\) and
\[
  \sup_{n \in \N} \int_{A} \Bigl(\int_{B} \abs{U_n (x, y)}^p\dif y \Bigr)^q \dif x < \infty,
\]
then
\[
  \lim_{n \to \infty} \int_{A} \Bigl(\int_{B} \bigabs{\abs{U_n (x, y)}^p - \abs{U_n (x, y) - U (x, y)}^p - \abs{U (x, y)}^p} \dif y \Bigr)^q \dif x = 0.
\]
\end{proposition}
\begin{proof}
The proof follows the strategy of the original proof of the classical Brezis--Lieb lemma \cite{BrezisLieb1983}.
We first observe that by Fatou's lemma,
\begin{equation}\label{eqBrezisLiebBoundedLimit}
\begin{split}
\int_{A} \Bigl(\int_{B} \abs{U (x, y)}^p\dif y \Bigr)^q \dif x
&\le \int_{A} \Bigl(\liminf_{n \to \infty} \int_{B} \abs{U_n (x, y)}^p\dif y \Bigr)^q \dif x\\
&\le \liminf_{n \to \infty} \int_{A} \Bigl(\int_{B} \abs{U_n (x, y)}^p\dif y \Bigr)^q \dif x < \infty.
\end{split}
\end{equation}
Given \(\varepsilon > 0\), there exists \(C_\varepsilon > 0\) such that if \(s, t \in \R\), then
\[
  \bigabs{ \abs{s}^p - \abs{s - t}^p - \abs{t}^p} \le \varepsilon \abs{s}^p + C_\varepsilon \abs{t}^p.
\]
Hence, for every \(n \in \N\) and every \((x, y) \in A \times B\),
\[
  \bigl(\bigabs{ \abs{U_n (x, y)}^p - \abs{U_n (x, y) - U (x, y)}^p - \abs{U (x, y)}^p} - \varepsilon \abs{U_n (x, y)}^p\bigr)_+
  \le C_\varepsilon \abs{U (x, y)}^p.
\]
By \eqref{eqBrezisLiebBoundedLimit}, for almost every \(x \in A\),
we have \(U (x, \cdot) \in L^p (B)\), and thus by Lebesgue's dominated convergence theorem,
\[
  \lim_{n \to \infty} \int_{B} \bigl(\bigabs{ \abs{U_n (x, y)}^p - \abs{U_n (x, y) - U (x, y)}^p - \abs{U (x, y)}^p} - \varepsilon \abs{U_n (x, y)}^p\bigr)_+ \dif y= 0.
\]
Moreover, for every \(n \in \N\), we have
\[
  \int_{B} \bigl(\bigabs{ \abs{U_n (x, y)}^p - \abs{U_n (x, y) - U (x, y)}^p - \abs{U (x, y)}^p} - \varepsilon \abs{U_n (x, y)}^p\bigr)_+ \dif y
  \le C_\varepsilon \int_{B} \abs{U (x, y)}^p \dif y.
\]
Thanks to \eqref{eqBrezisLiebBoundedLimit}, we can apply a second time Lebesgue's dominated convergence theorem to deduce that
\[
  \lim_{n \to \infty} \int_{A} \Bigr( \int_{B} \bigl(\bigabs{ \abs{U_n (x, y)}^p - \abs{U_n (x, y) - U (x, y)}^p - \abs{U (x, y)}^p} - \varepsilon \abs{U_n (x, y)}^p\bigr)_+ \dif y\Bigr)^q \dif x
  = 0.
\]
Finally, we observe that
\begin{multline*}
  \int_{A} \Bigr( \int_{B} \bigabs{ \abs{U_n (x, y)}^p - \abs{U_n (x, y) - U (x, y)}^p - \abs{U (x, y)}^p} \dif y\Bigr)^q \dif x\\
  \le
  \int_{A} \Bigr( \int_{B} \bigl(\bigabs{ \abs{U_n (x, y)}^p - \abs{U_n (x, y) - U (x, y)}^p - \abs{U (x, y)}^p} - \varepsilon \abs{U_n (x, y)}^p\bigr)_+ + \varepsilon \abs{U_n (x, y)}^p\dif y\Bigr)^q \dif x.
\end{multline*}
Hence if \(q \ge 1\), by the Minkowski inequality,
\begin{multline*}
 \limsup_{n \to \infty} \int_{A} \Bigr( \int_{B} \bigabs{ \abs{U_n (x, y)}^p - \abs{U_n (x, y) - U (x, y)}^p - \abs{U (x, y)}^p} \dif y\Bigr)^q \dif x\\
  \le
  \limsup_{n \to \infty}
  \Bigl(\Bigl(\int_{A} \Bigr( \int_{B} \bigl(\bigabs{ \abs{U_n (x, y)}^p - \abs{U_n (x, y) - U (x, y)}^p - \abs{U (x, y)}^p} - \varepsilon \abs{U_n (x, y)}^p\bigr)_+\dif y\Bigr)^q \dif x \Bigr)^\frac{1}{q}\\
  + \Bigr(\int_{A} \Bigr(\int_{B}\varepsilon \abs{U_n (x, y)}^p\dif y\Bigr)^q \dif x\Bigr)^\frac{1}{q} \Bigr)^q
  \le \varepsilon^q \limsup_{n \to \infty}  \int_{A} \Bigr( \int_{B} \abs{U_n (x, y)}^p\dif y\Bigr)^q \dif x,
\end{multline*}
whereas if \(q \le 1\),
\begin{multline*}
  \limsup_{n \to \infty} \int_{A} \Bigr( \int_{B} \bigabs{ \abs{U_n (x, y)}^p - \abs{U_n (x, y) - U (x, y)}^p - \abs{U (x, y)}^p} \dif y\Bigr)^q \dif x\\
  \le
  \limsup_{n \to \infty}
  \int_{A} \Bigr( \int_{B} \bigabs{ \abs{U_n (x, y)}^p - \abs{U_n (x, y) - U (x, y)}^p - \abs{U (x, y)}^p} - \varepsilon \abs{U_n (x, y)}^p\bigr)_+\dif y\Bigr)^q\\
  + \Bigr(\int_{B}\varepsilon \abs{U_n (x, y)}^p\dif y\Bigr)^q \dif x \le \varepsilon^q \limsup_{n \to \infty}  \int_{A} \Bigr( \int_{B} \abs{U_n (x, y)}^p\dif y\Bigr)^q \dif x,
\end{multline*}
Since \(\varepsilon > 0\), is arbitrary, the conclusion follows in both cases.
\end{proof}

We are now in a position to prove the nonlocal Brezis--Lieb lemma.

\begin{proof}[Proof of proposition~\ref{propositionBrezisLieb}]
We define the function \(U_n : \R^N \times \R^N \to \R\) and \(U : \R^N \times \R^N \to \R\) for every \(x, y \in \R^N\) and \(n \in \N\) by \(U_n (x, y) = I_{\alpha/2} (x - y)^{1/p} u_n (y)\) and \(U (x, y) = I_{\alpha/2} (x, y)^{1/p} u (y)\). By assumption and by construction, the sequence \((U_n)_{n \in \N}\) converges almost everywhere to \(U\) in \(\R^N \times \R^N\). Moreover, since the Riesz kernel \(I_\alpha\) is nonnegative,
\[
  \sup_{n \in \N} \int_{\R^N} \Bigl(\int_{\R^N} \abs{U_n (x, y)}^p \dif y \Bigr)^2
  =\sup_{n \in \N} \int_{\R^N} \bigabs{I_{\alpha/2} \ast \abs{u_n}^p}^2 < \infty.
\]
Hence by proposition~\ref{propositionBrezisLiebMixed},
\begin{multline*}
  \lim_{n \to \infty} \int_{\R^N} \Bigabs{I_{\alpha/2} \ast \bigabs{\abs{u_n}^p - \abs{u_n - u}^p - \abs{u}^p}}^2\\
  = \lim_{n \to \infty}  \int_{\R^N} \Bigl(\int_{\R^N} \bigabs{\abs{U_n (x, y)}^p - \abs{U_n (x, y) - U (x, y)}^p - \abs{U (x, y)}^p} \dif y \Bigr)^2 \dif x = 0.
\end{multline*}
By the Cauchy--Schwarz inequality,
\begin{multline*}
  \int_{\R^N} \bigabs{\abs{I_{\alpha/2} \ast (\abs{u_n}^p)}^2 - \abs{I_{\alpha/2} \ast (\abs{u_n - u}^p + \abs{u}^p)}^2}\\
  = \int_{\R^N} \Bigabs{I_{\alpha/2} \ast (\abs{u_n}^p - \abs{u_n - u}^p - \abs{u}^p)}\bigabs{I_{\alpha/2} \ast (\abs{u_n}^p + \abs{u_n - u}^p + \abs{u}^p)}\\
  \le \Bigl(\int_{\R^N} \bigabs{I_{\alpha/2} \ast \bigabs{\abs{u_n}^p - \abs{u_n - u}^p - \abs{u}^p}}^2 \Bigr)^\frac{1}{2}
  \Bigl(\int_{\R^N} \bigabs{I_{\alpha/2} \ast (\abs{u_n}^p+ \abs{u_n - u}^p + \abs{u}^p)}^2 \Bigr)^\frac{1}{2}
\end{multline*}
and the conclusion follows.
\end{proof}

\subsection{Refined nonlocal Brezis--Lieb identity}

For our purpose the inequality \eqref{ineqBrezisLieb} given by  proposition~\ref{propositionBrezisLieb} will be sufficient.
For comparison with existing results and further use, we give some sufficient conditions for equality in \eqref{ineqBrezisLieb}.

We first obtain a general principle for equality.

\begin{proposition}
\label{propositionBrezisLiebEquiv}
Let \(N \in \N\), \(\alpha \in (0, N)\) and \(p \ge 1\).
Assume that \((u_n)_{n \in \N}\) is a sequence of measurable functions from \(\R^N\) to \(\R\)
that converges to \(u : \R^N \to \R\) almost everywhere.
If the sequence \((I_{\alpha/2} \ast \abs{u_n}^p)_{n \in \N}\) is bounded in \(L^2 (\R^N)\) and if \(u \ne 0\), then the following conditions are equivalent:
\begin{enumerate}[(i)]
\item \label{itEquivBLL1}
\(\displaystyle
  \lim_{n \to \infty}
  \int_{\R^N} \Bigabs{\bigabs{I_{\alpha/2} \ast \abs{u_n}^p}^2 - \bigabs{I_{\alpha/2} \ast \abs{u_n - u}^p}^2 -\bigabs{I_{\alpha/2} \ast \abs{u}^p}^2} = 0,
\)
\item \label{itEquivBL}
\(\displaystyle
  \lim_{n \to \infty}
  \int_{\R^N} \bigabs{I_{\alpha/2} \ast \abs{u_n}^p}^2 - \bigabs{I_{\alpha/2} \ast \abs{u_n - u}^p}^2
  = \int_{\R^N} \bigabs{I_{\alpha/2} \ast \abs{u}^p}^2,
\)
\item \label{itEquivLpLoc} the sequence \((u_n)_{n \in \N}\) converges to \(u\) strongly in \(L^p_\mathrm{loc} (\R^N)\) ,
\item \label{itEquivMeas} the sequence \((\abs{u_n}^p)_{n \in \N}\) converges to \(\abs{u}^p\) in the sense of measures,
\item \label{itEquivWeak} the sequence \((I_{\alpha/2} \ast \abs{u_n}^p)_{n \in \N}\) converges to \(I \ast \abs{u}^p\) weakly in \(L^2 (\R^N)\).
\end{enumerate}
\end{proposition}

The proof of proposition~\ref{propositionBrezisLiebEquiv} will use the local Brezis--Lieb lemma.

\begin{proposition}[Local Brezis--Lieb lemma]\label{propositionBrezisLiebClassical}
Let \(N \in \N\), \(p \ge 1\) and \(A \subset \R^N\) be a measurable set.
If the sequence \((u_n)_{n \in \N}\) of measurable functions from \(A\) to \(\R\) converges to \(u : \R^N \to \R\) almost everywhere and if \((u_n)_{n \in \N}\) is bounded in \(L^p (A)\), then
\[
 \lim_{n \to \infty} \int_A \bigabs{\abs{u_n}^p - \abs{u_n - u}^p - \abs{u}^p} = 0.
\]
\end{proposition}

This statement can be found in \cite{Lieb}*{lemma 2.6}. It might seem slightly stronger then the original Brezis--Lieb lemma, but it is however a direct consequence of the original proof \cite{BrezisLieb1983}*{theorem 2} (see also \cite{Willem2013}*{proof of theorem 4.2.7}). It also follows from proposition~\ref{propositionBrezisLiebMixed}.

The proof of proposition~\ref{propositionBrezisLiebEquiv} will also rely on the following classical continuity property of Riesz potentials.

\begin{proposition}[Weak continuity of Riesz potentials]
\label{propositionWeakContinuity}
Let \(N \in \N\) and \(\alpha \in (0, N)\).
Assume that \((\mu_n)_{n \in \N}\) is a sequence of signed Radon measures on $\R^N$ that converges weakly to a signed Radon measure \(\mu\).
If the sequence \((I_{\alpha/2} \ast \abs{\mu_n - \mu})_{n \in \N}\) is bounded in \(L^2 (\R^N)\), then the sequence \((I_{\alpha/2} \ast \mu_n)_{n \in \N}\) converges to \(I_{\alpha/2} \ast \mu\) weakly in \(L^2 (\R^N)\).
\end{proposition}

In particular, if \((\mu_n)_{n \in \N}\) is a sequence of \emph{nonnegative} measures then the sequence
\((I_{\alpha/2} \ast \mu_n)_{n \in \N}\) converges weakly in \(L^2 (\R^N)\) if and only if it is bounded in \(L^2 (\R^N)\).

Proposition~\ref{propositionWeakContinuity} can be deduced directly from the
weak convergence of the sequence \((\mu_n)_{n \in \N}\) in the space \(\mathcal E^\alpha (\R^N)\)
of Radon measures \(\nu \in \mathcal M(\R^N)\) such that
\(\norm{\nu}_{\mathcal E^\alpha}:=\norm{I_{\alpha/2}*\nu}_{L^2}<\infty\) \cite{DuPlessis}*{\S 3.4 and lemma 3.12}.

\begin{proof}[Proof of proposition~\ref{propositionWeakContinuity}]
Let \(\varphi \in C_c (\R^N)\).
Let \(\eta \in C_c (\R^N)\) such that \(\eta = 1\) on \(B_1\) and define \(\eta_R (x) = \eta (x/R)\).
For every \(n \in \N\) and every \(R > 0\), we have
\[
  \int_{\R^N} \bigl(I_{\alpha/2} \ast (\mu_n-\mu) \bigr) \varphi  =
  \int_{\R^N} \eta_R\bigl(I_{\alpha/2} \ast \varphi \bigr)(\mu_n - \mu)+
  \int_{\R^N} \bigl(I_{\alpha/2} \ast (1 - \eta_R)(\mu_n - \mu)\bigr) \varphi.
\]
We first observe that for every \(x, y, z \in \R^N\), if \(\abs{x} \ge 2 \min (\abs{y}, \abs{z})\) then
$$I_\alpha (x - z) \le 3^{N - \alpha} I_\alpha (x - y).$$
Hence, for every \(y \in \R^N \setminus B_R\),
\[
  \abs{I_\alpha \ast (1 - \eta_R) (\mu - \mu_n) (y)}
  \le 3^{N - \alpha} \int_{B_R} \bigabs{I_\alpha \ast (1 - \eta_R) \abs{\mu - \mu_n}}.
\]
Hence, by the Cauchy-Schwarz inequality, if \(\supp \varphi \subset B_{R/2}\),
\[
  \Bigabs{\int_{\R^N} \bigl(I_{\alpha/2} \ast (\mu_n-\mu) \bigr) \varphi}
  \le
  \Bigabs{\int_{\R^N} \eta_R\bigl(I_{\alpha/2} \ast \varphi \bigr)(\mu_n - \mu)}
  + \frac{C}{R^\frac{N}{2}} \Bigl(\int_{\R^N} \bigabs{\bigl(I_{\alpha/2} \ast \abs{\mu_n - \mu}}^2 \Bigr)\bigr)^\frac{1}{2} \int_{\R^N} \abs{\varphi}.
\]
The conclusion follows by letting \(R \to \infty\), and then \(n \to \infty\), since \(\eta_R\bigl(I_{\alpha/2} \ast \varphi \bigr) \in C_c (\R^N)\).
\end{proof}

\begin{proof}[Proof of proposition~\ref{propositionBrezisLiebEquiv}]
It is clear that \eqref{itEquivBLL1} implies \eqref{itEquivBL}.
If \eqref{itEquivBL} holds, we observe that in view of proposition~\ref{propositionBrezisLieb},
\begin{equation}
\label{eqLimEquality}
  \lim_{n \to \infty} \int_{\R^N} (I_{\alpha/2} \ast \abs{u}^p)(I_{\alpha/2} \ast \abs{u_n - u}^p) = 0.
\end{equation}
Since \(u \ne 0\), we conclude that \(I_{\alpha/2} \ast \abs{u}^p\) is locally bounded from below, and \(\abs{u_n - u}^p\) converges to \(0\) in \(L^1_{\mathrm{loc}} (\R^N)\).

If \eqref{itEquivLpLoc} holds, then by the classical Brezis--Lieb property of proposition~\ref{propositionBrezisLiebClassical}, the sequence
\((\abs{u_n}^p)_{n \in \N}\) converges in \(L^1_{\mathrm{loc}} (\R^N)\) to \(\abs{u}^p\).

If \eqref{itEquivMeas} holds then \eqref{itEquivWeak} follows from proposition~\ref{propositionWeakContinuity}.

Finally, assume that \eqref{itEquivWeak} holds. By the classical Brezis--Lieb
property of proposition~\ref{propositionBrezisLiebClassical},
the sequence \((\abs{u_n}^p - \abs{u_n - u}^p)_{n \in \N}\) converges weakly in the sense of measures to \(\abs{u}^p\).
Thus by proposition~\ref{propositionWeakContinuity} we conclude that the sequence \((I_{\alpha/2} \ast \abs{u_n - u}^p)_{n \in \N}\) converges weakly to \(0\) in \(L^2 (\R^N)\). This implies \eqref{eqLimEquality}, and hence \eqref{itEquivBLL1} follows by proposition~\ref{propositionBrezisLieb}.
\end{proof}

In view of proposition~\ref{propositionBrezisLiebEquiv} it is natural to ask for examples where the convergence does not occur.
We shall rely on the next lemma, which states that there exists a set of positive \(H^{\alpha/2}\)--capacity and vanishing Lebesgue measure (see for example \citelist{\cite{Carleson1967}*{theorem IV.3}\cite{Ohtsuka1957}\cite{AdamsHedberg1996}*{theorem 5.3.2}}). We give here an explicit construction in order to give more insight on the failure of the Brezis--Lieb equality.

\begin{lemma}
\label{lemmaCantor}
Let \(N \in \N\), \(\alpha \in (0, N)\) and \(p \ge 1\).
There exists a sequence \((f_n)_{n \in \N}\) in \(C^\infty_c (\R^N)\) such that:
\begin{enumerate}[(a)]
  \item the set \(\bigcup_{n \in \N} \supp f_n\) is bounded,
  \item for every \(n \in \N\), \(f_n \ge 0\),
  \item for every \(n \in \N\),
  \(
    \displaystyle \int_{\R^N} f_n = 1,
  \)
  \item the sequence \((f_n)_{n \in \N}\) converges almost everywhere to \(0\) in \(\R^N\),
  \item the sequence \((I_{\alpha/2} \ast \abs{f_n}^p)_{n \in \N}\) is bounded in \(L^2 (\R^N)\).
\end{enumerate}
\end{lemma}

In particular, if we take \(u_n = (u^p + f_n^p)^\frac{1}{p}\) for some given function \(u \in C^\infty_c (\R^N)\setminus \{0\}\), it is clear that the sequence \((u_n)_{n \in \N}\) converges to \(u\) almost everywhere, that the sequence \((I_{\alpha/2} \ast \abs{u_n}^p)\) is bounded in \(L^2 (\R^N)\) but that \((\abs{u_n}^p)_{n \in \N}\) does not converge to \(\abs{u}^p\) in the sense of measures.

\begin{proof}
Let \(Q = [-1, 1]^n\) be the unit cube. We take \(f_0 \in C^\infty_c (Q)\) such that
\(
  \int_{Q} f_0 = 1
\)
and \(f_0 \ge 0\).

Fix \(\rho > \frac{1}{2}\), and we define inductively the function \(f_n \in C^\infty_c (Q)\) for every \(n \in \N\) and \(x \in Q\) by
\[
  f_{n + 1} (x) = \frac{1}{(2 \rho)^N} \sum_{a \in \{-1/2, 1/2\}^N} f_n \Bigl(\frac{x - a}{\rho}\Bigr).
\]
It is clear that \(f_n \ge 0\) and that \(\int_{Q} f_n = 1\).
We compute
\[
\begin{split}
  \int_{\R^N} \abs{I_{\alpha/2} \ast f_{n + 1}}^2
  &= \rho^{N + \alpha} \sum_{\substack{a \in \{-1/2, 1/2\}^N\\ b \in \{-1/2, 1/2\}^N}}
  \int_{Q} \int_{Q} I_\alpha (\rho (x-y) + a - b) f_n (x) f_n (y)\dif x \dif y\\
  &\le \frac{1}{2^N \rho^{N - \alpha}} \int_{\R^N} \abs{I_{\alpha/2} \ast f_{n}}^2
  + \frac{C}{(1 + 2 \rho)^{N - \alpha}}.
\end{split}
\]
Therefore,
\[
  \int_{\R^N} \abs{I_{\alpha/2} \ast f_{n}}^2
  \le \frac{1}{(2^N \rho^{N - \alpha})^n} \int_{\R^N} \abs{I_{\alpha/2} \ast f_0}^2
  + \frac{1 - \frac{1}{(2^N \rho^{N - \alpha})^n}}{1 - \frac{1}{2^N \rho^{N - \alpha}}} \frac{C}{(1 + 2 \rho)^{N - \alpha}}.
\]
In particular, if \(\rho > 2^{-N/(N - \alpha)}\) then the sequence has the announced properties.
\end{proof}

The reader will recognize that the sequence \((f_n)_{n \in \N}\) converges to a measure concentrated on a generalized Cantor set of Hausdorff dimension \(N \log 2 / \log (1/\rho) > N - \alpha\). This is consistent with the known relationship between Hausdorff measure and capacity.
In particular, small values of \(\alpha\) require mildly concentrating sequences.

\begin{remark}
Note that construction in the last part of the proof of lemma~\ref{lemmaCoulombSobolev-noncompact}
(case $\alpha>2$ and $p>\frac{2\alpha}{\alpha-2}$) provides also an alternative proof
of lemma~\ref{lemmaCantor}. Indeed,
the restriction $\alpha>2$ and $p>\frac{2\alpha}{\alpha-2}$ in lemma~\ref{lemmaCoulombSobolev-noncompact}
was necessary to control the gradient term, otherwise one could take arbitrary $p\ge 1$.
\end{remark}

We now give some sufficient conditions for proposition~\ref{propositionBrezisLiebEquiv} to apply.

\begin{proposition}[Brezis--Lieb lemma with high local integrability]
\label{propositionBrezisLiebLp}
Let $N\in\mathbb N$, \(\alpha \in (0, N)\) and \(p \ge 1\).
Assume that \((u_n)_{n \in \N}\) is a sequence of measurable functions from \(\R^N\) to \(\R\)
that converges to \(u : \R^N \to \R\) almost everywhere.
If the sequence \((I_{\alpha/2} \ast \abs{u_n}^p)_{n \in \N}\) is bounded in \(L^2 (\R^N)\) and if there exists \(q > p\) such that the sequence \((u_n)_{n \in \N}\) is bounded in \(L^q_{\mathrm{loc}} (\R^N)\), then \((I_{\alpha/2} \ast \abs{u_n}^p)_{n \in \N}\) converges weakly to \(I_{\alpha/2} \ast \abs{u}^p\) in \(L^2 (\R^N)\) and
\[
  \lim_{n \to \infty}
  \int_{\R^N} \Bigabs{\bigabs{I_{\alpha/2} \ast \abs{u_n}^p}^2 - \bigabs{I_{\alpha/2} \ast \abs{u_n - u}^p}^2}
  = \int_{\R^N} \bigabs{I_{\alpha/2} \ast \abs{u}^p}^2.
\]
\end{proposition}

The local boundedness assumption is satisfied in particular if the sequence \((u_n)_{n \in \N}\) is bounded in \(L^q (\R^N)\).
When \(q = \frac{2 N p}{N + \alpha}\), we recover a result of Moroz and Van Schaftingen \cite{MorozVanSchaftingen} whereas when \(p = 2\) and \(q > 2\) this is due to Bellazzini, Frank and Visciglia \cite{BellazziniFrankVisciglia}.

\begin{proof}[Proof of proposition~\ref{propositionBrezisLiebLp}]
The almost everywhere convergence and boundedness in \(L^q_{\mathrm{loc}} (\R^N)\) imply that the sequence \((\abs{u_n}^p)_{n \in \N}\) converges weakly in \(L^{q/p}_\mathrm{loc} (\R^N)\) \citelist{\cite{Bogachev2007}*{proposition~4.7.12}\cite{Willem2013}*{proposition 5.4.7}}.
\end{proof}

\begin{proposition}[Brezis--Lieb lemma in Coulomb--Sobolev spaces]
Let $N\in\N$, \(\alpha \in (0, N)\) and $p\ge 1$ be such that
\[
  \frac{1}{p} > \frac{1}{2}-\frac{1}{\alpha}.
\]
Assume that \((u_n)_{n \in \N}\) is a sequence of measurable functions from \(\R^N\) to \(\R\)
that converges to \(u : \R^N \to \R\) almost everywhere.
If the sequence \((I_{\alpha/2} \ast \abs{u_n}^p)_{n \in \N}\) is bounded in \(L^2 (\R^N)\) and if  the sequence \((D u_n)_{n \in \N}\) is bounded in \(L^2 (\R^N)\), then
\[
  \lim_{n \to \infty}
  \int_{\R^N} \Bigabs{\bigabs{I_{\alpha/2} \ast \abs{u_n}^p}^2 - \bigabs{I_{\alpha/2} \ast \abs{u_n - u}^p}^2}
  = \int_{\R^N} \bigabs{I_{\alpha/2} \ast \abs{u}^p}^2.
\]
\end{proposition}
\begin{proof}
We observe that by proposition~\ref{propositionAdvancedLocalWeakCompactness}, the sequence \((u_n)_{n \in \N}\) is relatively compact in \(L^p_{\textrm{loc}} (\R^N)\). Then we apply proposition~\ref{propositionBrezisLiebEquiv}.
\end{proof}

\section{Groundstates of Schr\"odinger--Poisson--Slater equations}

\label{sect-equation}

\subsection{Existence of optimizers to multiplicative inequalities and proof of theorem \ref{t-multiplicative-intro}}
\settocdepth{subsection}

Let \(N \in \N\), \(\alpha \in (0, N)\), \(p \ge 1\) and $q\ge 1$ be such that assumption \eqref{Q} holds.
For each $u\in E^{\alpha,p}(\R^N)\setminus\{0\}$, we define the quotient
$$\mathcal R(u)=\frac{\Big(\int_{\R^N}\abs{D u}^2 \Big)^{\frac{\theta}{2}}\Bigl(\int_{\R^N} \abs{I_{\alpha/2} \ast \abs{u}^p}^2\Bigr)^{\frac{1-\theta}{2p}}}{\Big(\int_{\R^N}\abs{u}^q \Big)^{\frac{1}{q}}},$$
where the parameter $\theta$ is given by \eqref{e-theta} of theorem~\ref{theoremEmbedding}.
For $a, b>0$, define
\begin{equation}\label{quotient}
S:=\inf\Big\{\mathcal R(u)\, : \, u\in E^{\alpha,p}(\R^N)\setminus\{0\} \Big\},
\end{equation}
\begin{equation*}
S_{a}:=\inf\Big\{\mathcal R(u)\, : \, u\in E^{\alpha,p}(\R^N)\setminus\{0\},\,\int_{\R^N}\abs{u}^{q} =a\Big\}
\end{equation*}
and
\begin{equation*}
S_{a,b}:=\inf\Big\{\mathcal R(u)\, : \, u\in E^{\alpha,p}(\R^N)\setminus\{0\},\,\int_{\R^N}\abs{D u}^2  =a,\, \int_{\R^N}\abs{u}^{q}=b\Big\}.
\end{equation*}
By using the scaling and homogeneity invariance of $\mathcal R$ we see that
\begin{equation*}
S=S_a=S_{a,b}.
\end{equation*}
In fact, using the invariance, it is possible to minimize constraining either one or two of the three integrals
involved in the quotient $\mathcal R$, without changing the value of the infimum.

The following lemma is in the spirit of a classical result of Lieb, see e.g.
\cite{LiebLoss2001}*{p.\thinspace{} 215}.
Together with proposition~\ref{propositionElementaryLocalWeakCompactness} it implies that, under a suitable condition,
a given bounded sequence in $E^{\alpha,p}(\R^N)$ contains a subsequence which after translations has a nontrivial
limit in $L^1_{\textrm{loc}}(\R^N)$. In view of proposition~\ref{weakL1loc},
when $p>1$ this is equivalent for a bounded sequence in $E^{\alpha,p}(\R^N)$ to
have a nonzero weak limit, up to translations and a subsequence.
A similar statement was recently established in the fractional case
\cite{BellazziniFrankVisciglia}*{Lemma 2.1}.
We give a short self-contained proof, in the spirit of the proofs
of interpolation inequalities between Sobolev spaces and Morrey spaces
given recently in \cite{VanSchaftingen2013}.

\begin{lemma}[On vanishing bounded sequences]\label{liebtypelemma0}
Let $N\in\mathbb N$ and $(u_n)_{n\in\N}$ be a sequence in $W^{1,
1}_{\mathrm{loc}} (\R^N)$.
If
\[
 \limsup_{n \to \infty} \int_{\R^N} \abs{\nabla u_n}^2 < \infty
\]
and there exists \(R > 0\) such that
\[
 \lim_{n \to \infty} \sup_{x \in \R^N} \int_{B_R (x)} \abs{u_n}
 = 0,
\]
then for every \(\varepsilon > 0\),
\begin{equation*}
  \lim_{n \to \infty} \mathcal L^N(\{x : \abs{u_n(x)}> \varepsilon\}) = 0.
\end{equation*}
\end{lemma}

\begin{proof}
We first observe that the assumption implies that for \emph{every} \(R >
0\),
\[
 \lim_{n \to \infty} \sup_{x \in \R^N} \int_{B_R (x)} \abs{u_n}
 = 0.
\]

Following the strategy of \cite{VanSchaftingen2013}, we start from the Sobolev
representation formula: for every \(n \in \N\) and almost every \(x \in \R^N\),
we have
\[
 u (x) = \frac{1}{\omega_N R^N}
 \int_{B_R(x)} u + \int_{B_R(x)} \nabla u_n (x) \cdot S_R (x - y) \dif y,
\]
where \(\omega_N =\pi^{N/2}/\Gamma (\frac{N}{2} + 1)\) is the volume of the \(N\)--dimensional unit ball and
the Sobolev kernel \(S_R : B_R(0) \to \R^N\) is
defined by
\[
 S_R (z)
 = - \Bigl(\frac{1}{\abs{z}^N} - \frac{1}{R^N} \Bigr)\frac{z}{N\omega_N}.
\]
For every \(\varepsilon > 0\), we have
\begin{equation*}
 \{ x \in \R^N : u_n (x) \ge \varepsilon\}\\
 \subseteq \Bigl\{ x \in \R^N : \int_{B_R(x)} \abs{u_n} \ge \frac{\omega_N R^N
\varepsilon}{2} \Bigr\}
\cup \Bigl\{ x \in \R^N : \bigabs{S_R \ast \nabla u_n} (x) \ge
\frac{\varepsilon}{2}\Bigr\}
 \cup E,
\end{equation*}
with \(\mathcal{L}^N (E) = 0\).

We now observe that
\[
 \mathcal{L}^N \bigl(\bigl\{ x \in \R^N : (S_R \ast \nabla u_n) (x) \ge
\tfrac{\varepsilon}{2}\bigr\}\bigr)
\le \frac{4}{\varepsilon^2} \int_{\R^N} \abs{S_R \ast \nabla u_n}^2
\le \frac{4}{\varepsilon^2} \Bigl(\int_{B_R(0)} \abs{S_R} \Bigr)^2
\int_{\R^N} \abs{\nabla u_n}^2.
\]
Since
\[
 \int_{B_R(0)} \abs{S_R} = R \frac{N}{N + 1},
\]
we conclude that
\begin{equation*}
 \mathcal{L}^N \bigl(\{ x \in \R^N : u_n (x) \ge \varepsilon\}\bigr)
 \le \mathcal{L}^N \Bigl(\Bigl\{ x \in \R^N : \int_{B_R(x)} \abs{u_n} \ge
\frac{\omega_N R^N
\varepsilon}{2} \Bigr\}\Bigr)
+ \Bigl(\frac{2 NR}{(N+1)\varepsilon}\Bigr)^2
\int_{\R^N} \abs{\nabla u_n}^2.
\end{equation*}
In order to obtain the conclusion, for every \(\eta > 0\), we take \(R >
0\) small enough so that for each \(n \in \N\),
\[
\Bigl(\frac{2NR}{(N+1)\varepsilon}\Bigr)^2
\int_{\R^N} \abs{\nabla u_n}^2 \le \eta;
\]
By the other assumption, when \(n \in \N\) is large enough
\[
 \Bigl\{ x \in \R^N : \int_{B_R(x)} \abs{u_n} \ge
\frac{\omega_N R^N
\varepsilon}{2} \Bigr\} = \emptyset,
\]
and the conclusion follows.
\end{proof}
An immediate consequence of the preceding lemma is the following
\begin{lemma}[Nonzero weak limit after translations]\label{liebtypelemma}
Let $N\in\mathbb N$, \(\alpha \in (0, N)\) and \(p \ge 1\).
Let $(u_n)_{n\in\N}$ be a sequence in $E^{\alpha,p}(\R^N)$ such that $\norm{u_n}_{E^{\alpha,p}(\R^N)}< C$ and let $\varepsilon,\delta>0$ be such that for all $n\in \N$
\begin{equation}\label{measurebound}
\mathcal L^N(x : |u_n(x)|> \varepsilon) > \delta .
\end{equation}
Then there exists a sequence $(a_n)_{n\in\N}$ in $\R^N$ such that $v_n:=u_n(\cdot+a_n)$ does not contain any subsequence converging to zero in $L^1_\loc(\R^N).$
\end{lemma}

\begin{proof}
By Lemma~\ref{liebtypelemma0} $R,C_0>0$ and a sequence  $(a_n)_{n\in\N}$ in $\R^N$ exist such that $$\int_{B_R(0)} |v_n|= \int_{B_{R} (a_n)} \abs{u_n} > C_0>0,$$ where  $v_n:=u_n(\cdot+a_n)$. And this concludes the proof.
\end{proof}

Now we are in a position to prove theorem~\ref{t-multiplicative-intro} of the Introduction.
For convenience, we reproduce here the statement.

\begin{theorem}[Existence of optimizers]\label{t-multiplicative}
Let \(N \in \N\), \(\alpha \in (0, N)\), \(p\ge 1\) and assumption \eqref{Q_0} holds.
Then the best constant $S$ in \eqref{quotient} is achieved.
\end{theorem}

\begin{proof}
Let \((u_n)_{n \in \N}\) be a minimizing sequence for $S$ such that \(\norm{D u_n}_{L^2} = 1\) and \(\norm{u_n}_{L^q}= 1\). Such a minimizing sequence is obviously bounded in $E^{\alpha,p}(\R^N)$. By the so-called $p,q,r$ theorem, \cite{FrolichLiebLoss}*{ Lemma 2.1 p.\thinspace{}258}, using the assumption on $q$, and by theorem~\ref{theoremEmbedding} it follows that \eqref{measurebound} holds. By lemma~\ref{liebtypelemma} and proposition~\ref{propositionElementaryLocalWeakCompactness},
up to translations and a subsequence \((u_n)_{n \in \N}\) converges in $L^1_{\textrm{loc}}(\R^N)$ and almost everywhere in \(\R^N\) to a nontrivial limit \(u\).

Passing if necessary to a subsequence, by proposition~\ref{propositionSemiContinuityLocalConvergence}
and by proposition~\ref{propositionBrezisLieb},
\[
\begin{split}
 S &= \lim_{n \to \infty} \frac{\norm{D u_n}_{L^2 (\R^N)}^\theta \norm{I_{\alpha/2} \ast \abs{u_n}^p}_{L^2 (\R^N)}^\frac{1 - \theta}{p}}{\norm{u_n}_{L^q (\R^N)}}\\
 & \ge \limsup_{n \to \infty} \frac{(\norm{D u}^2_{L^2 (\R^N)} + \norm{D (u_n - u)}^2_{L^2 (\R^N)})^\frac{\theta}{2} (\norm{I_{\alpha/2} \ast \abs{u}^p}_{L^2 (\R^N)}^2 + \norm{I_{\alpha/2} \ast \abs{u_n - u}^p}_{L^2 (\R^N)}^2)^\frac{1 - \theta}{2p}}{\norm{u_n}_{L^q (\R^N)}},
\end{split}
\]
where
\[
  \theta = \frac{\frac{1}{q} - \frac{1}{2 p}\bigl(1 + \frac{\alpha}{N}\bigr)}{\frac{1}{2} - \frac{1}{N} - \frac{1}{2 p} \bigl(1 + \frac{\alpha}{N}\bigr)}.
\]
By the discrete H\"older inequality, this implies that
\[
 S \ge \limsup_{n \to \infty} \frac{\bigl(\norm{D u}_{L^2 (\R^N)}^\frac{2\theta p}{\theta p + (1 - \theta)}\norm{I_{\alpha/2} \ast \abs{u}^p}_{L^2 (\R^N)}^\frac{2(1 - \theta)}{\theta p + (1 - \theta)} + \norm{D (u_n - u)}_{L^2 (\R^N)}^\frac{2 \theta p}{\theta p + (1 - \theta)} \norm{I_{\alpha/2} \ast \abs{u-u_n}^p}_{L^2 (\R^N)}^\frac{2(1 - \theta)}{\theta p + (1 - \theta)} \bigr)^{\frac{\theta}{2} + \frac{1 - \theta}{2 p}}}{\norm{u_n}_{L^q (\R^N)}}.
\]
By definition of \(S\) and by the classical Brezis--Lieb lemma this implies that
\[
S \ge \limsup_{n \to \infty} S \frac{\bigl(\norm{u }_{L^q (\R^N)}^\frac{2 p}{\theta p + (1 -\theta)} + \norm{u_n - u}_{L^q (\R^N)}^\frac{2 p}{\theta p + (1 - \theta)}\bigr)^{\frac{\theta}{2} + \frac{(1 - \theta)}{2p}}}{\bigl(\norm{u}_{L^q (\R^N)}^q + \norm{u_n - u}^q_{L^q (\R^N)} \bigr)^\frac{1}{q}}.
\]
Since by our assumption,
\[
  \frac{1}{q} < \frac{\theta}{2} + \frac{1 - \theta}{2 p},
\]
it follows by strict concavity and since \(\norm{u}_{E^{\alpha,p}(\R^N)} \ne 0\), that
\[
  \lim_{n \to \infty}\norm{u-u_n}_{L^q (\R^N)}= 0,
\]
passing if necessary to a subsequence. In view of the Fatou property (proposition~\ref{propositionSemiContinuityLocalConvergence}) this implies that
$$
S = \lim_{n \to \infty} \frac{\norm{D u_n}_{L^2 (\R^N)}^\theta \norm{I_{\alpha/2} \ast \abs{u_n}^p}_{L^2 (\R^N)}^\frac{1 - \theta}{p}}{\norm{u_n}_{L^q  (\R^N)}}\geq  \frac{\norm{D u}_{L^2 (\R^N)}^\theta \norm{I_{\alpha/2} \ast \abs{u}^p}_{L^2 (\R^N)}^\frac{1 - \theta}{p}}{\norm{u}_{L^q (\R^N)}},
$$
which is enough to prove the claim.
\end{proof}

We emphasise that assumptions of theorem~\ref{t-multiplicative} include $p=1$, although in this case
there is no obvious Euler--Lagrange equation which could be associated to  $\mathcal R$.
If $p>1$ then the Euler--Lagrange equation of the quantity $\log\mathcal R(u)$
for $u\in E^{\alpha,p}(\R^N)\setminus\{0\}$ has the form
\begin{equation}\label{EulerMinim1}
 A(-\Delta)u + B(I_\alpha \ast \abs{u}^p)\abs{u}^{p - 2} u= C\abs{u}^{q-2}u\quad\text{in \(\R^N\)},
\end{equation}
where
$$A=\frac{\theta}{\|D u\|_{L^2(\R^N)}^2},\quad
B=\frac{1-\theta}{\|I_{\alpha/2} \ast \abs{w}^p\|_{L^2(\R^N)}^2},\quad C=\frac{1}{\|u\|_{L^q(\R^N)}^q}.$$
In particular, minimizers for $S$ constructed in theorem~\ref{t-multiplicative}
are weak solutions of \eqref{EulerMinim1} and, after a rescaling, of equation \eqref{sps}.
Note also that if $u$ is a minimizer for $S$ then $\abs{u}$ is also a minimizer and hence
we can assume that $u$ is nonnegative.

In the next section, we are going to consider an equivalent to \eqref{quotient} additive minimization problem
for the functional which has a meaning of the physical energy which is naturally associated to \eqref{sps}.

\subsection{Additive minimization problem}\label{sect-additive}

For $u\in E^{\alpha,p}(\R^N)$, set
\begin{equation}\label{E-defn}
\mathcal{E}_*(u)=\frac{1}{2}\int_{\R^N}\abs{D u}^2+\frac{1}{2p}\int_{\R^N} \abs{I_{\alpha/2} \ast \abs{u}^p}^2 ,
\end{equation}
and for $c>0$ define a minimization problem
\begin{equation}\label{varprob1nonrad}
M_c=\inf_{u\in \mathcal A_c} \mathcal{E}_*(u)
\end{equation}
where
\begin{equation*}
\mathcal A_c=\Bigl\{u \in E^{\alpha,p}(\R^N) \,: \int_{\R^N}\abs{u}^q=c\Bigr\}.
\end{equation*}
Up to a rescaling, for $p>1$ this problem shares with \eqref{quotient} the same Euler-Lagrange equation.
In fact, one can show that minimization problems \eqref{quotient} and \eqref{varprob1nonrad}
are equivalent. Indeed, given $u\in\mathcal A_c$,
consider a rescaling $u_\lambda(x)=\lambda^{-N/q} u(x/\lambda)\in\mathcal A_c$.
Minimizing $\mathcal \mathcal{E}_*(u_\lambda)$ with respect to $\lambda>0$, after a direct computation we find
that
$$\min_{\lambda>0} \mathcal{E}_*(u_\lambda)=C_*\big(\mathcal R(u)\big)^{2\sigma},$$
where
$$\sigma={\frac{(N+\alpha)-p(N-2)}{(\alpha+2)-\frac{2N}{q}(p-1)}}$$
and
$$C_*=
\left(\Big(\frac{1-\theta}{p\theta}\Big)^{\sigma\theta}+\Big(\frac{1-\theta}{p\theta}\Big)^{-\sigma\frac{1-\theta}{p}}\right)
\Big(\frac{1}{2^\theta(2p)^\frac{1-\theta}{p}}\Big)^\sigma,$$
with $\theta$ as in \eqref{eq-theta}.
We conclude that the best constant $S$
in the multiplicative minimization problem \eqref{quotient} is achieved if and only if $M_c$ is achieved in the additive minimization problem \eqref{varprob1nonrad}. Moreover,
\begin{equation}\label{scaledinfimum}
M_c=C_*\big(c^\frac{1}{q}S\big)^{2\sigma}=c^\frac{2\sigma}{q} M_1.
\end{equation}
In particular, Theorem~\ref{t-multiplicative} implies the existence of a minimizer for \eqref{varprob1nonrad}.
Nevertheless, below we sketch an independent proof,
which uses the same tools as the proof of Theorem~\ref{t-multiplicative} but provides an additional information
about the properties of the minimizing sequences.

\begin{theorem}\label{thm-additive}
Let $N\in\mathbb N,$ \(\alpha \in (0, N)\), \(p \ge 1\)  and assumption \eqref{Q_0} holds.
Then all the minimising sequences for $M_c$ in \eqref{varprob1nonrad} are relatively compact in $E^{\alpha,p}(\R^N)$ modulo translations. In particular, the best constant $M_c$ is achieved.
\end{theorem}

\begin{proof}
 Let \((u_n)_{n \in \N}\) be a minimizing sequence for $M_1$. By $p,q,r$ theorem \cite{FrolichLiebLoss}*{Lemma 2.1 p. 258}, using the assumption on $q$ and by Theorem~\ref{theoremEmbedding} it follows that \eqref{measurebound} holds. By using Lemma~\ref{liebtypelemma} and Proposition~\ref{propositionElementaryLocalWeakCompactness},
up to translations and a subsequence we can assume that \((u_n)_{n \in \N}\) converges in $L^1_{\textrm{loc}}(\R^N)$ and almost everywhere in \(\R^N\) to a nontrivial limit \(u\).

Now, up to subsequence, there holds
\[
  \lim_{n \to \infty}\norm{u-u_n}_{L^q (\R^N)}= 0,
\]
and in particular  $\norm{u}_{L^q (\R^N)}=1$.

In fact, passing if necessary to a subsequence, by Proposition~\ref{propositionSemiContinuityLocalConvergence},
by Proposition~\ref{propositionBrezisLieb} and by \eqref{scaledinfimum} we obtain
\begin{multline*}
M_1 = \lim_{n \to \infty} \Big(\frac{1}{2}\norm{D u_n}_{L^2 (\R^N)}^2+ \frac{1}{2p}\norm{I_{\alpha/2} \ast \abs{u_n}^p}_{L^2 (\R^N)}^{2}\Big)\\
 \ge \limsup_{n \to \infty} \Big(\frac{1}{2}\norm{D u}^2_{L^2 (\R^N)} + \frac{1}{2}\norm{D (u_n - u)}^2_{L^2 (\R^N)}\hspace{12em}\\ +\frac{1}{2p}\norm{I_{\alpha/2} \ast \abs{u}^p}_{L^2 (\R^N)}^2 + \frac{1}{2p}\norm{I_{\alpha/2} \ast \abs{u_n - u}^p}_{L^2 (\R^N)}^2\Big)\\
\ge M_1\limsup_{n \to \infty} \Big(\norm{u}^{2\sigma}_{L^q (\R^N)}+\norm{u-u_n}^{2\sigma}_{L^q (\R^N)}\Big).
\end{multline*}
In view of the assumption \eqref{Q_0}, we have $2\sigma/q\in (0,1)$. The strong convergence in $L^q (\R^N)$ follows by strict concavity, as  \(\norm{u}_{L^q (\R^N)} \ne 0\).

As a consequence $M_1$ is achieved, since by Proposition~\ref{propositionSemiContinuityLocalConvergence}  (Fatou's property) we have
$$M_1 \ge \frac{1}{2}\norm{D u}_{L^2 (\R^N)}^2+\frac{1}{2p}\norm{I_{\alpha/2} \ast \abs{u}^p}_{L^2 (\R^N)}^{2}.$$
Notice that the convergence is in $E^{\alpha,p}(\R^N)$ in view of Proposition~\ref{propositionBrezisLieb} and  Proposition~\ref{propositionSemiContinuityLocalConvergence}.
\end{proof}

\begin{remark}
By interpreting \((|u_n|^q)_{n \in \N}\) as a sequence of probability measures, in the language of the concentration-compactness principle of P.L. Lions \cite{Lions1984CC1}, Lemma~\ref{liebtypelemma} rules out the vanishing case, and \eqref{scaledinfimum} yields the strict subadditive inequality
\begin{equation*}
M_{c_1+c_2}<M_{c_1}+M_{c_2},\qquad c_1,c_2>0,
\end{equation*}
which rules out the dichotomy case.
\end{remark}

\subsection{Existence of groundstates and proof of theorem \ref{t-groundstate}}

We now formulate our main result on the existence of groundstates solutions of equation \eqref{sps}, namely theorem \ref{t-groundstate} from the introduction, which we recall here for reader's convenience:

\begin{customthm}{\ref{t-groundstate}}
Let \(N \in \N\), \(\alpha \in (0, N)\), \(p>1\) and  assumption \eqref{Q_0} holds.
Then there exists a nontrivial nonnegative groundstate solution $u\in E^{\alpha,p}(\R^N)\cap C^2(\R^N)$ to equation \eqref{sps} and $u\in C^\infty(\R^N\setminus u^{-1}(0))$.
In addition, if $p\ge 2$ then $u(x)>0$.
\end{customthm}
\begin{proof}
Let $w\in E^{\alpha,p}(\R^N)$ be a minimiser for $M_1$ in \eqref{varprob1nonrad}.
Since $|w|\in E^{\alpha,p}(\R^N)$ is also a minimiser for $M_1$, we can assume that $w$ is nonnegative.
Since $p>1$, the energy $E$ is of class $C^1(E^{\alpha,p}(\R^N);\mathbb R)$ and
the Euler-Lagrange equation for $w$ can be written in the form
\begin{equation*}
 - \Delta w + (I_\alpha \ast \abs{w}^p)\abs{w}^{p - 2} w=\mu\abs{w}^{q-2}w\quad\text{in \(\R^N\)},
\end{equation*}
with a Lagrange multiplier $\mu>0$.

Writing
$w(x)=\gamma u(\delta x)$,  for arbitrary $\gamma,\delta>0$,
it follows that $u$ is a solution of
\begin{equation}\label{EulerMinim2+}
 - \Delta u + \gamma^{2p-2}\delta^{-\alpha-2}(I_\alpha \ast \abs{u}^p)\abs{u}^{p - 2} u= \mu\gamma^{q-2}\delta^{-2}\abs{u}^{q-2}u\quad\text{in \(\R^N\).}
\end{equation}
Since $q\neq 2\frac{2p+\alpha}{2+\alpha}$ we can define $\gamma,\delta>0$ such that
$$ \gamma^{2p-2}\delta^{-\alpha-2}=1 \qquad \textrm{and}\qquad \mu\gamma^{q-2}\delta^{-2}=1.$$
It follows that $u$ is a nonnegative solution to the equation \eqref{sps}.

Regularity and positivity properties of the ground states of \eqref{sps} will follow from the results
in the remaining part of this section.
\end{proof}

\begin{remark}
The above simple scaling argument shows that the condition $q\neq 2\frac{2p+\alpha}{2+\alpha}$ is necessary and sufficient to get rid of any arising multipliers for this class of Euler-Lagrange equations.
\end{remark}

\subsection{Regularity and positivity of weak solutions}
We first establish regularity of weak solutions of \eqref{sps}.

\begin{proposition}[Local regularity]\label{proposition:Regularity}
Let \(N \in \N\), \(\alpha \in (0, N)\), \(p>1\) and  assumption \eqref{Q_0} holds.
If \(u \in E^{\alpha, p} (\R^N)\) is a weak solution of the equation
\begin{equation}\label{main-reg}
  -\Delta u + \bigl(I_\alpha \ast \abs{u}^p\bigr) \abs{u}^{p - 2} u = \abs{u}^{q - 2} u\quad\text{in \(\R^N\),}
\end{equation}
where $\mu>0$, then the following holds:
\begin{itemize}
\item
\(u \in L^r (\R^N)\) for every \(r \in [1, \infty)\) such that
\[
  \frac{1}{r} \le \max \Bigl\{\frac{1}{2} - \frac{1}{N}, \frac{2 + \alpha}{2(2 p + \alpha)}\Bigr\},
\]
\item
\(u \in C^{k, \gamma}_{\mathrm{loc}} (\R^N)\)
for every \(k \in \N\) and \(\gamma \in (0, 1)\) such that \(k + \gamma \le \min \{\Tilde{p}, \Tilde{q}\} + 1\),
where for $s\in\R$ we denote \(\Tilde{s} = \infty\) if \(s\) is an even integer and \(\Tilde{s} = s\) otherwise,
\smallskip
\item
\(u \in C^\infty(\R^N \setminus u^{-1} (0))\).
\end{itemize}

\end{proposition}

\begin{proof}[Proof of proposition~\ref{proposition:Regularity}.]
\setcounter{claim}{0}
The essential new step in the proof is the following claim which improves regularity using the Coulomb--Sobolev embedding.

\begin{claim}
\label{claimIteration}
Let \(r \ge q\) and \(u \in L^r (\R^N)\).
If \(s \in [1, \infty)\) is such that
either \(r - q \le \frac{2}{2 + \alpha} \bigl((N - 2)p - (N + \alpha) \bigr)\) and
\[
  \frac{N - 2}{N (2 + r - q)} \ge \frac{1}{s} \ge \frac{\alpha + 2}{2(\alpha + 2p)+(\alpha+2)(r - q)}
\]
or \(r - q \ge \frac{2}{2 + \alpha} \bigl((N - 2)p - (N + \alpha) \bigr)\) and
\[
  \frac{(N - 2)_+}{N (2 + r - q)} \le \frac{1}{s} \le \frac{\alpha + 2}{2(\alpha + 2p)+(\alpha+2)(r - q)},
\]
then \(u \in L^s (\R^N)\).
\end{claim}

In particular, if $N\ge 3$ then $u\in L^\frac{N(2+r-q)}{N-2}(\R^N)\cap L^{2\frac{2p+\alpha}{2+\alpha}+r-q}(\R^N)$.

\begin{proofclaim}
Given \(\mmu > 0\), we define the truncated solution \(u_\mmu \in E^{\alpha,p} (\R^N)\) for each \(x \in \R^N\) by
\[
  u_\mmu (x)
  =\begin{cases}
      -\mmu &\text{if \(u (x) < -\mmu\)},\\
      u (x)&\text{if \(-\mmu \le u (x) \le \mmu\)},\\
      \mmu & \text{if \(u (x) > \mmu\)}.
   \end{cases}
\]
Given \(\beta > 1/2\), we test the equation against the function \(\abs{u_\mmu}^{2\beta - 2} u_\mmu\) and we get
\[
  \int_{\R^N} \abs{u}^{q - 2}u \abs{u_\mmu}^{2\beta - 2} u_\mmu
  = \int_{\R^N} \nabla u \cdot \nabla (\abs{u_\mmu}^{2\beta - 2} u_\mmu) + (I_\alpha \ast \abs{u}^p)  \abs{u}^{p-2}u \abs{u_\mmu}^{2\beta - 2} u_\mmu.
\]
Since \( u_\mmu u \le u^2\), we first have
\[
  \int_{\R^N} \abs{u}^{q - 2}u \abs{u_\mmu}^{2\beta - 2}  u_\mmu \le \int_{\R^N} \abs{u}^{q + 2\beta-2}.
\]
Since \(\nabla u_\mmu = \nabla u\) on $\supp(\nabla u_\mmu)$, we also have
\[
  \int_{\R^N} \nabla u \cdot \nabla (\abs{u_\mmu}^{2\beta - 2} u_\mmu)
  =\frac{2\beta - 1}{\beta^2} \int_{\R^N} \abs{\nabla \abs{u_\mmu}^{\beta}}^2.
\]
Finally we have, by the Cauchy--Schwarz inequality
\[
\begin{split}
  \int_{\R^N} (I_\alpha \ast \abs{u}^p) \abs{u}^{p - 2} u\abs{u_\mmu}^{2\beta - 2} u_\mmu
  &\ge \int_{\R^N} (I_\alpha \ast \abs{u_\mmu}^p) \abs{u_\mmu}^{p + 2\beta - 2}\\
  &\ge \int_{\R^N} (I_\alpha \ast \abs{u_\mmu}^{p + \beta - 1}) \abs{u_\mmu}^{p + \beta - 1}.
\end{split}
\]
In summary, we have the estimate
\[
  \frac{2\beta - 1}{\beta^2} \int_{\R^N} \abs{\nabla \abs{u_\mmu}^\beta}^2 + \int_{\R^N} (I_\alpha \ast (\abs{u_\mmu}^\beta)^{1 + \frac{p - 1}{\beta}}) (\abs{u_\mmu}^\beta)^{1 + \frac{p - 1}{\beta}}
  \le \int_{\R^N} \abs{u}^{q + 2 \beta - 2}.
\]
We now take \(\beta = 1 + \frac{r - q}{2}\). Then
\[
  1 + \frac{p - 1}{\beta} = \frac{2 p + r - q}{2 + r - q}
\]
and hence $|u_\mmu|^\beta\in E^{\alpha,\frac{2 p + r - q}{2 + r - q}}(\R^N)$.
By the Coulomb--Sobolev embedding (theorem~\ref{theoremEmbedding}), the integral
\begin{equation}\label{s-regularity}
  \int_{\R^N} \abs{u_\mmu}^{s}
\end{equation}
is bounded uniformly with respect to \(\mmu > 0\). We conclude by Lebesgue's monotone convergence theorem that \(u \in L^s (\R^N)\).
\end{proofclaim}

\begin{claim}
\label{claimLr}
\(u \in L^r (\R^N)\) for every \(r \in [1, \infty)\) such that
\[
  \frac{1}{r} \le \max \Bigl\{\frac{1}{2} - \frac{1}{N}, \frac{2 + \alpha}{2(2 p + \alpha)}\Bigr\}.
\]
\end{claim}
\begin{proofclaim}
This follows by the Coulomb--Sobolev embeddinq (theorem~\ref{theoremEmbedding}) and by iterating claim~\ref{claimIteration}.

Indeed, for $p>\frac{N+\alpha}{N-2}$ we can set
$$r_0:=2\frac{\alpha+2p}{\alpha+2}.$$
Then by claim~\ref{claimIteration}, $u\in L^{r_k}(\R^N)$ for all $k\in\N\cup\{0\}$, where
$$r_{k+1}:=\max\left\{\frac{N}{N-2}\big(r_k-q\big)+\frac{2N}{N-2},r_k+\Big(2\frac{\alpha+2p}{\alpha+2}-q\Big)\right\}.$$
Since $\frac{2N}{N-2}<q<2\frac{\alpha+2p}{\alpha+2}$, it is clear that $r_k\to\infty$ as $k\to\infty$.
\smallskip

On the other hand, for $N\ge 3$, $1<p<\frac{N+\alpha}{N-2}$ and $2\frac{\alpha+2p}{\alpha+2}<q<\frac{2N}{N-2}$,
we can choose
$$r_0:=\frac{2N}{N-2}.$$
Then by claim~\ref{claimIteration}, $u\in L^{r_k}(\R^N)$ for all $k\in\N\cup\{0\}$, where
$$r_{k+1}:=\max\left\{\frac{N}{N-2}\big(r_k-q\big)+\frac{2N}{N-2},r_k-\Big(q-2\frac{\alpha+2p}{\alpha+2}\Big)\right\}=
\frac{N}{N-2}\big(r_k-q\big)+\frac{2N}{N-2},$$
since $q<\frac{2N}{N-2}$. Clearly $r_k\to\infty$ as $k\to\infty$.
\end{proofclaim}

\begin{claim}
\(u \in W^{2,r}_{\mathrm{loc}} (\R^N)\) for every \(r \in [1, \infty)\).
\end{claim}
\begin{proofclaim}
By claim~\ref{claimLr} and the H\"older inequality, one has for every \(r \in [1, \infty)\), \(u \in L^r_{\mathrm{loc}} (\R^N)\). It follows that \(I_\alpha \ast \abs{u}^p \in L^\infty(\R^N)\), and therefore for every \(r \in [1, \infty)\), \((I_\alpha \ast \abs{u}^p) \abs{u}^{p - 2} u \in L^r_{\mathrm{loc}} (\R^N)\). Since we also have for every \(r \in [1, \infty)\), \(\abs{u}^{q - 2} u \in L^r_{\mathrm{loc}}(\R^N)\), we conclude by the classical Calder\'on--Zygmund regularity estimates that for every \(r \in (1, \infty)\), \(u \in W^{2, r}_{\mathrm{loc}} (\R^N)\)\cite{GilbargTrudinger1983}*{chapter 9}.
\end{proofclaim}

The additional H\"older and $C^\infty$--regularity follows from the classical Schauder estimates.
\end{proof}

\begin{remark}
If \(p(N - 2)<N+\alpha\) then regularity of weak solutions can be obtained by a classical bootstrap argument.
Indeed, by the Kato inequality, every weak solution \(u \in E^{\alpha, p} (\R^N)\) of \eqref{main-reg} satisfies
\[
  -\Delta \abs{u} \le -\Delta \abs{u} + V (I_\alpha \ast \abs{u}^p) \abs{u}^{p - 1} \le \abs{u}^{q - 1}\quad\text{in \(\R^N\)}.
\]
By the Coulomb--Sobolev embedding we have \(u \in L^r(\R^N)\) for \(\frac{1}{q} \le \frac{1}{r} \le \frac{1}{2} - \frac{1}{N}\). Then regularity of $u$ follows by iterating standard linear regularity estimates (see \citelist{\cite{GilbargTrudinger1983}*{proof of theorem 8.16}\cite{Trudinger1973}} and also \cite{DiCosmoVanSchaftingen}).
\end{remark}

One of the important consequences of proposition~\ref{proposition:Regularity} is
positivity of nontrivial nonnegative solutions of \eqref{main-reg} in the case $p\ge 2$.

\begin{proposition}[Positivity]\label{proposition:Positivity}
Let \(N \in \N\), \(\alpha \in (0, N)\), \(p\ge 2\) and assumption \eqref{Q_0} holds.
If \(u \in E^{\alpha, p} (\R^N)\setminus\{0\}\) is a nonnegative weak solution of the equation
\eqref{main-reg} then $u(x)>0$ for all $x\in\R^N$.
\end{proposition}

\begin{proof}
Under the assumptions, $u$ satisfies the equation
\[
-\Delta u + Vu=0\quad\text{in \(\R^N\)},
\]
where $V=(I_\alpha \ast u^p) u^{p - 2} - u^{q - 2}$.
By proposition~\ref{proposition:Regularity} and since $p\ge 2$, $V\in L^r(\R^N)$ for all $r>\frac{N}{2}$.
Then $u(x)>0$ for all $x\in\R^N$ by the strong maximum principle
(see \cite{GilbargTrudinger1983}*{theorem 8.19}).
\end{proof}

It is an interesting open question whether equation \eqref{main-reg}
with $p<2$ admits nontrivial nonnegative solutions
which vanish on subsets of $\R^N$.

\subsection{Poho\v zaev identity}

An important feature of equations of type \eqref{sps} is that under mild regularity assumptions
its finite energy solutions satisfy a Poho\v zaev type integral identity.

\begin{proposition}[Poho\v zaev identity]
\label{theoremPohozhaev}
Let \(N \in \N\), $\alpha\in(0,N)$, $p>1$, $q>1$ and $\mu>0$.
Assume that \(u \in E^{\alpha, p} (\R^N)\cap L^q(\R^N)\) is a weak solution of the equation
\begin{equation}\label{main-Poh}
  -\Delta u + \bigl(I_\alpha \ast \abs{u}^p\bigr) \abs{u}^{p - 2} u = \mu\abs{u}^{q - 2} u\quad\text{in \(\R^N\).}
\end{equation}
If \(\nabla u \in E^{\alpha, p}_\loc (\R^N)\cap L^q_\loc(\R^N)\)
then
\begin{equation}
\label{eqPohozaev}
 \frac{N - 2}{2} \int_{\R^N} \abs{\nabla u}^2 + \frac{N + \alpha}{2p} \int_{\R^N} \bigl|I_{\alpha/2}*\abs{u}^p\bigr|^2
 = \frac{N\mu}{q} \int_{\R^N} \abs{u}^q.
\end{equation}
\end{proposition}

By \(\nabla u \in E^{\alpha, p}_{\mathrm{loc}} (\R^N) \cap L^q_\loc (\R^N)\) we mean that for every
\(\psi \in C^{\infty}_c (\R^N; \R^N)\), \(\psi \cdot \nabla u \in E^{\alpha, p}(\R^N) \cap L^q (\R^N)\).

In particular, if assumption \eqref{Q_0} holds then \(u \in C^2(\R^N)\) by proposition~\ref{proposition:Regularity},
and therefore \eqref{eqPohozaev} is valid.

For $N=3$, $\alpha=2$, $p=2$ a Poho\v zaev type identity for Schr\"odinger--Poisson--Slater type equations is well known \citelist{\cite{DAprile-Mugnai}\cite{Ruiz-JFA}}.
The proof of proposition \ref{theoremPohozhaev} for the general case is an adaptation of the argument in \cite{MorozVanSchaftingen}*{Proposition 3.1}. The strategy is classical and consists in testing the equation against a suitable cut-off of \(x \cdot \nabla u(x)\)
and integrating by parts, cf. \cite{WillemMinimax}*{appendix B}.
We omit the details.

A direct consequence of Poho\v zaev identity \eqref{eqPohozaev}
is the following nonexistence result for \eqref{sps}.

\begin{proposition}[Nonexistence of solutions]
Let \(N \in \N\), \(\alpha \in (0, N)\) and \(p> 1\). Assume that
\begin{equation}\tag{$\mathcal P$}\label{P}
\begin{split}
\text{either}\quad & \frac{1}{p} > \frac{(N - 2)_+}{N + \alpha} \quad\text{and}\quad
   \frac{1}{q} \not\in \left(\frac{1}{2} - \frac{1}{N}, \frac{1}{2p} + \frac{\alpha}{2Np} \right)\,,\bigskip\\
\text{or}\quad & \frac{1}{p} < \frac{(N - 2)_+}{N + \alpha} \quad\text{and}\quad
   \frac{1}{q} \not\in \left(\frac{1}{2p} + \frac{\alpha}{2Np},\frac{1}{2} - \frac{1}{N} \right)\,.
\end{split}
\end{equation}
Then equation \eqref{sps} has no nontrivial weak solutions $u\in E^{\alpha,p}(\R^N)\cap L^q(\R^N)$ such
that $D u\in E^{\alpha,p}_\loc(\R^N)\cap L^q_\loc(\R^N)$.
\end{proposition}

\begin{proof}
Testing \eqref{sps} against the solution $u$ we obtain a Nehari type identity
\begin{equation*}
\int_{\R^N}|\nabla u|^2+\int_{\R^N}\big|I_{\alpha/2}*\abs{u}^p\big|^2=\int_{\R^N}\abs{u}^q.
\end{equation*}
Combining this with the Poho\v zaev identity \eqref{eqPohozaev} we conclude that the solution $u$ must satisfy the relation
\begin{equation}\label{pohozaev-realtions}
\int_{\R^N}|\nabla u|^2=\frac{\frac{2N}{q}-\frac{N+\alpha}{p}}{N-2-\frac{N+\alpha}{p}}\int_{\R^N}\abs{u}^q,\quad
\int_{\R^N}\big|I_{\alpha/2}*\abs{u}^p\big|^2=\left(1-\frac{\frac{2N}{q}-\frac{N+\alpha}{p}}{N-2-\frac{N+\alpha}{p}}\right)\int_{\R^N}\abs{u}^q.
\end{equation}
The conclusion follows.
\end{proof}

Comparing nonexistence assumption \eqref{P} with the existence range \eqref{Q_0} we observe that there is a gap
between the two assumptions. We are going to show that the existence range \eqref{Q_0} indeed could be extended.

\section{Estimates for radial functions and radial Coulomb--Sobolev embeddings}\label{sect-radial}
In this section we study embedding properties of the subspace of radial functions $E^{\alpha, p}_{\rad} (\R^N)$ into Lebesgue spaces $L^q(\R^N)$ and prove theorem~\ref{theoremRadialSobolev-intro}.

\subsection{Necessary conditions for the radial Coulomb--Sobolev embeddings}

We begin by justifying the necessity of the embedding assumptions \eqref{Q} and \eqref{Qrad} of theorem~\ref{theoremRadialSobolev-intro}.

\begin{lemma}[Criticality of the classical Sobolev exponent]
\label{lemmaRadialExampleSobolev}
Let \(N \ge 1\), $\alpha\in(0,N)$ and $p\ge 1$.
Then there exists a family of radial functions \(u_R \in C^1_c (\R^N)\) such that
if \(\frac{1}{p} \ge \frac{N - 2}{N + \alpha}\), then
\[
  \limsup_{R \to 0} \int_{\R^N} \abs{D u_R}^2 + \bigabs{I_{\alpha/2} \ast \abs{u_R}^p}^2
  < \infty
\]
and for every \(q \in [1, \infty]\),
\[
  \liminf_{R \to 0} \norm{u_R}_{L^q (\R^N)}R^{\frac{N - 2}{2} - \frac{N}{q}} > 0;
\]
and if \(p \ge \frac{N + \alpha}{N - 2}\), then
\[
  \limsup_{R \to \infty} \int_{\R^N} \abs{D u_R}^2 + \abs{I_{\alpha/2} \ast \abs{u_R}^p}^2
  < \infty
\]
and for every \(q \in [1, \infty]\),
\[
  \liminf_{R \to \infty} \norm{u_R}_{L^q (\R^N)} R^{\frac{N - 2}{2} - \frac{N}{q}} > 0.
\]
\end{lemma}

In particular, the classical Sobolev exponent \(q = \frac{2 N}{N - 2}\) is always an extremal exponent
for the radial Coulomb--Sobolev embeddings.

\begin{proof}%
[Proof of lemma~\ref{lemmaRadialExampleSobolev}]
Choose a radial function \(u \in C^1_c (\R^N)\setminus\{0\}\).
For $R>0$ and \(x \in \R^N\), set \(u_R (x) = R^{-\frac{N - 2}{2}} u \big(\frac{x}{R}\big)\).
Then we compute
\[
  \int_{\R^N} \abs{D u_R}^2 + \abs{I_{\alpha/2} \ast \abs{u_R}^p}^2
  = \int_{\R^N} \abs{D u}^2 + R^{N + \alpha-p (N - 2)}\int_{\R^N} \abs{I_{\alpha/2} \ast \abs{u}^p}^2
\]
and
\[
  \norm{u_R}_{L^q (\R^N)} = R^{\frac{N}{q} - \frac{N - 2}{2}} \norm{u}_{L^q (\R^N)}.
\]
The conclusion follows.
\end{proof}

\begin{lemma}[Criticality of the Coulomb--Sobolev exponents]
\label{lemmaRadialExampleLocal}
Let \(N \ge 2\), $\alpha\in(0,N)$ and $p\ge 1$.
Then there exists a family of radial functions
\(u_R \in C^1_c (B_{2R} \setminus B_R)\) such that if \(p(N-2) < N + \alpha\),
\[
  \limsup_{R \to \infty} \int_{\R^N} \abs{D u_R}^2 + \abs{I_{\alpha/2} \ast \abs{u_R}^p}^2
  < \infty
\]
and, for every \(q \in [1, \infty]\),
\begin{align*}
  \liminf_{R \to \infty} \norm{u_R}_{L^q (B_{2R} \setminus B_R)} R^{\frac{3 N + \alpha - 4}{2 (p + 2)}-\frac{2 p (N - 1) + N - \alpha}{q (p + 2)}} &> 0 & & \text{if \(\alpha > 1\)},\\
  \liminf_{R \to \infty} \norm{u_R}_{L^q (B_{2R} \setminus B_R)} R^{\frac{N - 1}{p + 2} (\frac{3}{2} - \frac{2 p + 1}{q})} \abs{\log R}^{\frac{1}{2 p}} & > 0 & & \text{if \(\alpha = 1\)},\\
    \liminf_{R \to \infty} \norm{u_R}_{L^q (B_{2R} \setminus B_R)} R^{\frac{N - 1}{p + \alpha + 1} (\frac{\alpha + 2}{2} - \frac{2 p + \alpha}{q})} & > 0 & & \text{if \(\alpha < 1\)},
\end{align*}
and if \(p(N-2) > N + \alpha\), then
\[
  \limsup_{R \to 0} \int_{\R^N} \abs{D u_R}^2 + \abs{I_{\alpha/2} \ast \abs{u_R}^p}^2
  < \infty
\]
and, for every \(q \in [1, \infty]\),
\begin{align*}
  \liminf_{R \to 0} \norm{u_R}_{L^q (B_{2R} \setminus B_R)} R^{\frac{3 N + \alpha - 4}{2 (p + 2)}-\frac{2 p (N - 1) + N - \alpha}{q (p + 2)}} & > 0 & & \text{if \(\alpha > 1\)},\\
  \liminf_{R \to 0} \norm{u_R}_{L^q (B_{2R} \setminus B_R)} R^{\frac{N - 1}{p + 2} (\frac{3}{2} - \frac{2 p + 1}{q})} \abs{\log R}^{\frac{1}{2 p}} & > 0 & & \text{if \(\alpha = 1\)},\\
  \liminf_{R \to 0} \norm{u_R}_{L^q (B_{2R} \setminus B_R)} R^{\frac{N - 1}{p + \alpha + 1} (\frac{\alpha + 2}{2} - \frac{2 p + \alpha}{q})} & > 0 & & \text{if \(\alpha < 1\)}.
\end{align*}
In particular, when $\alpha<1$ the embedding of $E_{\rad}^{\alpha,p}(\R^N)$ into \(L^q (\R^N)\) is noncompact for $q=2\frac{2p+\alpha}{2+\alpha}.$

\end{lemma}
We point out that when  \(\alpha > 1\) and \(p(N-2) < N + \alpha\) there is no embedding of $E^{\alpha,p}_\rad(\R^N)$ into \(L^q (\R^N)\) if
\[
  q < 2 \frac{2 p (N - 1) + N - \alpha}{3 N + \alpha - 4},
\]
from which we recover in the case where \(N = 3\), \(\alpha = 2\) and \(p = 2\), the condition \(q < \frac{18}{7}\) of Ruiz \cite{Ruiz-ARMA}*{theorem~1.2}.
When \(\alpha \le 1\) there is no embedding of $E^{\alpha,p}_\rad(\R^N)$ into \(L^q (\R^N)\) if
\[
  q < 2 \frac{2 p  + \alpha}{\alpha + 2},
\]
which coincides with the necessary condition for the embedding of $E^{\alpha,p}(\R^N)$ into \(L^q (\R^N)\) (Theorem~\ref{theoremEmbedding}).
The case $p(N-2)=N+\alpha$ is covered by lemma~\ref{lemmaRadialExampleSobolev}.

In order to prove lemma~\ref{lemmaRadialExampleLocal}, we study the form of the Riesz integral kernel on radial functions, where it reduces to less singular kernel.

\begin{lemma}
\label{lemmaRadialKernel}
If \(N \ge 2\) and $\alpha\in(0,N)$, then for every measurable function \(f : [0, \infty) \to [0, \infty)\),
\[
  \iint_{\R^N} \frac{f (\abs{x}) f (\abs{y})}{\abs{x - y}^{N - \alpha}}\dif x \dif y
  =\int_0^\infty \int_0^\infty f (r) K^R_{\alpha, N}(r,s) f (s) r^{N - 1} s^{N - 1} \dif r \dif s,
\]
where the kernel \(K^R_{\alpha,N} : [0, \infty) \times [0, \infty) \to [0, \infty)\) is defined for \(r, s \in [0, \infty) \times [0, \infty)\) by
\[
  K^R_{\alpha, N} (r, s) =
  C_N \int_0^1 \frac{z^{\frac{N - 3}{2}} (1 - z)^{\frac{N - 3}{2}}}{((s + r)^2- 4 rs z)^\frac{N - \alpha}{2}}\dif z
  = C'_N \frac{F (\frac{N - \alpha}{2}, \frac{N - 1}{2}, N - 1, \frac{4 rs}{(s + r)^2})}{(s + r)^{N - \alpha}}.
\]
Moreover there exists \(M > 0\) such that
\[
  K^R_{\alpha, N} (r, s) \le M H_\alpha (r, s).
\]
where
\[
  H_\alpha (r, s) = \left\{\begin{aligned}
       &\frac{1}{(r s)^\frac{N - 1}{2}} \frac{1}{|r - s|^{1 - \alpha}}& &\text{if \(\alpha < 1\)},\\
       & \frac{1}{(r s)^\frac{N - 1}{2}} \ln \dfrac{2\abs{r + s}}{\abs{r - s}} & &\text{if \(\alpha = 1\)},\\
       &\frac{1}{(r s)^\frac{N - \alpha}{2}} && \text{if \(\alpha > 1\)}.
    \end{aligned}\right.
\]
\end{lemma}

In this statement, \(F\) denotes the classical hypergeometric function (see for example \cite{Lebedev1965}*{chapter 9}).
Similar estimates were previously obtained in \citelist{\cite{Rubin-1982}\cite{Duoandikoetxea13}\cite{Thim2016}}.

When \(N = 3\), the kernel \(K^R_{\alpha, N} (r, s)\) is particularly easy to compute:
\[
  K^R_{\alpha, 3} (r, s)= \frac{C_3}{2 rs}
    \begin{cases}
       \dfrac{1}{1 - \alpha} \Bigl(\dfrac{1}{\abs{r - s}^{1 - \alpha}} - \dfrac{1}{\abs{r + s}^{1 - \alpha}} \Bigr) & \text{if \(\alpha < 1\)},\\
        \ln \dfrac{\abs{r + s}}{\abs{r - s}} & \text{if \(\alpha = 1\)},\\
       \dfrac{1}{\alpha - 1} \bigl(\abs{r + s}^{\alpha - 1} - \abs{r - s}^{\alpha - 1} \bigr) & \text{if \(\alpha > 1\)}.
    \end{cases}
\]
In particular, when \(N = 3\) and \(\alpha = 2\), we recover \citelist{\cite{LiebLoss2001}*{proof of theorem 9.7}\cite{Ruiz-ARMA}*{p. 359}}
\[
  K^R_{2, 3} (r, s)
  = \frac{C_3}{2 rs} \bigl(\abs{r + s} - \abs{r - s}\bigr) = \frac{C_3 \min (r, s)}{rs} .
\]

\begin{proof}[Proof of lemma~\ref{lemmaRadialKernel}]
By writing the integral in spherical coordinates, we have
\[
  \int_{\R^N} \int_{\R^N} \frac{f (\abs{x}) f (\abs{y})}{\abs{x - y}^{N - \alpha}}\dif x \dif y
  =\int_0^\infty \int_0^\infty f (r) \Bigl(\int_{\Sset^{N - 1}} \int_{\Sset^{N - 1}}\frac{1}{\abs{r u - r v}^{N - \alpha}}\dif u \dif v\Bigr) f (s) r^{N - 1} s^{N - 1} \dif r \dif s.
\]
By writing the spherical integral in azimuthal coordinates and by using the trigonometric identities \(\sin \theta = 2 \sin \frac{\theta}{2} \cos \frac{\theta}{2}\) and \(\cos \theta = 2\cos^2 \frac{\theta}{2}-1\), we obtain
\[
\begin{split}
  \int_{\Sset^{N - 1}} \int_{\Sset^{N - 1}}\frac{1}{\abs{r u - r v}^{N - \alpha}}
  & = C \int_0^\pi \frac{(\sin \theta)^{N - 2}}{(s^2 + r^2 - 2 r s \cos \theta)^\frac{N - \alpha}{2}} \dif \theta\\
  & = C \int_0^\pi \frac{2^{N - 2} (\sin \frac{\theta}{2})^{N - 2}(\cos \frac{\theta}{2})^{N - 2}}{((s + r)^2 - 4 r s (\cos \frac{\theta}{2})^2)^\frac{N - \alpha}{2}} \dif \theta.
\end{split}
\]
By a change of variables \(z = (\cos \frac{\theta}{2})^2\), this gives
\[
  \int_{\Sset^{N - 1}} \int_{\Sset^{N - 1}}\frac{1}{\abs{r u - r v}^{N - \alpha}}
  = C \int_0^1 \frac{2^{N - 2} z^\frac{N - 3}{2} (1 - z)^\frac{N - 3}{2}}{((s + r)^2 - 4 r s z)^\frac{N - \alpha}{2}} \dif z.
\]

In order to prove the bounds, we rewrite the kernel \(K^R_{\alpha, N}\) for every \(r, s \in (0, \infty)\) as
\[
K^R_{\alpha, N} (r, s) = C \int_0^\pi \frac{2^{N - 2} (\sin \frac{\theta}{2})^{N - 2}(\cos \frac{\theta}{2})^{N - 2}}{((r - s)^2 + 4 r s (\sin \frac{\theta}{2})^2)^\frac{N - \alpha}{2}} \dif \theta.
\]
If \(\alpha \ge 2\), we have
\[
  \int_0^\pi \frac{2^{N - 2} (\sin \frac{\theta}{2})^{N - 2}(\cos \frac{\theta}{2})^{N - 2}}{((r - s)^2 + 4 r s (\sin \frac{\theta}{2})^2)^\frac{N - \alpha}{2}} \dif \theta
  \le \int_0^\pi \frac{2^{\alpha - 2} (\sin \frac{\theta}{2})^{\alpha-2}}{(rs)^{\frac{N - \alpha}{2}} }\dif \theta \le \frac{2^{\alpha-2}\pi}{(rs)^{\frac{N - \alpha}{2}}}.
\]
If \(\alpha < 2\), we bound
\[
  \int_0^\pi \frac{2^{N - 2} (\sin \frac{\theta}{2})^{N - 2}(\cos \frac{\theta}{2})^{N - 2}}{((r - s)^2 + 4 r s (\sin \frac{\theta}{2})^2)^\frac{N - \alpha}{2}} \dif \theta
  \le \frac{1}{(rs)^\frac{N - 2}{2}} \int_0^\pi \frac{1}{((r - s)^2 + 4 r s (\sin \frac{\theta}{2})^2)^\frac{2 - \alpha}{2}}\dif \theta.
\]
Since for every \(\theta \in [0, \pi]\) we have \(\sin \frac{\theta}{2} \ge \frac{\theta}{\pi}\) and for every \(a, b \in \R\) we have \(a^2 + b^2 \ge (\abs{a} + \abs{b})^2/2\), we deduce that
\[
  \int_0^\pi \frac{1}{((r - s)^2 + 4 r s (\sin \frac{\theta}{2})^2)^\frac{2 - \alpha}{2}}\dif \theta
  \le \int_0^\pi \frac{2^{1 - \frac{\alpha}{2}}} {(\abs{r - s} + 2 \sqrt{rs} \theta/\pi)^{2 - \alpha}}\dif \theta.
\]
and the conclusion follows.
\end{proof}

\begin{proof}[Proof of lemma~\ref{lemmaRadialExampleLocal}]
Observe that if we choose a radial function \(u \in C^1_c (B_1 \setminus \{0\})\setminus \{0\}\) and define \(v_{R} \in C^1_c (\R^N)\) for \(x \in \R^N\), $R>0$ and $\gamma\in\R$ by
\[
  v_{R} (\abs{x}) = u \Bigl(\frac{\abs{x} - R}{R^\gamma} \Bigr),
\]
then, taking into account lemma~\ref{lemmaRadialKernel},
as $R^{\gamma-1}\to 0$ we have
\begin{gather*}
\int_{\R^N} \abs{v_R}^q
  =O\big(R^{N-1+\gamma}\big),\\
  \int_{\R^N} \abs{D v_R}^2
  =O\big(R^{N-1-\gamma}\big),\\
    \int_{\R^N} \bigl(I_{\alpha/2} \ast \abs{v_R}^p\bigr)^2
  \le \begin{cases}
       O\big(R^{N-2+\alpha+2\gamma}\big), & \text{if \(\alpha > 1\)},\\
            O\big(R^{N-1+2\gamma}\big)\ln(R^{1-\gamma}), & \text{if \(\alpha = 1\)},\,  \\
       O\big(R^{N-1+\gamma(1+\alpha)}\big), & \text{if \(\alpha < 1\)}.
    \end{cases}
\end{gather*}

To construct the required family of functions, when \(\alpha > 1\) choose a radial function \(u \in C^1_c (B_1 \setminus \{0\})\setminus\{0\}\) and we define \(u_{R} \in C^1_c (\R^N)\) for \(x \in \R^N\) by
\begin{equation}\label{uR-radial}
  u_{R} (\abs{x}) = \frac{1}{R^\frac{3 N + \alpha - 4}{2(p + 2)}} u \Bigl(\frac{\abs{x} - R}{R^{\frac{p (N - 1) - (N + \alpha) + 2}{2 + p}}} \Bigr).
\end{equation}
When \(\alpha = 1\), we take similarly
\[
  u_{R} (|x|) = \frac{1}{R^\frac{3(N - 1)}{2(p + 2)} \abs{\log R}^{\frac{1}{2 p}}}  u \Bigl(\frac{\abs{x} - R}{R^{\frac{(N - 1)(p - 1)}{p + 2}}} \Bigr),
\]
whereas when \(\alpha < 1\), we set
\[
  u_{R} (|x|) = \frac{1}{R^\frac{(N - 1)(\alpha + 2)}{2(p + 1 + \alpha)}} u \Bigl(\frac{\abs{x} - R}{R^{\frac{(N - 1)(p - 1)}{p + 1 + \alpha}}} \Bigr).
\]
 The choice of $\gamma$ in all the above cases ensures that $R^{\gamma-1}\to 0,$ and the support of $u_R$ is contained in the annulus with radii $R, R+R^\gamma$ which is contained in $B_{2R}\setminus B_R
,$ as $R^{\gamma-1}$ is small enough. In particular, since $\gamma <1$ (resp.$\gamma>1$) as $p(N-2)<N+\alpha$ (resp. as $p(N-2)>N+\alpha$), we take $R\rightarrow\infty$ (resp. $R\rightarrow 0$).
The required conclusion now follows by direct computation.
\end{proof}

Finally, we are going to show optimality of the assumption \eqref{Qrad} for the continuous embedding
$E^{\alpha,p}_\rad(\R^N)$ into $L^q(\R^N)$.

\begin{lemma}[No limiting radial estimate]
\label{lemmaNoRadialCritical}
Let $N\ge 2$, \(\alpha > 1\) and $p\ge 1$.
If \(p(N-2) \ne N + \alpha\) then there exists a sequence of radial functions $(u_n)_{n\in\N}$ in $C^1_c (\R^N)$
such that
\[
  \sup_{n \to \infty} \int_{\R^N} \abs{D u_n}^2 + \abs{I_{\alpha/2} \ast \abs{u_n}^p}^2
  < \infty
\]
and
\[
  \lim_{n \to \infty}
  \int_{\R^N} \abs{u_n}^{2 \frac{2 p (N - 1) + N - \alpha}{3 N + \alpha - 4}} = \infty.
\]
\end{lemma}

When \(N = 3\), \(\alpha = 2\) and \(p = 2\), lemma~\ref{lemmaNoRadialCritical} was conjectured by D.\thinspace{} Ruiz \cite{Ruiz-ARMA}*{remark 4.1}.

\begin{proof}[Proof of lemma~\ref{lemmaNoRadialCritical}]
Let $q=2 \frac{2 p (N - 1) + N - \alpha}{3 N + \alpha - 4}$ and \(p(N-2) < N + \alpha\).
Define for \(R > 1\),
\[
  u_{R, k} = \sum_{i = 1}^k u_{R^i},
\]
where \(u_R\) is the family of functions constructed in \eqref{uR-radial} in the proof of lemma~\ref{lemmaRadialExampleLocal}.
Since \(p(N-2) < N + \alpha\), we have
\begin{gather}
\lim_{R \to \infty}\int_{\R^N} \abs{I_{\alpha/2} \ast \abs{u_{R,k}}^p}^2 \le c  k\lim_{R \to \infty}\int_{\R^N} \abs{I_{\alpha/2} \ast \abs{u_{R}}^p}^2,\label{e-Coulomb-k}\\
\lim_{R \to \infty} \int_{\R^N} \abs{D u_{R, k}}^2 = k \lim_{R \to \infty} \int_{\R^N} \abs{D u_R}^2,\nonumber\\
\lim_{R \to \infty} \int_{\R^N} \abs{u_{R, k}}^q = k \lim_{R \to \infty} \int_{\R^N} \abs{u_R}^q.\nonumber
\end{gather}
To deduce \eqref{e-Coulomb-k}, we observe that for $\alpha>1$ by lemma \ref{lemmaRadialExampleLocal},
$$\int_{\R^N} \abs{I_{\alpha/2} \ast \abs{u_{R}}^p}^2=O(1),\qquad
\int_{\R^N}\abs{u_{R}}^p=O(R^\frac{N-\alpha}{2}).$$
Then for a fixed $k\in\N$ and for sufficiently large $R>1$ we have
\begin{multline*}
  \int_{\R^N} \abs{I_{\alpha/2} \ast \abs{u_{R,k}}^p}^2
  \le ck \int_{\R^N} \abs{I_{\alpha/2} \ast \abs{u_R}^p}^2+
  2\sum_{\substack{i,j=1\\i>j}}^k
  \iint_{\R^N} I_\alpha(x-y)|u_{R^i}(x)|^p|u_{R^j}(y)|^p\dif x \dif y\\
  \le ck \int_{\R^N} \abs{I_{\alpha/2} \ast \abs{u_R}^p}^2+
  c_1\sum_{\substack{i,j=1\\i>j}}^k\left(
  \frac{1}{(R^i-R^j)^{N-\alpha}}\int_{\R^N}\abs{u_{R^i}}^p\int_{\R^N}\abs{u_{R^j}}^p\right)\\
  \le ck \int_{\R^N} \abs{I_{\alpha/2} \ast \abs{u_R}^p}^2+
  c_2\sum_{\substack{i,j=1\\i> j}}^k\frac{(R^i R^j)^\frac{N-\alpha}{2}}{(R^i-R^j)^{N-\alpha}},
\end{multline*}
and
$\lim_{R\to\infty}\sum_{i,j=1,\,i> j}^k\frac{(R^{i+j})^\frac{N-\alpha}{2}}{(R^i-R^j)^{N-\alpha}}=
\lim_{R\to\infty}\sum_{i,j=1,\,i> j}^k\Big(\frac{R^{-\frac{i-j}{2}}}{1-R^{-(i-j)}}\Big)^{N-\alpha}=0$.

If we define
\[
  v_{R,k}(x) = \frac{1}{k^\frac{\alpha + 2}{2(N + \alpha - p (N - 2))}} u_{R,k} \Bigl(\frac{x}{k^{\frac{p - 1}{N + \alpha - p (N - 2)}}} \Bigr),
\]
we have
\begin{gather}
  \lim_{R \to \infty}\int_{\R^N} \abs{I_{\alpha/2} \ast \abs{v_{R,k}}^p}^2 \le c  \lim_{R \to \infty}\int_{\R^N} \abs{I_{\alpha/2} \ast \abs{u_{R}}^p}^2,\nonumber\\
  \lim_{R \to \infty} \int_{\R^N} \abs{D v_{R, k}}^2 = \lim_{R \to \infty} \int_{\R^N} \abs{D u_R}^2,\nonumber\\
  \lim_{R \to \infty} \int_{\R^N} \abs{v_{R, k}}^q = k^\frac{2 (\alpha - 1)}{3 N + \alpha - 4} \lim_{R \to \infty} \int_{\R^N} \abs{u_R}^q.\label{e-qnormgrowth}
\end{gather}
Since \(\alpha > 1\), the conclusion follows from \eqref{e-qnormgrowth} by a diagonal argument.
The case \(p(N-2) > N + \alpha\) is similar, by letting $R\rightarrow 0.$
\end{proof}

\subsection{Radial estimates and proof of theorem \ref{theoremRadialSobolev-intro}}
We establish additional estimates for radial functions.

\begin{theorem}[Uniform decay estimates for radial functions]\label{theoremRadialEstimates}
Let \(N \ge 2\), $\alpha\in(0,N)$ and $p\ge 1$.
The estimate
\begin{equation}\label{uniformradialestimate}
 \abs{u (x)} \le \frac{C}{\abs{x}^\beta} \Bigl(\int_{\R^N} \abs{Du}^2 \Bigr)^\frac{\theta}{2}
\Bigl( \int_{\R^N} \abs{I_{\alpha/2} \ast \abs{u}^p}^2\Bigr)^\frac{1 - \theta}{2 p}
\end{equation}
is satisfied for every function \(u \in E^{\alpha, p}_\rad(\R^N)\) for almost every \(x \in \R^N\) if and only if
\[
 \frac{1}{1 + \frac{p}{1 + \min (1, \alpha)}} \le \theta \le 1
\]
and
\[
  \beta = \theta \frac{N - 2}{2} + (1 - \theta) \frac{N + \alpha}{2 p} > 0.
\]
\end{theorem}

The assumption \(\frac{1}{1 + \frac{p}{1 + \min (1, \alpha)}}\le \theta \le 1\) implies immediately that \(\beta > 0\), unless \(\theta = 1\) and \(N = 2\).

When \(\theta = 1\) and \(N > 2\), we recover the classical estimate of Ni \cite{Ni1982}
\begin{equation}
\label{ineqNi}
  \abs{u (x)} \le \frac{C}{\abs{x}^{\frac{N - 2}{2}}} \Bigl(\int_{\R^N} \abs{D u}^2 \Bigr)^\frac{1}{2}
\end{equation}
(see also \citelist{\cite{Strauss1977}*{radial lemma 1}\cite{SuWangWillem2007}*{lemma 1}}).
The other cases are new. In particular, when \(\theta = \frac{1}{1 + \frac{p}{2}}\), we have the estimate
\begin{equation}
\label{eqRadialLargeAlpha}
  \abs{u (x)} \le \frac{C}{\abs{x}^{\frac{3 N + \alpha - 4}{2 (p + 2)}}} \Bigl(\int_{\R^N} \abs{D u}^2 \Bigr)^\frac{1}{p + 2} \Bigl(\int_{\R^N} \abs{I_{\alpha/2} \ast \abs{u}^p}^2\Bigr)^\frac{1}{2 (p + 2)},
\end{equation}
whereas when \(\theta = \frac{1}{1 + \frac{p}{1 + \alpha}}\), we have
\begin{equation}
\label{eqRadialSmallAlpha}
  \abs{u (x)}
  \le \frac{C}{\abs{x}^\frac{(N - 1) (\alpha + 2)}{2 (p + 1 + \alpha)}}
  \Bigl(\int_{\R^N} \abs{D u}^2 \Bigr)^\frac{1 + \alpha}{2 (p + 1 + \alpha)}
  \Bigl(\int_{\R^N} \abs{I_{\alpha/2} \ast \abs{u}^p}^2 \Bigr)^\frac{1}{2 (p + 1 + \alpha)}.
  \end{equation}
When \(N \ge 3\) it is possible to deduce the estimates of theorem~\ref{theoremRadialEstimates} from the endpoint estimates \eqref{ineqNi}, \eqref{eqRadialLargeAlpha} and \eqref{eqRadialSmallAlpha}, in order to cover the case \(N = 2\) we prove directly the whole scale of estimates.

When $N\ge 3$, the range of admissible decay rates \(\beta\) can de rewritten explicitly as follows:
\smallskip
\begin{itemize}
\item either \(\alpha \ge 1\), \(\frac{1}{p} \ge \frac{N - 2}{N + \alpha}\) and
  \(
    \frac{3 N + \alpha - 4}{2 (p + 2)} \ge \beta \ge \frac{N - 2}{2},
  \)\smallskip

\item or \(\alpha \ge 1\), \(\frac{1}{p} \le \frac{N - 2}{N + \alpha}\) and
  \(
    \frac{3 N + \alpha - 4}{2 (p + 2)} \le \beta \le \frac{N - 2}{2},
  \)\smallskip

\item or \(\alpha \le 1\), \(\frac{1}{p} \ge \frac{N - 2}{N + \alpha}\) and $
    \frac{(N - 1) (\alpha + 2)}{2 (p + 1 + \alpha)} \ge \beta \ge \frac{N - 2}{2},
 \)\smallskip

\item or \(\alpha \le 1\), \(\frac{1}{p} \le \frac{N - 2}{N + \alpha}\) and
  \(
    \frac{(N - 1) (\alpha + 2)}{2 (p + 1 + \alpha)} \le \beta \le \frac{N - 2}{2};
  \)\smallskip
\end{itemize}
whereas when \(N = 2\), we have:
\smallskip
\begin{itemize}
\item
either \(\alpha \ge 1\) and
  \(
    \frac{2 + \alpha}{2 (p + 2)} \ge \beta > 0,
  \)
\smallskip

\item
or \(\alpha \le 1\) and \(
    \frac{2+\alpha}{2 (p + 1 + \alpha)} \ge \beta > 0\).
\smallskip
\end{itemize}

\begin{proof}[Proof of theorem~\ref{theoremRadialEstimates}]
We split the proof into three separate claims.

\setcounter{claim}{0}
\begin{claim}
Estimate \eqref{uniformradialestimate} holds if \(\beta = \theta \frac{N - 2}{2} + (1 - \theta) \frac{N + \alpha}{p} > 0\) and \(\theta \ge \frac{1}{1 + \frac{p}{2}}\).
\end{claim}
\begin{proofclaim}
Let \(\gamma > 0\).
For almost every \(r > 0\), we have
\begin{equation}
\label{eqRadialIntegration}
  \abs{u (r)}^\gamma \le \frac{1}{\gamma} \int_r^\infty \abs{u' (s)} \abs{u (s)}^{\gamma - 1} \dif s.
\end{equation}
If \(\gamma \le 1 + \frac{p}{2}\), then by the generalized H\"older inequality
\[
  \abs{u (r)}^\gamma
  \le \frac{1}{\gamma} \Bigl(\int_r^\infty \abs{u' (s)}^2 s^{N - 1} \dif s\Bigr)^\frac{1}{2}
  \Bigl(\int_r^\infty \frac{\abs{u (s)}^p}{s^{\frac{N - \alpha}{2} + \delta}} s^{N - 1} \dif s\Bigr)^\frac{\gamma - 1}{p}\Bigl(\int_r^\infty \frac{\dif s}{s^{1 + \delta}} \Bigr)^{\frac{1}{2}-\frac{\gamma - 1}{p}},
\]
where
\[
  \delta = N - 2 + \frac{N + \alpha}{p}(\gamma - 1).
\]
Since \(\delta > 0\), we deduce in view of proposition~\ref{propositionRuizPowerExterior} and the explicit computation of the last integral that
\[
  \abs{u (x)}^\gamma \le \frac{C}{\abs{x}^\frac{\delta}{2}} \Bigl(\int_{\R^N} \abs{D u}^2 \Bigr)^\frac{1}{2} \Bigl(\int_{\R^N} \abs{I_\alpha \ast \abs{u}^p}^2 \Bigr)^\frac{\gamma - 1}{2 p}.
\]
Since \(\theta \ge \frac{1}{1 + \frac{p}{2}}\), we conclude by taking \(\gamma = \frac{1}{\theta}\) so that in particular \(\delta = \frac{\beta}{2 \theta}\).
\end{proofclaim}

\begin{claim}
\label{claimSmallAlpha}
Estimate \eqref{uniformradialestimate} holds if \(\beta = \theta \frac{N - 2}{2} + (1 - \theta) \frac{N + \alpha}{2 p} > 0\) and \(\theta \ge \frac{1}{1 + \frac{p}{1 + \alpha}}\).
\end{claim}
\begin{proofclaim}
We start from \eqref{eqRadialIntegration} with \(1 \le \gamma < 2 \frac{p + \alpha + 1}{2 + \alpha}\) and apply the generalized H\"older inequality to obtain
\[
  \abs{u (r)}^\gamma
  \le \frac{1}{\gamma} \Bigl(\int_r^\infty \abs{u' (s)}^2 s^{N - 1} \dif s\Bigr)^\frac{1}{2}
  \Bigl(\int_r^\infty \abs{u (s)}^{2\frac{2 p + \alpha}{2 + \alpha}} s^{N - 1} \dif s\Bigr)^\frac{(\gamma - 1)(2 + \alpha)}{2 (2p + \alpha)}\Bigl(\int_r^\infty \frac{\dif s}{s^{1 + \delta}} \Bigr)^{\frac{1}{2}-\frac{(\gamma - 1)(2 + \alpha)}{2 (2p + \alpha)}},
\]
where
\[
  \delta = \frac{2 (N - 1)}{1 - \frac{(\gamma - 1)(2 + \alpha)}{2p + \alpha}} - N.
\]
and therefore
\[
 \abs{u (r)}^\gamma
  \le \frac{C}{\abs{x}^{\frac{N - 2}{2} + \frac{(\gamma - 1)N (2 + \alpha)}{2 p + \alpha}}} \Bigl(\int_r^\infty \abs{u' (s)}^2 s^{N - 1} \dif s\Bigr)^\frac{1}{2}
  \Bigl(\int_r^\infty \abs{u (s)}^{2\frac{2 p + \alpha}{2 + \alpha}} s^{N - 1} \dif s\Bigr)^\frac{(\gamma - 1)(2 + \alpha)}{2 (2p + \alpha)}.
\]
We observe that this inequality also holds when \(\gamma = 2 \frac{p + \alpha + 1}{2 + \alpha}\)
by the Cauchy--Schwarz inequality and the fact that \(s \ge r\) in the integrals.

We deduce from proposition~\ref{propositionRieszSobolevinterpolation} that if \(\delta > 0\),
\[
  \abs{u (x)}^\gamma \le \frac{C'}{\abs{x}^{\frac{N - 2}{2} + \frac{(\gamma - 1)N (2 + \alpha)}{2 p + \alpha}}} \Bigl(\int_{\R^N} \abs{D u}^2 \Bigr)^{\frac{2 p + \gamma \alpha}{2(2 p + \alpha)}} \Bigl(\int_{\R^N} \abs{I_\alpha \ast \abs{u}^p}^2 \Bigr)^\frac{\gamma - 1}{2 p + \alpha}.
\]
We take
\[
 \gamma = \frac{2 p}{\theta 2 p - (1 - \theta) \alpha}.
\]
It is clear that \(\gamma \ge 1\).
Moreover, since \(\theta > \frac{1}{1 + \frac{2 p}{\alpha + 1}}\), we have \(\gamma < 2 \frac{p + \alpha + 1}{2 + \alpha}\), so that the conclusion follows.
\end{proofclaim}

\begin{claim}
The assumptions of theorem~\ref{theoremRadialEstimates} are necessary for estimate \eqref{uniformradialestimate} to hold.
\end{claim}

\begin{proofclaim}
First we remark that when \(N = 2\), the estimate cannot hold if \(\beta = 0\): indeed if \(u \in C^1 (\R^2 \setminus \{0\})\) is radial, \(\supp u\) is compact and \(u (x) = (\log \abs{x})^\gamma\) for some \(\gamma \in (0, 1/2)\) for \(x\) in a neighbourhood of \(0\), then \(u\in E^{\alpha, p}_\rad (\R^N)\) but \(u\) is unbounded.

The necessity of all other conditions follows directly from the examples of lemmas~\ref{lemmaRadialExampleSobolev} and
~\ref{lemmaRadialExampleLocal}.
\end{proofclaim}

This completes the proof of theorem~\ref{theoremRadialEstimates}.
\end{proof}

The proof of Claim~\ref{claimSmallAlpha} also works when \(\beta = 0\) and \(\theta = \frac{1 + \alpha}{1 + p + \alpha}\) and gives back the proof of Theorem~\ref{DimensionOne}.

Using uniform estimates of theorem~\ref{theoremRadialEstimates}, we are going to prove a weighted version of
the radial embedding of theorem~\ref{theoremRadialSobolev-intro}. The weight will be used later
in order to prove compactness of the embedding.

\begin{proposition}[Weighted radial embedding]\label{propositionRadialSobolevLargeAlpha}
Let \(N \ge 2\), $\alpha\in(0,N)$, $\gamma\in\R$, $p\ge 1$ and $q\ge p$.
If \(p (N - 2) \ne N + \alpha\),
\[
 \frac{1 - \frac{p}{q}}{1 + \frac{p}{1 + \min (1, \alpha)}} < \theta < 1,
\]
and
\[
 \frac{N - \gamma}{q} = \theta \frac{N - 2}{2} + (1 - \theta) \frac{N + \alpha}{2 p},
\]
then for every $u\in E^{\alpha,p}_\rad(\R^N)$,
\[
 \int_{\R^N} \frac{\abs{u (x)}^q}{\abs{x}^\gamma}\dif x \le C \Bigl(\int_{\R^N} \abs{D u}^2 \Bigr)^\frac{\theta}{2} \Bigl(\int_{\R^N} \abs{I_{\alpha/2} \ast \abs{u}^p}^2\Bigr)^\frac{1 - \theta}{2}.
\]
\end{proposition}

If \(\alpha > 1\), the assumptions reduce to
\begin{equation}\tag{$\mathcal Q_{\rad,\gamma}$}\label{Qrad-beta}
\begin{split}
\text{either}\quad & \frac{1}{p} > \frac{N - 2}{N + \alpha} \quad\text{and}\quad
   \frac{N-2}{2 (N - \gamma)} < \frac{1}{q} < \frac{3 N + \alpha - 4}{2(p (2N - 2 - \gamma) - (N - \alpha - 2 \gamma))}\,,\bigskip\\
\text{or}\quad & \frac{1}{p} < \frac{N - 2}{N + \alpha} \quad\text{and}\quad
\frac{N-2}{2 (N - \gamma)} > \frac{1}{q} > \frac{3 N + \alpha - 4}{2(p (2N - 2 - \gamma) - (N - \alpha - 2 \gamma))}\,.
\end{split}
\end{equation}

\begin{proof}
If \(\frac{1}{p} > \frac{N - 2}{N + \alpha}\), then for every \(\theta^* \ge 1/(1 + p/(1 + \min(1, \alpha)))\) and \(\beta^* = \theta^* \frac{N - 2}{2} + (1 - \theta^*) \frac{N + \alpha}{p}\) by Theorem~\ref{theoremRadialEstimates},
we have
\[
 \abs{u (x)} \le \frac{C}{\abs{x}^{\beta^*}} \Bigl(\int_{\R^N} \abs{D u}^2 \Bigr)^\frac{\theta^*}{2}\Bigl( \int_{\R^N} \abs{I_{\alpha/2} \ast \abs{u}^p}^2\Bigr)^\frac{1 - \theta^*}{2 p}.
\]
Since \(1/(1 + p/(1 + \min(1, \alpha))) < \theta < 1\) and \(\frac{N - \gamma}{q} = \theta \frac{N - 2}{2} + (1 - \theta) \frac{N + \alpha}{2 p}\), we can choose \(\theta^* \in (\theta, 1)\) such that
\(\gamma + \beta^* q < N\), and we have thus
\begin{equation}
\label{eqRadialSobolevLargeAlphaInterior}
\begin{split}
  \int_{B_R} \frac{\abs{u (x)}^q}{\abs{x}^\gamma}\dif x
  &\le C^q \Bigl(\int_{\R^N} \abs{D u}^2 \Bigr)^\frac{\theta^* q}{2} \Bigl( \int_{\R^N} \abs{I_{\alpha/2} \ast \abs{u}^p}^2\Bigr)^\frac{(1 - \theta^*)q}{2 p} \int_{B_R} \frac{1}{\abs{x}^{\gamma + \beta^* q}}\dif x\\
  &\le \frac{C'}{R^{\gamma + \beta^* q - N}} \Bigl(\int_{\R^N} \abs{D u}^2 \Bigr)^\frac{\theta^* q}{2} \Bigl( \int_{\R^N} \abs{I_{\alpha/2} \ast \abs{u}^p}^2\Bigr)^\frac{(1 - \theta^*)q}{2 p}.
\end{split}
\end{equation}
On the other hand, by theorem~\ref{theoremRadialEstimates} again we have
\[
 \int_{\R^N \setminus B_R} \frac{\abs{u (x)}^q}{\abs{x}^\gamma}\dif x
 \le C^{q - p}  \Bigl(\int_{\R^N} \abs{D u}^2 \Bigr)^\frac{\theta_* (q - p)}{2} \Bigl(\int_{\R^N} \abs{I_{\alpha/2} \ast \abs{u}^p}^2\Bigr)^\frac{(1 - \theta_*) (q - p)}{2 p}
   \int_{\R^N \setminus B_R} \frac{\abs{u (x)}^p}{\abs{x}^{\gamma + \beta_* (q - p)}} \dif x,
\]
with \(\theta_* = 1/(1 + 1/(1 + \min (1, \alpha)))\). By our assumption \(\theta > (1 - \frac{p}{q})/(1 + 1/(1 + \min (1, \alpha)))\), we have
\[
 \gamma + \beta_* (q - p) = \frac{N - \alpha}{2} + \Bigl(\frac{N + \alpha}{2 p} - \frac{N - 2}{2} \Bigr)\bigl(q \theta - (q - p) \theta_* \bigr) > \frac{N - \alpha}{2}
\]
By our assumption , we have \(\gamma + \beta_* (q - p) > \frac{N - \alpha}{2}\), and thus by the estimate of proposition~\ref{propositionRuizPowerExterior},
\begin{equation}\label{eqRadialSobolevLargeAlphaExterior}
 \int_{\R^N \setminus B_R} \frac{\abs{u (x)}^q}{\abs{x}^\beta}\dif x
 \le \frac{C^{q - p}}{R^{\gamma + \beta_* (q - p) - \frac{N - \alpha}{2}} }  \Bigl(\int_{\R^N} \abs{D u}^2 \Bigr)^\frac{\theta_* (q - p)}{2} \Bigl(\int_{\R^N} \abs{I_{\alpha/2} \ast \abs{u}^p}^2\Bigr)^\frac{(1 - \theta_*)q - \theta_* p}{2 p}.
\end{equation}
The result follows by combining the estimates \eqref{eqRadialSobolevLargeAlphaInterior} and \eqref{eqRadialSobolevLargeAlphaExterior} and optimizing with respect to \(R > 0\).

The case \(\frac{1}{p} < \frac{N - 2}{N + \alpha}\) is similar.
\end{proof}

Theorem~\ref{theoremRadialSobolev-intro} follows directly from the previous results.

\begin{proof} [Proof of theorem~\ref{theoremRadialSobolev-intro}]
The estimates follow from the weighted estimates of Proposition~\ref{propositionRadialSobolevLargeAlpha} in the constant weight case \(\gamma = 0\) and from the nonradial estimates of Theorem~\ref{theoremEmbedding}.

The necessity follows from Lemmas~\ref{lemmaRadialExampleSobolev}, \ref{lemmaRadialExampleLocal} and \ref{lemmaNoRadialCritical}.
\end{proof}

Next we are going to show that away from the end points of the embedding interval the embedding
\(E^{\alpha, p}_\rad (\R^N)\subset L^q (\R^N)\) is compact.

\begin{proposition}[Compactness of the radial embedding]\label{propositionRadialCompactness}
Let $N\geq2$, $\alpha\in (0,N)$ and $p\ge 1$.
Assume that
\[
  \frac{1 - \frac{p}{q}}{1 + \frac{p}{1 + \min (1, \alpha)}} < \frac{\frac{N}{q} - \frac{N + \alpha}{2 p}}{\frac{N - 2}{2} - \frac{N + \alpha}{2 p}} < 1.
\]
Then the embedding \(E^{\alpha, p}_\rad (\R^N)\subset L^q (\R^N)\) is compact.
When $\alpha\neq 1$ the conditions are also necessary for compactness of the embedding.
\end{proposition}

Equivalently, we are assuming that either $\alpha\le 1$ and \eqref{Q_0} holds, or $\alpha>1$ and \eqref{Qrad0} holds.

\begin{proof}[Proof of proposition~\ref{propositionRadialCompactness}]
Let \((u_n)_{n \in \N}\) be a bounded sequence in \(E^{\alpha, p}_\rad (\R^N)\) radial functions is bounded.
By proposition~\ref{propositionElementaryLocalWeakCompactness},
by passing to a subsequence, we can assume that the sequence \((u_n)_{n \in \N}\) converges strongly in \(L^1_{\mathrm{loc}} (\R^N)\) to \(u \in E^{\alpha, p}_\rad (\R^N)\).
In view of the uniform radial estimate of theorem~\ref{theoremRadialEstimates}, the sequence \((u_n)_{n \in \N}\) is bounded in \(L^\infty_\mathrm{loc} (\R^N\setminus\{0\})\) and thus converges strongly in \(L^q_{\mathrm{loc}} (\R^N\setminus\{0\})\) to \(0\).

Moreover, in view of our assumptions, for \(\gamma \in \R\) close to \(0\), the weighted estimates of Proposition~\ref{propositionRadialSobolevLargeAlpha} apply. In particular, we have for \(\gamma > 0\),
\[
 \sup_{n \in \N} R^{-\gamma} \int_{B_R} \abs{u_n}^q
 \le \limsup_{n \in \N} \int_{\R^N} \frac{\abs{u_n (x)}^q}{\abs{x}^\gamma} \dif x
 < \infty,
\]
and therefore
\[
 \limsup_{R \to 0} \sup_{n \in \N} \int_{B_R} \abs{u_n}^q = 0.
\]
Similarly, for \(\gamma < 0\) close enough  to \(0\), we have
\[
 \sup_{n \in \N} R^{\gamma} \int_{\R^N \setminus B_R} \abs{u_n}^q
 \le \limsup_{n \in \N} \int_{\R^N} \frac{\abs{u_n (x)}^q}{\abs{x}^\gamma} \dif x
 < \infty,
\]
and thus
\[
 \limsup_{R \to \infty} \sup_{n \in \N} \int_{\R^N \setminus B_R} \abs{u_n}^q = 0.
\]
It follows that the sequence \((u_n)_{n \in \N}\) converges strongly to \(u\) in \(L^q (\R^N)\).

When $\alpha\neq 1$ as a result of lemma \ref{lemmaRadialExampleSobolev}, lemma \ref{lemmaRadialExampleLocal} and lemma \ref{lemmaNoRadialCritical},
the conditions are also necessary for the compactness of the embedding.
\end{proof}

In the case $\alpha=1$ the compactness of the critical embedding $E^{1,p}_\rad (\R^N)\subset L^{\frac{2}{3}(2p+1)}(\R^N)$ in the case $(N-2)p\neq N+1$ is an open question.

\section{Existence of radially symmetric solutions and proof of theorem~\ref{t-radial-intro}}

Consider the radial minimization problem
\begin{equation}\label{varprob1}
M_{c,\rad}=\inf_{u\in \mathcal A_{c,\rad}} \mathcal{E}_*(u)
\end{equation}
where $E$ is the energy functional defined in \eqref{E-defn},
\begin{equation*}
\mathcal A_{c,\rad}=\{u \in E^{\alpha,p}_\rad(\R^N): \int_{\R^N}\abs{u}^q=c\},
\end{equation*}
and $c>0$. Since the rescaling relation \eqref{scaledinfimum} also applies to $M_{c,\rad}$,
in what follows we can restrict ourself to the case $c=1$.

\begin{proposition}\label{MinRadialExistence}
Let $N\geq2,$ \(\alpha \in (0, N)\), \(p \ge 1\).
Assume that either $\alpha\le 1$ and assumption \eqref{Q_0} holds,
or $\alpha> 1$ and assumption \eqref{Qrad0} holds.
Then all the minimising sequences for $M_{1,\rad}$ in \eqref{varprob1} are relatively compact in $E^{\alpha,p}_\rad(\R^N)$.
In particular, the best constant $M_{1,\rad}$ is achieved.
\end{proposition}

\begin{proof}
Let \((u_n)_{n \in \N}\subset\mathcal A_{1,\rad}\) be a minimizing sequence for $M_{1,\rad}$.
By the minimization property, we have $\sup_{n\in\N} \mathcal{E}_*(u_n)< \infty$.
In view of proposition~\ref{propositionElementaryLocalWeakCompactness} and proposition~\ref{propositionSemiContinuityLocalConvergence}, up to a subsequence, $(u_n)_{n\in\N}$ converges to a function $u\in E^{\alpha,p}(\R^N)$ strongly in $L^1_{\textrm{loc}}(\R^N)$.
Moreover, $\mathcal{E}_*(u)\le M_{1,\rad}$.
By proposition~\ref{propositionRadialCompactness}, the  embedding $E^{\alpha,p}_\rad(\R^N)\subset L^q(\R^N)$ is compact.
Hence, $(u_n)_{n\in\N}$ converges to $u$ strongly in \(L^q (\R^N)\), $u\in \mathcal A_{1,\rad}$ and $\mathcal{E}_*(u)=M_{1,\rad}$.
Therefore, the sequence \((D u_n)_{n \in \N}\) converges strongly to \(Du\) in \(L^2 (\R^N)\),
while the sequence \((I_{\alpha/2} \ast \abs{u_n}^p)_{n \in \N}\) converges strongly to \(I_{\alpha/2} \ast \abs{u}^p\) in \(L^2 (\R^N)\) by the Brezis--Lieb property of proposition~\ref{propositionBrezisLiebLp}.
\end{proof}

\begin{proof}[Proof of theorem~\ref{t-radial-intro}]
The compactness part of the statement of the theorem is contained in proposition~\ref{propositionRadialCompactness}.
The existence of a minimiser $w\in E^{\alpha,p}_\rad(\R^N)$ for $M_{1,\rad}$ follows from proposition \ref{MinRadialExistence}.
Since $|w|\in E^{\alpha,p}_\rad(\R^N)$ is also a minimiser for $M_{1,\rad}$, we can assume that $w$ is nonnegative.
Since $p>1$, the energy $E$ is of class $C^1(E^{\alpha,p}_\rad(\R^N);\mathbb R)$ and
by the Palais symmetric criticality principle \cite{WillemMinimax}*{theorem 1.28},
$w$ is a weak solution of
\begin{equation}\label{sps-mu-rad}
 - \Delta w + (I_\alpha \ast \abs{w}^p)\abs{w}^{p - 2} w= q\mu\abs{w}^{q-2}w\quad\text{in \(\R^N\)},
\end{equation}
with a Lagrange multiplier $\mu>0$.

Local regularity of $w$ follows via the same arguments as in the proof of proposition~\ref{proposition:Regularity},
but using the radial Coulomb--Sobolev embedding of theorem~\ref{theoremRadialSobolev-intro} in order
to estimate the norm in \eqref{s-regularity}. In addition to local regularity,
uniform decay estimate of theorem~\ref{theoremRadialEstimates} implies that $w(\abs{x})\to 0$ as $\abs{x}\to\infty$.
Finally, positivity of $w$ for $p\ge 2$ follows by the arguments of proposition~\ref{proposition:Positivity}.

If $q\neq 2\frac{2p+\alpha}{2+\alpha}$,
then as in \eqref{EulerMinim2+}, minimizer $w$ can be rescaled to a solution of the original equation \eqref{sps} with $\mu=1$.
On the other hand, if $q=2\frac{2p+\alpha}{2+\alpha}$ then \eqref{sps-mu-rad} is invariant
with respect to the scaling  \eqref{EulerMinim2+}. In this case,
using Poho\v zaev identity of proposition~\ref{theoremPohozhaev},
similarly to Poho\v zaev relations \eqref{pohozaev-realtions} and since $\mathcal{E}_*(w)=M_{1,\rad}$,
we compute that
$$\mu=2M_{1,\rad}\frac{(N-2)p-(N+\alpha)}{2N(p-1)-q(2+\alpha)}.$$
In particular, if $q=2\frac{2p+\alpha}{2+\alpha}$ then $\mu=M_{1,\rad}$, and this concludes the proof.
\end{proof}

\begin{remark}
The significance of the threshold  $q=2\frac{2p+\alpha}{2+\alpha}$ in the statement
of theorem~\ref{t-radial-intro} is clarified by the analysis of the geometry of the unconstrained energy functional
$$\mathcal{J}_*(u)=\frac{1}{2}\int_{\R^N}|Du|^2+\frac{1}{2p}\int_{\R^N}\big|I_{\alpha/2}*\abs{u}^p\big|^2-\frac{1}{q}\int_{\R^N}\abs{u}^q.$$
Similarly to the arguments in \cite{Ruiz-ARMA}*{theorem 1.3}, one can show that
if the  assumption \eqref{Q_0} holds, then $\mathcal J_*$ has a mountain pass geometry on $E^{\alpha,p}_\rad(\R^N)$.
On the other hand, if \(\alpha > 1\) and the complementary to \eqref{Q_0} assumption
\begin{equation}\tag{$\mathcal Q''$}\label{Qrad3}
\begin{split}
\text{either}\quad & \frac{1}{p} > \frac{(N - 2)_+}{N + \alpha} \quad\text{and}\quad
   \frac{1}{2} - \frac{p - 1}{\alpha + 2 p} < \frac{1}{q} < \frac{3 N + \alpha - 4}{2(2 p (N - 1) + N - \alpha)}\,,\bigskip\\
\text{or}\quad & \frac{1}{p} < \frac{(N - 2)_+}{N + \alpha} \quad\text{and}\quad
   \frac{1}{2} - \frac{p - 1}{\alpha + 2 p} > \frac{1}{q}> \frac{3 N + \alpha - 4}{2(2 p (N - 1) + N - \alpha)}\,
\end{split}
\end{equation}
holds, then $\mathcal J_*$ is coercive on $E^{\alpha,p}_\rad(\R^N)$ and $\inf_{u\in E^{\alpha,p}_\rad(\R^N)} \mathcal{J}_*(u)<0$.

Indeed, a scaling argument similar to \cite{Ruiz-ARMA}*{p.\thinspace{}361}
together with the continuous radial embedding of $E^{\alpha,p}_\rad(\R^N)$ into $L^q(\R^N)$
shows that there exists a constant $C>0$ such that
\begin{equation}\label{MP1}
\mathcal{J}_*(u)\geq \mathcal{E}_*(u)-C \mathcal{E}_*(u)^{\frac{2N(p-1)-(2+\alpha)q}{2(p(N-2)-N-\alpha)}},
\end{equation}
where $\mathcal{E}_*$ is defined in \eqref{E-defn}.
Then \eqref{Q_0} implies that $u=0$ is a strict local minimum of $\mathcal J_*$,
while if the assumption \eqref{Qrad3} holds
then $\mathcal J_*$ is coercive on $E^{\alpha,p}_\rad(\R^N)$.

If $p>1$, then for $u\in E^{\alpha,p}_\rad(\R^N)\setminus\{0\}$ and $\lambda>0$ we define
$$u_\lambda(x)=\lambda^{\frac{2+\alpha}{2(p-1)}}u(\lambda x).$$
Then
\begin{equation*}
\mathcal{J}_*(u_\lambda)=\lambda^{\frac{2+\alpha+(2-N)(p-1)}{p-1}}\Big(\frac{1}{2}\int_{\R^N}\abs{D u}^2+\frac{1}{2p}\int_{\R^N}\big|I_{\alpha/2}*\abs{u}^p\big|^2\Big)-\lambda^{\frac{(2+\alpha)q-2N(p-1)}{2(p-1)}}\frac{1}{q}\int_{\R^N}\abs{u}^q.
\end{equation*}
If \eqref{Q_0} holds then $\mathcal{J}_*(u_\lambda)$ is unbounded below,
while \eqref{Qrad3} implies that $\mathcal{J}_*(u_\lambda)$ attains negative values for small $\lambda>0$.

Finally, if \eqref{Qrad3} holds and $p=1$, then for $u\in E^{\alpha,p}_\rad(\R^N)\setminus\{0\}$ and $\lambda>0$,
\begin{equation*}
\mathcal{J}_*(\lambda u)=\lambda^{2}\Big(\frac{1}{2}\int_{\R^N}\abs{D u}^2+\frac{1}{2p}\int_{\R^N}\big|I_{\alpha/2}*\abs{u}\big|^2\Big)-\lambda^{q}\frac{1}{q}\int_{\R^N}\abs{u}^q.
\end{equation*}
Since $q<2$ in view of \eqref{Qrad3} we conclude that $\mathcal{J}_*(\lambda u)$ attains negative values for small $\lambda>0$.

This geometric characterisation of $\mathcal J_*$ combined with the compactness of the embedding
of $E^{\alpha,p}_\rad(\R^N)$ into $L^q(\R^N)$ could be used to provide an alternative proof
of the existence of the radial solution of \eqref{sps} via direct minimization or via mountain--pass lemma,
similarly to \citelist{\cite{Ruiz-ARMA}\cite{Ruiz-JFA}}, where the case $N=3$, $p=2$, $\alpha=2$ was considered.
Such approach however excludes the Coulomb--Sobolev critical case $q=2\frac{2p+\alpha}{2+\alpha}.$
\end{remark}

\begin{remark}\label{universalmubound}
Let $p\neq\frac{N+\alpha}{N-2}$, $q= 2\frac{2p+\alpha}{2+\alpha}$ and $u\in E^{\alpha,p}(\R^N)$ be a nontrivial radial solution of
\begin{equation*}
 - \Delta u + (I_\alpha \ast \abs{u}^p)\abs{u}^{p - 2} u= q\mu\abs{u}^{q-2}u\quad\text{in \(\R^N\)}.
\end{equation*}
Then following universal bound holds for $\mu$:
\begin{equation}\label{boundmu}
\mu\ge M_{1,\rad}.
\end{equation}
In order to see this notice that in view of the relation \eqref{scaledinfimum}, taking into account that for $q= 2\frac{2p+\alpha}{2+\alpha}$ we have $2\sigma/q=1$, we can rewrite the Coulomb-Sobolev inequality \eqref{ineqCriticalCoulombSobolev}
as
\begin{equation*}
M_{1,\rad}C_*^{-1}\int_{\R^N} \abs{u}^{2 \frac{2 p + \alpha}{2 + \alpha}}
\le  \Bigl(\int_{\R^N} \abs{Du}^2 \Bigr)^\frac{\alpha}{2 + \alpha}
\Bigl( \int_{\R^N} \abs{I_{\alpha/2} \ast \abs{u}^p}^2\Bigr)^\frac{2}{2 + \alpha}.
\end{equation*}
By combining Poho\v zaev's and Nehari's identities the above inequality becomes
\begin{equation}\label{universalbound}
M_{1,\rad}C_*^{-1}
\le  \mu 2\frac{2p+\alpha}{2+\alpha}\Bigl(\theta\Bigr)^\frac{\alpha}{2 + \alpha}
\Bigl( 1-\theta \Bigr)^\frac{2}{2 + \alpha}.
\end{equation}
Here
$\theta= \frac{\alpha}{2p+\alpha}$
and
$C_*=\left(\Big(\frac{2}{\alpha}\Big)^{\frac{\alpha}{\alpha+2}}+\Big(\frac{2}{\alpha}\Big)^{-\frac{2}{\alpha+2}}\right)
\Big(\frac{1}{2p^{\frac{2}{\alpha+2}}}\Big).$
It is easy to check that
$$C_*2\frac{2p+\alpha}{2+\alpha}\theta^\frac{\alpha}{2 + \alpha}( 1-\theta )^\frac{2}{2 + \alpha}=1.$$
Hence \eqref{boundmu} follows.
\end{remark}

\section*{Acknowledgements}
C.M. would like to thank Antonio Ambrosetti for drawing his attention to several questions related to nonlinear Schr\"odinger-Poisson systems.
J.V.S was funded by the Fonds de la Recherche Scientifique - FNRS (grant T.1110.14)

\end{document}